\documentclass[11pt]{amsart}
\usepackage{amsbsy}
\usepackage{graphicx,epsfig,subfigure,psfrag,color}
\usepackage{caption}
\begin{document}

\newtheorem{theorem}{Theorem}
\newtheorem{proposition}{Proposition}
\newtheorem{lemma}{Lemma}
\newtheorem{corollary}{Corollary}
\newtheorem{definition}{Definition}
\newtheorem{remark}{Remark}
\newcommand{\tex}{\textstyle}
\numberwithin{equation}{section} \numberwithin{theorem}{section}
\numberwithin{proposition}{section} \numberwithin{lemma}{section}
\numberwithin{corollary}{section}
\numberwithin{definition}{section} \numberwithin{remark}{section}
\newcommand{\ren}{\mathbb{R}^N}
\newcommand{\re}{\mathbb{R}}
\newcommand{\n}{\nabla}
\newcommand{\p}{\partial}
\newcommand{\iy}{\infty}
\newcommand{\pa}{\partial}
\newcommand{\fp}{\noindent}
\newcommand{\ms}{\medskip\vskip-.1cm}
\newcommand{\mpb}{\medskip}
\newcommand{\AAA}{{\bf A}}
\newcommand{\BB}{{\bf B}}
\newcommand{\CC}{{\bf C}}
\newcommand{\DD}{{\bf D}}
\newcommand{\EE}{{\bf E}}
\newcommand{\FF}{{\bf F}}
\newcommand{\GG}{{\bf G}}
\newcommand{\oo}{{\mathbf \omega}}
\newcommand{\Am}{{\bf A}_{2m}}
\newcommand{\CCC}{{\mathbf  C}}
\newcommand{\II}{{\mathrm{Im}}\,}
\newcommand{\RR}{{\mathrm{Re}}\,}
\newcommand{\eee}{{\mathrm  e}}
\newcommand{\LL}{L^2_\rho(\ren)}
\newcommand{\LLL}{L^2_{\rho^*}(\ren)}
\renewcommand{\a}{\alpha}
\renewcommand{\b}{\beta}
\newcommand{\g}{\gamma}
\newcommand{\G}{\Gamma}
\renewcommand{\d}{\delta}
\newcommand{\D}{\Delta}
\newcommand{\e}{\varepsilon}
\newcommand{\var}{\varphi}
\newcommand{\lll}{\l}
\renewcommand{\l}{\lambda}
\renewcommand{\o}{\omega}
\renewcommand{\O}{\Omega}
\newcommand{\s}{\sigma}
\renewcommand{\t}{\tau}
\renewcommand{\th}{\theta}
\newcommand{\z}{\zeta}
\newcommand{\wx}{\widetilde x}
\newcommand{\wt}{\widetilde t}
\newcommand{\noi}{\noindent}
\newcommand{\uu}{{\bf u}}
\newcommand{\xx}{{\bf x}}
\newcommand{\yy}{{\bf y}}
\newcommand{\zz}{{\bf z}}
\newcommand{\aaa}{{\bf a}}
\newcommand{\cc}{{\bf c}}
\newcommand{\jj}{{\bf j}}
\newcommand{\ggg}{{\bf g}}
\newcommand{\UU}{{\bf U}}
\newcommand{\YY}{{\bf Y}}
\newcommand{\HH}{{\bf H}}
\newcommand{\GGG}{{\bf G}}
\newcommand{\VV}{{\bf V}}
\newcommand{\ww}{{\bf w}}
\newcommand{\vv}{{\bf v}}
\newcommand{\hh}{{\bf h}}
\newcommand{\di}{{\rm div}\,}
\newcommand{\ii}{{\rm i}\,}
\newcommand{\inA}{\quad \mbox{in} \quad \ren \times \re_+}
\newcommand{\inB}{\quad \mbox{in} \quad}
\newcommand{\inC}{\quad \mbox{in} \quad \re \times \re_+}
\newcommand{\inD}{\quad \mbox{in} \quad \re}
\newcommand{\forA}{\quad \mbox{for} \quad}
\newcommand{\whereA}{,\quad \mbox{where} \quad}
\newcommand{\asA}{\quad \mbox{as} \quad}
\newcommand{\andA}{\quad \mbox{and} \quad}
\newcommand{\withA}{,\quad \mbox{with} \quad}
\newcommand{\orA}{,\quad \mbox{or} \quad}
\newcommand{\atA}{\quad \mbox{at} \quad}
\newcommand{\onA}{\quad \mbox{on} \quad}
\newcommand{\ef}{\eqref}
\newcommand{\mc}{\mathcal}
\newcommand{\mf}{\mathfrak}

\newcommand{\ssk}{\smallskip}
\newcommand{\LongA}{\quad \Longrightarrow \quad}
\def\com#1{\fbox{\parbox{6in}{\texttt{#1}}}}
\def\N{{\mathbb N}}
\def\A{{\cal A}}
\newcommand{\de}{\,d}
\newcommand{\eps}{\varepsilon}
\newcommand{\be}{\begin{equation}}
\newcommand{\ee}{\end{equation}}
\newcommand{\spt}{{\mbox spt}}
\newcommand{\ind}{{\mbox ind}}
\newcommand{\supp}{{\mbox supp}}
\newcommand{\dip}{\displaystyle}
\newcommand{\prt}{\partial}
\renewcommand{\theequation}{\thesection.\arabic{equation}}
\renewcommand{\baselinestretch}{1.1}
\newcommand{\Dm}{(-\D)^m}

\title{\bf ON COUNTABLE SUBSETS OF SOLUTIONS OF NONLINEAR HIGHER-ORDER ODES
AND ELLIPTIC PDES WITH INDEFINITE OPERATORS}

\author{Pablo~\'Alvarez-Caudevilla$^*$, Jonathan D.~Evans and Victor A.~Galaktionov}

\address{Universidad Carlos III de Madrid,
Av. Universidad 30, 28911-Legan\'es, Spain -- Work phone number:
+34-916249099} \email{pacaudev@math.uc3m.es}

\address{Department of Mathematical Sciences, University of Bath,
 Bath BA2 7AY, UK -- Work phone number: +44 (0)1225 386994 }
\email{masjde@bath.ac.uk}

\address{Frome Mathematical Institute, Frome BA11 1PA, UK}



\thanks{This work has not been directly funded by any institution}

\thanks{$*$ Corresponding author}


\date{\today}


\begin{abstract}

Countable subsets of solutions of higher-order nonlinear ODEs and elliptic PDEs with indefinite  non-coercive
operators from the
reaction-diffusion, thin film  and dynamical system (DS) theories are obtained via a gluing/matching argument. 
In particular we study some classic and  quasilinear degenerate ODEs in $\re$, with boundary conditions at infinity $F(\iy)=0$,  with non-odd nonlinearities such as
 $$
 \begin{matrix}
 F^{(4)} =-F+F^2,\, \,\, F^{(4)}=-F+F^2{\rm e}^{F-1}, 
 \,\,\, (|F''|F'')''=-F + F^2,
 \\
 F^{(4)} =-|F|F+F^2, \,\,\,F^{(4)} =-F^3+F^4,
 \,\,\, F^{(6)}=F-F^2, \,\,\,\mbox{\em etc.},
 \end{matrix}
 $$
{\em etc.}, as well as 
  equations   with odd and non-smooth nonlinearities like
  $$
  F^{(4)}=-F+F^3, \,\, 
  \tex{
   F^{(4)}=-F-(|F|F-F)'', \,\, F^{(4)}=- \frac  F{\sqrt{|F|}}-(F^3-F)'',
   }
 \,\,\,\mbox{\em etc.}
  $$
Some of these ODEs are Hamiltonian and were studied in detail in the DS theory. On the basis of nonlinear operators, elliptic PDEs and variational theory, related
 polyharmonic   elliptic equations
  in $\ren$, $F(\iy)=0$,  such as
  $$
   \D^2 F=-F+F^2, \quad \D^2 F=-F -\D (F^2-F), \quad  \D^3 F=F-F^2, \quad \mbox{\em etc.};
 $$
 are also shown to admit countable families of solutions.
For such equations with non-odd functionals associated
 Lusternik--Schnirel'man (L--S) genus/category variational theory guaranteeing existence of a sequence of critical points
 does not apply. 
These  ODEs 
and  elliptic PDEs
(e.g., in the radial setting)
  are shown to admit at least two basic countable families ${\mathcal F}_{1,2}$
of positively dominant solutions  connected with two periodic orbits $\Gamma_{\rm max/min}$.
Patterns obtained by gluing together via exponentially decaying tails
of arbitrary finite samples from $\Gamma$'s  form a countable subset of homoclinics 
in $\re^4$ 
of an arbitrary complexity. These
create an unstable ``two-wings" attractor $W^{2,\infty}$also containing infinitely many periodic orbits and surrounded by an uncountable subset of semi-orbits blowing up at finite $x$'s and global chaotic ones.

\end{abstract}

\maketitle

\section{Introduction: higher-order  ODE and elliptic PDE models}
 \label{S1}

\subsection{Main nonlinear ODEs, elliptic PDEs, and a DS motivation}

In this paper we continue the study began in \cite{AEGnegI}, where we applied the  classic
 variational Lusternik-Schnirel'man (L--S) theory for nonlinear elliptic PDEs with {\bf non-coercive operators with odd nonlinearities}, which  guarantees existence of at least a {countable family} of critical points  ordered according to the genus (category) of the functional subsets involved.
 Our goal now is to show that {\sc infinite countable subsets of solutions} may exist for some non-linear higher-order ODEs and elliptic PDEs 
 with {\bf non-coercive indefinite} (non-odd-non-even nonlinearities) operators. In the sequel we will call them indefinite operators.

\smallskip

{\sc The classical canonical quadratic 4th-order ODE: a basis for quasilinear extensions.} Thus, to achieve our main target we begin with the simplest equation of this type:
 \begin{equation}
  \label{N1.1}
F^{(4)}= G(F) \equiv  -F+F^2   \inB \re, \quad F(\infty)=0,
  \end{equation}
to  show why and how (from a pure PDE point of view) it 
admits an
extremely wide  variety of solutions, which, in particular, cannot be
described by the classic variational L--S category/genus or fibering  tools
applied in  \cite{AEGnegI}, devoted to 
 non-coercive operators with odd nonlinearities. Since \ef{N1.1} in some aspects has been almost completely covered by the existing Hamiltonian DS  there (we will present some key references later on) we should mention that our eventual goal is to extend this experience doing with \ef{N1.1} to more complicated ODE models like
  \be  
  \label{Quas1}
  (|F''|F'')''=-F^3 - ((F')^3)' \inB \re, \quad F(\iy)=0,
  \ee
  or similar with indefinite operators. Such so-called  $(p,q)$-Laplacian type higher-order models  look quite exotic for the DS theory  and  applications  but are important and even crucial for a general nonlinear elliptic  operator theory.

Subsequently, we continue and extend our elliptic/ODE study to several other operators with non-odd nonlinearities.
 As a necessary extension (that we show below), we  deal with their  elliptic analogues
  of ODEs like  \ef{N1.1} and \ef{Quas1}:
  $$
   \D^2F=-F+F^3 \quad\mbox{and} \quad \D(|\D F|\D F)=-F^3 - \n \cdot (|\n F|^3 \n F) \inB \ren.
   $$
 Concerning some  applications, these  can be viewed as  semilinear and quasilinear stationary 4th-order equations from
reaction-diffusion theory  and
   were carefully studied for decades through the nonlinear operators theory (using techniques of homotopy of vector fields, rotations indices, Leray-Schauder and other degrees, Morse indices,  {\em etc.}) as well as in the Hamiltonian dynamical system (DS) theory; see references below. 

To this aim, in finding infinite countable sets of solutions for the models in hand, we shall present a matching/gluing patterns construction starting with
the  canonical (simplest) analytic fourth-order ODE with a non-odd nonlinearity \eqref{N1.1} and applying it to several other models, ODEs and elliptic PDEs. 
Thus 
we begin performing a careful and detailed asymptotic analysis for \eqref{N1.1} which will be important in constructing new patterns via our matching/gluing argument. 
Hence, Section \ref{SexpanA}  is devoted to 
a linearized analysis of exponentially decaying tails as $x \to \pm \iy$ of solution of the canonical ODE \eqref{N1.1} which might be extended to the other models shown in this paper. 

Once such analysis is performed we are able to construct the first basic countable family ${\mathcal F}_1=\{F_k\}_{k\geq 0}$ following our matching/gluing approach. 
Those $F_k$ are obtained after gluing $k$ elementary first variational patterns 
$\sim F_0$'s, that we will ascertain in the sequel of this paper, distributing them over the $x$-axis.
 See Section\;\ref{S2.LS} for variational applications, where  
through some numerical evidence, we identify $F_0$ among others, as providing the absolute minimum 
  of the associated functional for \eqref{N1.1}.

Moreover, to do so we must understand how to match two patterns $ \sim F_0$ by gluing their almost linear exponential oscillatory tails, after shifting them in space. In Section \ref{S2N}, we are able to perform this matching/gluing argument
catching exponential small tails through some geometric justification and a numerical evidence. 
Furthermore, we look for 
  patterns obtained by gluing at $x=0$ several 
$F_0$ patterns, under certain symmetry conditions. 
As a first simple  example, we take a superposition of two patterns of the form
$
F_\sigma(x)\approx F_0(x+a_n)+F_0(x-a_n)$ with $\s=\{+2,l_n,+2\},
$
  where $a_n  \sim \sqrt 2 \pi n\gg 1$ is a special sufficiently large shifting  parameter, to be determined from an elementary
  algebraic expression. Fixing the precise notation for such a gluing (shown in Section\;\ref{S2N}) 
  we denote by $F_{+l}$ a profile obtained via matching/gluing  with precisely $l$ intersections with the steady state of the ODE \eqref{N1.1}, i.e. $F_{*}\equiv 1$. 
  Now, assuming the profiles obtained via a matching/gluing argument we denote by $F_{+l,j,+l}$ as the gluing of two profiles $F_{+l}$ with $j$ number of zeros in between the two $F_{+l}$ profiles.
  In some situations we will use such number of zeros $j$ to measure the number of minimum points around the matching area. 
  Consequently, in such a process, and generalising the previous construction for more complicated matching of different profiles of the form $F_{+l}$, we identify several families of solutions 
  via the gluing argument described above.    
  
  After such a process we arrive at the situation when a simple higher-order problem such as \eqref{N1.1} admits infinite countable set of possible solutions, vanishing exponentially fast as $x\to \infty$. 
  We claim, and discuss throughout this paper that such a set of solutions of actually {\bf chaotic} in the sense that problem \eqref{N1.1} admits a family of solutions $F_\sigma(x)$ with an 
  arbitrary non-periodic index $\sigma$. 
  
  As we have already noticed, assuming equation \eqref{N1.1} with {odd non-coercive nonlinearities, replacing} 
  $F^2 \mapsto F^3$,
 at least a single
 infinite sequence of  L--S variational patterns
exists \cite{AEGnegI}.
Moreover, similar matching/gluing
of patterns 
can produce other
countable pattern families.

\smallskip    
    
{\sc ODEs with more general nonlinear operators.}    
 Let us describe a wider area of applications for the matching/gluing approach shown in this paper.    
  It is crucial for us that the canonical model \ef{N1.1}  has 
   a typical necessary {\sc indefinite} (non-odd or non-even) {\sc non-coercive} (such that 
 arbitrarily large solutions in  $L^2(\re)$ can exist)
{\sc quadratic} (i.e., a polynomial  of the smallest degree) operator. Using first \ef{N1.1}, we  extend the results (see Section\;\ref{SectExp4}) 
to $2m$th-order  equations with various nonlinearities:
$$
\begin{matrix}
D_x^{4} \mapsto (-1)^m D_x^{2m} \mapsto (-1)^m \D^m,\,\, m=2,3, \,\,\, \mbox{with}
\,\,\,G(F)=-F+ F^2{\rm e}^{F-1},
\\
G(F)=-F+ F^3, \,\,\, G(F)=-F+F^{21}\,\,\, (\mbox{a ``Black Jack equation"}), 
\end{matrix}
$$
and proper others, with an additional application to the 
study of six-order equations in Section \ref{SSix}.
For odd nonlinearities $G(F)$  the results apply to ``positive dominant" solutions, since there exist principally others essentially sign-changing L--S ones. 
Countable subsets of patterns
can be also observed for many other  ODEs as  1D stationary models of Cahn--Hilliard (C--H) type such as
 \be 
 \label{Pol34}
 F^{(4)}  =G(F) \equiv U(F) + (V(F))'+ (W(F)'' \,\,\,(\mbox{and $D_x^4 \mapsto -D_x^6 \mapsto D_x^8$...})
  \ee
  with polynomial (or similar) functions $U, \, V\,, W$
  regardless of odd/even, variational or Hamiltonian properties of the operators.
 The Fr\'echet derivative at $F=0$, 
  $$
  L_0=D_x^4- G_F'(0)I\equiv D_x^4 -U_F'(0)I- (V_F'(0)I)'-
 (W_F'(0)I)''
  $$
   is then assumed to have
  the defect indices $(2,2)$ in $\re_\pm$ (the origin $x=0$ is regular), that characterize the number of  $L^2(\re_\pm)$-solutions
  of
    $L_0 \psi= \pm {\rm i} \, \psi$. The defect index of $L_0$ in $\re$ is $(0,0)$.
         For the
     DSs in $\re^4$  this means that the origin O is a proper ``saddle-node"
      \cite{Champ94}.
  For polynomial nonlinearities in (\ref{Pol34}) of degree  3 and more, 
   there appear more equilibria, more key periodic orbits, complicated homoclinics, chaotic orbits, and  ``multi-wing" attractors
   $W^{k,\infty}$.

 From the point of view of the nonlinear elliptic theory,
 the analyticity or even any $C^l$-smoothness in (\ref{Pol34}) is not required. As a simple example, in Section\;\ref{SMod}, we consider
 \be
 \label{Mod33}
  F^{(4)}= -F - (|F|F-F)'' \,\,\,\mbox{in} \,\,\, \re, \quad F(\iy)=0,
   \ee 
where $|F|F$ is $C^{1,0}$ (the derivative is Lipschitz). Another example with $G \in  C^{0,1}$ is the following piece-wise linear (P-L) ``approximation" of \ef{N1.1} with a   Lipschitz nonlinearity \cite{AEGII}: 
\begin{equation}
  \label{PL111}
 \tex{
  F^{(4)}=G(F) \equiv |F-\frac 12| - \frac 12=
 }
  \left\{
  \begin{matrix}
  F-1, \,\,F \ge \frac 12,
  \\
  -F, \,\, F \le  \frac 12. \quad \,\,
  \end{matrix}
 \right.
  \end{equation}
 Even a Lipschitz condition on $G$ is not necessary
as our third non-standard example shows:
 \be 
 \label{Hold1}
\tex{
F^{(4)}= G(F) \equiv  - \frac {F} {\sqrt{|F|}}  -(F^3-F)'',\,\, x \in \re, \quad F(\iy)=0,
}
 \ee  
 where 
the first H\"older continuous  term 
 $- \frac {F}{|\sqrt F|}$  replaces the linear one $-F$ in \ef{Mod33} and can do the same in the most  ODEs above.
 Here we need a comment:
(\ref{Hold1})  admits compactly supported weak (being linear functionals) solutions $F \in C_0(\re)$ with bounded {\rm supp}$F=(x_0,x_1) \subset \re$. Oscillatory properties 
of solutions near the end points $x_{1,2}$ of the support are known \cite{EGK1, EGK2}, \cite[\S~1.4]{GMPBook}, and this allows us to generate an infinite number of patterns which are homoclinics on $(x_0,x_1)$ for each pattern.

 Curiously, this gives  quite a surprising conclusion: (\ref{Hold1}) {\sc  always} admits an {\sc uncountable} and moreover infinite-dimensional subset  of patterns in $\re$.
Here the theory of analytic or sufficiently  smooth 
equations, as  admitting only isolated ($x$-translations forbidden) solutions, stops. 
  Indeed having two  patterns $F_{1,2}$ (or as many as we want) with non-overlapping supports {\rm supp}$F_{1,2}$, the
   second 
   can be arbitrarily  moved in $x$ (the first one  stays fixed since the general $x$-transfer is prohibited) without overlapping of the supports to create new solutions.
  Of course, this is a striking property of the H\"older continuity only (the nonlinearity is not Lipschitz). We present such exotic examples in order to show that the class of  nonlinear ODEs and elliptic PDEs admitting infinite subsets of patterns is extremely wide and this does not depend on assumptions of  being Hamiltonian, gradient, variational, etc. Note also that such examples are not that exotic for reaction-diffusion-absorption, C--H and thin film applications, see further references and examples in \cite[Ch.~1]{GMPBook}.

We would like now to point out that there different ODEs for which we are able to perform a comparison of the first pattern $F_0(x)$
in order to show that these ODEs actually are not only similar but are very close metrically. 
As an example, we show that such pattern phenomena are quite common
and briefly  describe
first positive dominant patterns for three different ODEs. In particular, we might consider the ODEs
 \be
 \label{3ODE}
 \tex{
  F^{(4)}=G_1(F) \equiv -F+F^2, \,\,\, F^{(4)}=G_2(F) \equiv -F+F^3, \,\,\,  F^{(4)}=G_0(f) \equiv |F-\frac 12| - \frac 12.
 }
 \end{equation}  
 The last one is the P-L equation 
\ef{PL111} which serves here as  
 a suitable P-L ``approximation" \cite{AEGII} of the first two ones
   for positive dominant patterns, i.e., those admitting only  
   small negative values 
   due to oscillatory exponential tails.
Graphs of  three nonlinearities in \ef{3ODE} 
are given 
in Fig. \ref{FPL0}, while close to each other  first main 
patterns $F_0$ are presented in Fig. \ref{3ODEFig}.
It turns out that other more complicated positive dominated patterns are also close to each other.
Therefore, these 
nonlinear  equations are expected to belong to the same ``homotopy class", i.e., their pairs of solutions can
be transformed to each other by
 a non-degenerate analytic $\e$-deformation.
  It is crucial that the P-L problem  \ef{PL111} admits a technically doable algebraic classification  of all the patterns \cite{AEGII}, which then can be attributed  to other  equations from the same class. 
  In particular, this can reveal 
the actual true sense  of an irrational ``$\sigma$-index" which can be attributed to any pattern $F_\sigma$ from their countable subset and which, in fact, describes its geometric shape and nature.
In the most general and complicated cases, it cannot be replaced by any
finite ordered collections of integers.

\begin{figure}[htbp]
 \hfill   \hfill  \begin{minipage}[t]{0.45\textwidth} 
    \centering 
    \includegraphics[width=6cm,keepaspectratio]{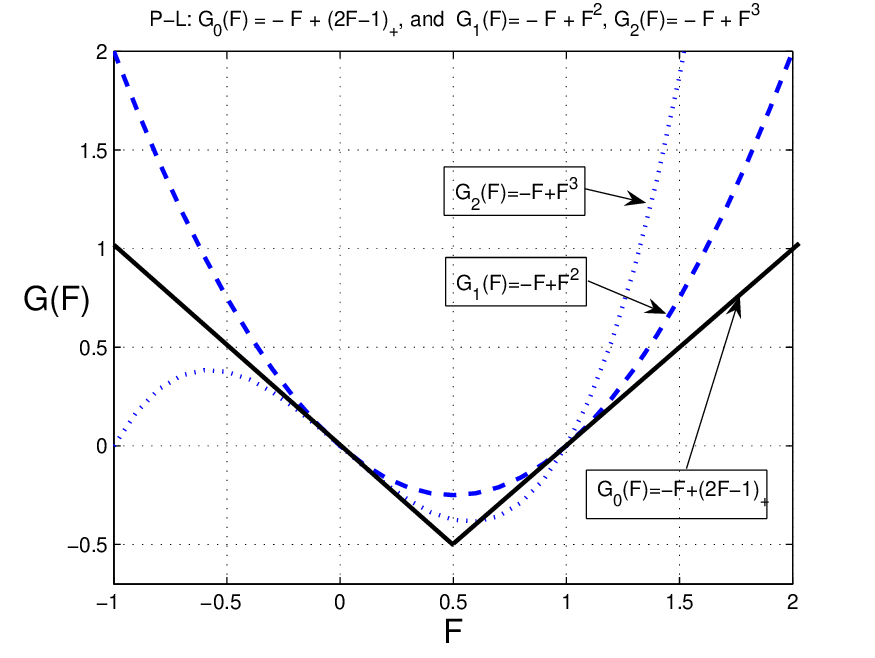}
    \caption{Comparison of nonlinearities in \ef{3ODE}.}
\label{FPL0}
\end{minipage}
  \hfill   \hfill 
\begin{minipage}[t]{0.5\textwidth}
    \centering 
    \includegraphics[width=6cm,keepaspectratio]{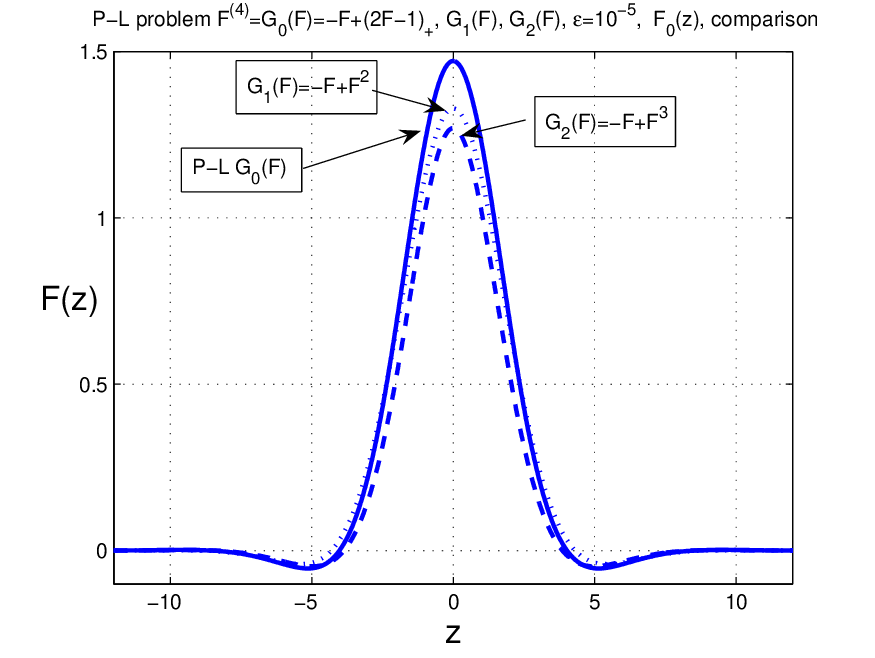}
    \caption{First pattern  $F_0(z)$ of the three
ODEs in \ef{3ODE}.}
\label{3ODEFig}
\end{minipage}
\end{figure}


Other more nonlinear ODEs having  {\em algebraic}
 {\em oscillatory tails} of solutions at infinity are
 \be 
 \label{Fmod1}
 F^{(4)}=-|F|F +F^4, \quad F^{(4)}=- F^3 + F^4,\quad F^{(4)}=-F^3 + F^4 - (F^5)'', \,\,\, \mbox{\em etc.}.
 \ee
 It is shown that a nonlinear version of the ``nonlinear" deficiency indices $(2,2)$ of the operators 
 is valid (the origin O is a ``nonlinear saddle-node"), the algebraic tails are 2D-{manifold}, and this allows us to study 
 countable patterns subsets by a matching/gluing approach.

Further extensions  lead to the so-called $p$-Laplacian operators (a term  from the filtration theory of non-Newtonian liquids) and thin film equations (TFEs),
 \be 
 \label{TFE2}
 (|F''| F'')''=-F+F^2, \quad (|F'| F')'''=-F+F^4, \quad (|F| F''')'=-F-(F^3-F)'', \,\,\,\mbox{\em etc.}, 
   \ee 
which 
present  a good challenge and keep a similar variety of patterns. Such equations, which also can be formally associated with some C--H models with a non-constant concentration  mobility 
 admit compactly supported weak solutions with algebraic
{\sc oscillatory tails} at end support points and require a special consideration of their homoclinics, periodics, attractors, {\em etc.}
All the operators in \ef{TFE2} are not variational and ODEs are not Hamiltonian.
By using irregular and degenerate operators similar to those  in \ef{Hold1} and \ef{TFE2} we are able to compose  a shocking  ``triple-troubled" quasilinear degenerate ODE which is  however variational in $W^{3,2} \cap W^{4,1}$
(embeddings are neglected):
 $$
 \begin{matrix}
 (|F''|F'')''= - \frac F{\sqrt{|F|}} - [3(F')^2-1]F'', \quad \mbox{with the functional}
 \\
 \Phi(F)= \frac 13 \int|F''|^3 + \frac 23 \int|F|^{\frac 32} + \frac 12 \int (F')^2- \frac 14 \int (F')^4.
 \end{matrix}
 $$
 The functional is even (the differential  operator is odd) and admits  the L--S countable sequence of patterns. This is  also true in the elliptic setting (see \cite{AEGnegI} for similar L--S families of solutions)
  $$
  \tex{
  \D (|\D F|\D F)=- \frac {F}{\sqrt |F|} - \nabla \cdot (|\nabla F|^2 \nabla F - \nabla  F) \quad \mbox{in}
  \quad \ren, \quad F(\infty)=0,
  }
  $$
  and  a further study of patterns  is possible but is
 connected with technical difficulties.

\smallskip

{\sc Nonlinear elliptic PDEs.}
Thus, continuing such nonlinear operators and elliptic PDEs presentation,
we extend our approach to the elliptic problems with  $F^{(4)} \mapsto \Delta^2 F$ in $\ren$, i.e., when dealing with the $N$-dimensional version of  (\ref{N1.1}) and others:
 \be 
 \label{El10} 
 \Delta^2 F=G(F) \equiv  -F+F^2 \quad \mbox{in} \quad \ren, \quad F(\iy)=0
 \quad(\mbox{or} \,\,\,-\Delta^3 F= -F+F^2).
 \ee
  We show that the first  pattern $F_0=F_0(|x|)$ is radially symmetric and is governed by an ODE
 with $\D \mapsto \Delta_r$, which also generates  
  some patterns $F_{2l}(|x|)$ with even indices.
  The same is true for $F_0(|x|)$ for the cubic equation
   $
   \Delta^2 F= - F+F^3$ or $-\Delta^3 F=-F+F^3$,
 where the next $F_1(x)$ and similar others  belonging to the L--S family require a true elliptic setting, {\em etc.}
 For various and related approaches to patterns of the 4th-order Swift--Hohenberg  elliptic equations and others, see \cite{Ll08, LS09} and references therein.
 Therefore, some basic patterns are shown to be  radially symmetric and  can be easily constructed, while
others assume a delicate matching of solutions of linear elliptic problems,  
\cite{AEGII}. Indeed, for such problems for $x \in \ren$,
  any construction of  complicated patterns
lead to difficult  problems even for a P-L approximation of \ef{El10}
of the same form as in 1D, \ef{PL111},
 $$
 \tex{
 \Delta ^2 F= |F-\frac 12|-\frac 12 \quad \mbox{in} \quad \ren, \quad F(\infty)=0.
  }
 $$
Concerning known facts and recent results for higher-order
nonlinear ODEs and related elliptic problems obtained using nonlinear operator techniques,
we refer to a survey on some mathematics results and more applied
aspects in \cite{AEGnegI} and \cite[Ch.~1]{GMPBook}, containing further and new key references on the subject, which equally can be added  to the DSs results and papers/books mentioned above.

\smallskip

{\sc On connections with dynamical system (DS) theory.}
Some aspects of the corresponding  DS theory will be partially  used at the end of the paper,  
In general, for DSs in $\re^4$ like \ef{N1.1} or related ones, from PDE applications  with quadratic and polynomial algebraic, or differential  nonlinearities and especially for the Hamiltonian ones,
there exists a vast literature containing   detailed explanations concerning possible patterns
as homoclinics of the origin in $\re^4$; see e.g. \cite{Champ94} (see more details below) where similar models are analysed
from the point of view of Hamiltonian DSs.

Various related important results for similar and close   Hamiltonian ODEs were obtained in a number of papers. We refer to  \cite{BCT96, Champ94, Champ98, C99, ChK04, Yag21} where a detailed description of patterns can be found.   
We use a DS representation of the obtained results by explaining 
 a general  structure of the two-wing unstable blow-up attractor named $W^{2,\infty}$ containing an infinite number of periodic and chaotic orbits and closely attracting 
a countable number of homoclinics of O in $\re^4$.  An uncountable number of local ``semi-homoclinics" (at $x=-\infty$ only) then blow-up in finite time $x$. 
We also present later on examples of  formal expansions to reveal $F_0$
via some transcendent algebraic systems which remind construction of patterns for the P-L approximation of ODEs.

We do not  present here any quality review/survey or a description of such deep and important results, papers, books and refer to some
 earlier ones, and to the papers mentioned  above we refer to  \cite{Champ94, ChK04, Han20, Hart98, Hoy06, Ll08, LS09, OvShil87, PelRod04, PelTroy, Pis06, Shil65, Shil70} with further important references which can be traced out by the 
{\tt MathSciNet}.

In particular, the papers mentioned above, \cite{BCT96, Champ94, Champ98, Ll08}, and the books \cite{Hoy06, PelTroy, Pis06} contain a vast survey on earlier research in the DS
and various PDE applications of pattern formations  areas.
As an excuse,  we note again that our study 
is  more oriented to the nonlinear PDEs theory  and our techniques are firstly  applied to \ef{N1.1}
as the simplest example to demonstrate the approach. As we have mentioned above, many indefinite operators  with non-Hamiltonian and non-variational features could serve  similarly. We are also interested in applications to various patterns driven by nonlinear elliptic equations in $\ren$ (some references above describe many results in this area).

\smallskip

\ssk

To avoid possible  future and somehow reasonable criticism we claim that in what follows:
 $$
  \begin{matrix}
  \mbox{\em we study some well-known Hamiltonian ODEs from the point of view of the }
    \\
    \mbox{\em nonlinear elliptic and variational theory  
 to be filially applied to PDEs in $\ren$,}  
\end{matrix}
$$
which cannot be covered by the current DS theory.

\section{Exponentially decaying patterns in $\re$: 2D stable and unstable
 asymptotic {manifolds} and some conclusions}
 \label{SexpanA}


Here we need to perform carefully and in full details an elementary asymptotic analysis since it will be crucial in what follows.

\subsection{First asymptotic analysis}

Thus, those 
simple asymptotic features of \ef{N1.1}  are crucial for performing the matching/gluing arguments used throughout  the construction of our families of solutions. 
As we mentioned, earlier, in \cite{AEGnegI},    related
(sometimes radial) versions of elliptic operators with odd nonlinearities in $\ren$ were treated.
 However,  unlike
\cite{AEGnegI}, where L--S and fibering variational approaches were in
charge due to the oddity of {nonlinearities} (functionals are even), here, firstly in the ODE setting,
 we have to  use other approaches.

We begin with the  exponential behaviour of all possible solutions
of \ef{N1.1} as $x \to \iy$. This has been done
in \cite{AEGnegI} for the radial geometry in $\ren$. The 1D exponential
behaviour is similar, but is more involved. Indeed,
in the standard linearized setting, keeping the leading terms in
\ef{N1.1}, we have a linear equation
 \begin{equation}
 \label{inf1}
  \tex{
  F^{(4)}=-F + O(F^2) \asA x \to \iy.
 }
 \end{equation}
Next, as usual, calculating the admissible decaying asymptotics
from \ef{inf1}, as a first approximation (sufficient for our
purposes), we use an exponential pattern leading to the
following characteristic equation: as $x \to \iy$,
 \begin{equation}
 \label{inf111}
  \tex{
 F(x)=\eee^{a x}+... \LongA a^4 +1=0 \LongA a=\pm {\mathrm i} \LongA a^2= \pm \mu  \pm
 {\mathrm i}\,\mu, \,\,\,\mu = \frac 1{\sqrt 2}.
 }
 \end{equation}
 This
 yields a {\em two-dimensional} exponential manifold in both limits $x \to \pm \iy$:
 \begin{equation}
 \label{inf1N}
   \left\{
    \begin{matrix}
 F(x)= {\mathrm e}^{-\mu x} \big[
C_1 \cos(\mu x)+ C_2 \sin(\mu x)\big]+ O({\rm e}^{-2 \mu x})\, , \,\,\, x \to +\iy, \ssk
\\
F(x)= {\mathrm e}^{\mu x} \big[ \hat C_1 \cos(\mu x)+  \hat C_2 \sin(\mu x)\big]+O({\rm e}^{2 \mu x}), \,\,\, x \to -\iy,
\quad 
 \end{matrix}
 \right.
 \end{equation}
 where $C_{i}$ and $\hat C_{i}$, with $i=1,2$, are arbitrary constants.
Note that, for our solutions, the pairs of those four parameters
can be different as $x \to \pm \iy$:
 \begin{equation}
 \label{NotEq1}
 \mbox{$C_{1,2}$ and $\hat C_{1,2}$, in general, do not
 coincide,}
 \end{equation}
  i.e., we are looking for principally non-symmetric 
  patterns, for which $C_{i}=-\hat C_{i}$, with $i=1,2$.
 This is a principal difference with the results in \cite{AEGnegI}
 achieved in the radial geometry, which, obviously, does not allow
 such a variety of patterns in $\ren$.

 A full
 asymptotic expansion   as $x \to \pm \iy$  is  then
 given by the fundamental system of solutions of the linear operator $D_x^4 +I$:
 \begin{equation}
 \label{full1}
 \tex{
  F(x)= {\mathrm e}^{-\mu x} \big[
C_1 \cos(\mu x)+ C_2 \sin(\mu x)\big]  
+  {\mathrm e}^{\mu x} \big[ C_3 \cos(\mu x)+ C_4 \sin(\mu x)\big] +... \,.
 }
  \end{equation}
Hence, for proper exponentially decaying at infinity solutions $F(x)$, one needs:
  \begin{equation}
  \label{full2}
  C_3=C_4=0 \,\, (\mbox{at} \,\, x = +\iy) \,\, \mbox{and} \,\,
 C_1=C_2=0 \,\, (\mbox{at} \,\, x = -\iy).
  \end{equation}

\subsection{Important consequences of the linearized analysis}
Let us begin with  simple  preliminary facts. In the class of even
 functions,
 any regular bounded solution of \ef{N1.1} must satisfy {\em
two} boundary conditions at the origin
 \begin{equation}
   \label{bc1}
   F'(0)=F'''(0)=0.
  \end{equation}
   Thus, using a standard shooting strategy from, say, $x = -\iy$,
    algebraically, at least two parameters are needed to
    satisfy both \ef{bc1}.
Looking again at \ef{inf1N}, where there exist {\em two}
parameters $C_{1,2} \in \re$ (then $\hat C_{1,2}$ are the same by
symmetry), we observe that
 matching with two symmetric boundary conditions \ef{bc1} (for even profiles) yields a
 well-posed and well-balanced algebraic ``2D--2D shooting problem".

Actually, it is a 1D problem since we can write \ef{inf1N} as a
one-parametric manifold:
 \begin{equation}
 \label{OneP1}
 \tex{ 
 F(x)= F_B(x)+O({\rm e}^{2 \mu x})  \,\,\, \hbox{as}\quad x \to -\iy,\,\,\,\mbox{with} 
  \,\,\, F_B(x)=B{\mathrm e}^{\mu x}
 \cos(\mu x), \,\,\, B \in \re,
  }
  \end{equation}
 where the second parameter is 
an arbitrary translation by $x_0 \in \re$. 
  Therefore, shooting by $B$ only, we will get, if possible,  the conditions
\ef{bc1} but at some point $x_0 \in \re$, and then the translated
and after that reflected pattern $F(x+x_0)$ will be a proper even
solution.  In the general case, we have:

 \begin{proposition}
 \label{TwoFund}
There hold:
 \be
 \label{Inv1}
 \begin{matrix}
 \mbox{{\rm (i)} Each pattern $F$ with a given $B$ is an isolated solution in $C(\re)$, and}
 \\
 \mbox{{\rm (ii)} for any constant $B_0>0$, all patterns $\{F\}$ are described by some},\,\,  B \in [B_0,B_0 {\rm e}^{2 \pi}). \,\,\,
 \end{matrix}
  \ee 
  \end{proposition}
  
  \begin{proof}
  {\rm (i)} For the analytic ODE \ef{N1.1},
  any local or global (a pattern) solution $F(x)$  satisfying \ef{OneP1} (i.e., $x$-translations are not allowed)
is uniquely determined by a fixed constant $B$ therein since the solution  is uniquely given by an analytic expansion in $B$
 uniformly converging on  intervals $(-\iy,L]$ with $L=L(B) \ll -1$.  
 \be 
 \label{Inv19}
 \tex{
\mbox{\rm (ii)}\,\, \mbox{by translation:} \,\, F_B(x+\frac {2 \pi}\mu)= B{\rm e}^{2\pi} {\rm e}^{\mu x} \cos(\mu x) \equiv F_{B'}(x),
\,\,\,
B'=B{\rm e}^{2\pi} \quad (\mu = \frac 1{\sqrt  2}).
 }
\ee\end{proof}

Thus, in view again  of the translational invariance of the
ODE, this parameter $B\in \re$ in \ef{OneP1} controls {\bf all}
the patterns of the problem \ef{N1.1}. Indeed, once we get
 a nontrivial solution of \ef{N1.1}, we can move it along the $x$-axis
 in a such a manner that when its asymptotic exponential tail, as $x \to
 -\iy$, takes the form \ef{OneP1} with some $B \ne 0$, which 
{\em uniquely} defines  this profile. Once again, this is also true since
 the obtained $B$ uniquely defines all {\em four} constants
$C_{1,2,3,4}$ ($C_1=C_2=0$ by \ef{full2}) of the unique solution
of the {\em 4th-order} ODE in \ef{N1.1}.

 Since our flow is analytic,
such a 1D shooting by a single parameter $B$ (or by a similar
$D \in \re$ at $x=+\iy$) cannot provide us with more that a {\em 
countable set of patterns}, with a possible  concentration point
at infinity only, meaning, roughly speaking, that these eventually
can be either very large in $L^\iy$ or very wide in the $x$-axis. To
become very large seems not an option since the corresponding
non-stationary parabolic equation
 $$
 u_t=-u_{xxxx} - u + u^2 \inB \re \times \re_+
 $$
admits finite-time blow-up for large initial data; see \cite{GMPBook} and earlier blow-up results in  \cite{GVSur02} for further details. We will discuss
some peculiarities of such a blow-up in Section\;\ref{S2N}.

 Of course, for the same 2D-2D, 3D-3D (to be
considered) and more equal dimensions of manifolds to be
matched/glued  for such analytic flows, the result remains the
same.
 We will discuss this important issue in a
greater detail.

For odd solutions, which  do not exist for our non-definite
operators but will be used for some odd ones later on, the
anti-symmetry/dipole-like conditions are posed:
 \begin{equation}
 \label{bc1anti}
 F(0)=F''(0)=0,
 \end{equation}
  which also form an algebraically non-contradictory matching. Then here $\hat C_{1,2}=-C_{1,2}$,
  so again two parameters occur, and a 1D shooting via expansions of the type \eqref{OneP1}
  (plus the translational parameter) occurs.
  We already know from \cite{AEGnegI} that for operators with odd nonlinearities, in the variational sense, both boundary conditions
  \eqref{bc1} and \eqref{bc1anti} can produce critical points of the
  L--S type, and even a countable subset of those for non-coercive
  operators. However, beyond L--S solutions, the analysis performed in this paper show that there exist countable
  sequences of other patterns, regardless definite or indefinite
  operators are involved, and we show this fact for the indefinite one
  in \ef{N1.1}.

  Consequently, our main goal here is to show that such boundary conditions like \ef{bc1} (for even patterns) and \ef{OneP1} (for arbitrary ones) at infinity
  produce an infinite countable subset of other solutions or critical points of the functional. Recall again
   that the L--S theory does not apply here.
   We are going to
  detect  patterns  which do not exhibit any symmetry
  or anti-symmetry of their geometric shapes via a matching/gluing argument.

\section{Variational approach for the first pattern and some numerics}
\label{S2.LS}


\subsection{Pohozaev's  radial fibering and the first (basic) variational pattern $F_0$}

Variational approaches in various metrics  apply to a number of  equations  in $\re$ or $\ren$ mentioned above,
including all the Hamiltonian ones (usually these require $L^2$-metric only). 
 See  typical examples in  \cite{AEGnegI,PV} and  \cite[Ch.~1]{GMPBook}, where further references and results  in particular  related to the {\em Pohozaev's Fibering Method} (PFM) (proposed in \cite{Poh79} and developed  in \cite{Poh84}--\cite{Poh08})
to be applied can be found.

Concerning our working convenient  model
 \ef{N1.1}, it admits a simple variational setting in $L^2$ by finding critical values and points of
 the functional
\begin{equation}
 \label{F01}
  \tex{
 \Phi(F)= \frac 12 \int [(F'')^2 +F^2] - \frac 13 \int F^3\inB H^2(\re),  \quad 
 \,\,\, \Phi'(F)=F^{(4)}+F-F^2.
 }
 \end{equation}
Clearly there exists the trivial solution
 \be  
 \label{Tr1}
 F(x)=0, \quad \Phi'(0)=0, \quad  \mbox{and} \quad \Phi''(0)=(D_x^4+I-3 F^2I)\vert|_{F=0}= D_x^4+I>0,
 \ee
 so that the critical point $F=0$ is a (at this moment local) minimum of $\Phi(F)$.
 
 Looking for nontrivial  critical points, according to Pohozaev \cite{Poh79} 
 we consider \ef{F01}
on a  natural for \ef{F01} fibering  manifold 
\begin{equation}
 \label{F02}
 \tex{
 H_0= \big\{v \in H^2(\re): \,\,\int [(v'')^2 + v^2]=1 \big\}
 }
 \end{equation}
 (almost a ``unit sphere" in $H^2$).
Thus, as customary,  we first apply the {\em radial} PFM:
\begin{equation}
 \label{F03}
 F=r(v)v, \quad v \in H_0,
 \end{equation}
 where $r(v)\in \re\setminus \{0\}$ is a scalar function.
 This gives the functional of two variables $(r,v)$:
\begin{equation}
 \label{F04}
 G(r,v) = \Phi(r(v)v)), \quad v \in H_0.
  \end{equation}
  Then (see Pohozaev's Lectures \cite{Poh08}) in a rather general case of $C^1$-functional with some natural assumptions 
  for standard regular  functionals like \ef{F01} any conditional critical
  point of $G(r,v)$: 
\begin{equation}
   \label{F05}
   \left\{
   \begin{matrix}
   G'_r(r,v)=0,
   \\
    G'_v(r,v)=0,
    \end{matrix}
    \right.
     \end{equation}
  gives  a critical point \ef{F03} of $\Phi(F)$, $\Phi'(F)=0$ with $F=r(v)v$.
 Since on $H_0$ there holds
\begin{equation}
  \label{F06}
  \tex{
  G(r,v) \equiv \frac 12 \, r^2 -\frac 13 r^3 \int v^3, 
\quad r \ne 0,
 }
 \end{equation}
  from the first scalar equation in \ef{F05} we have ($r=0$ leads to \ef{Tr1})
  \begin{equation}
  \label{F07}
  \tex{
  G'_r(r,v)=r-r^2 \int v^3=0, \quad r \ne 0.
}
 \end{equation}
This defines a unique solution $r=r_0(v) \ne 0$  as a local maximum point of the function \ef{F06}:
\begin{equation}
  \label{F08}
  \tex{
r_0(v)= \big({\int v^3}\big)^{-1}.
 }
 \end{equation}
 The second equation in \ef{F05} then yields a standard variational
 problem of the form
 \begin{equation}
  \label{F09}
  \tex{
 \mbox{min/max}_{H_0} \, \Phi(r_0(v)v), \,\,\, \Phi(r_0(v)v)= \frac 16\, \big(\int v^3 \big)^{-2},
 \,\,\,\, \mbox{or} \,\,\,\, \mbox{max/min}_{H_0} \, \int v^3.
 }
 \end{equation}
 Since the last functional is odd, both max and min are related  by $v \mapsto
 -v$ giving some $\pm v_0$.

  For similar equations with 
  odd-order nonlinearities like 
  $$
 F^{(4)}=-F+F^3 \,\,\, \mbox{or} \,\,\, ((F')^2 F'')''=-F-(F^5)'' \,\,\,(\mbox{an exotic variational in $H^{-1}(\re)$})
 $$
  and others, this fibering approach
 establishes  a connection with the L--S category/genus theory
 \cite{Berger,Kras,Kras51,KrasZ}. 
 Unfortunately, in the present case with indefinite operators, the final fibering result \eqref{F09}
 principally does not allow such an effective L--S application to guarantee existence of an L--S countable sequence of critical points as usually can happen for suitable even smooth functionals; see \cite[Ch.~6]{Berger} for a full L--S critical point theory. Extra comments and   applications to similar elliptic equations in $\ren$ can be found in \cite{AEGnegI}.  
 
 \ssk

Thus, we arrive at   a direct minimax problem \ef{F09} and consider it here  as a simple but typical example of such calculus. For $N=1$ all the necessary embeddings are valid while in $\ren$ one needs to be in the subcritical Sobolev range in order to deal with classical solutions.
We have  $v,v',v'' \in
L^2(\re)$ ($v''$ is a distribution) on $H_0$, so that,
 we deal with $C^1$-smooth functions (actually, we deal with analytic  exponentially decaying
functions only and the critical points must be among them) and 
$v(x), \, v'(x) \to 0$ as $x \to \iy$.
 The Cauchy-Buniakovskii inequality yields
$$
\tex{
 \int (v')^2 = \int v' \, v'= - \int v v'' \le
 \big( \int v^2\big)^{\frac 12} \big(\int (v'')^2 \big)^{\frac
 12}.
 }
$$
Using  it again  yields
\begin{equation}
 \label{emb1}
 \begin{matrix}
 v^2(y) = \int_{-\iy}^y (v^2)'
  = 2 \int_{-\iy}^y v v'
 \le 2 \big( \int v^2\big)^{\frac 12}
\big( \int (v')^2\big)^{\frac 12}
 \le 2\big( \int v^2\big)^{\frac 34}\big( \int (v'')^2\big)^{\frac
 14}
  \\
\LongA  |v| \le  {\sqrt 2}\big( \int v^2\big)^{\frac 38}\big( \int
(v'')^2\big)^{\frac
 18}.
 \end{matrix}
 \end{equation}
Therefore,  by \ef{emb1} on $H_0$
$$
  \tex{
  \big| \int v^3 \big| = \big| \int v \, v^2  \big| \le
  {\sqrt 2}
\big(\int v^2 \big)^{\frac 38}\big(\int (v'')^2 \big)^{\frac 18}
\, \int v^2 = {\sqrt 2} \big(\int v^2 \big)^{\frac {11}8}\big(\int
(v'')^2 \big)^{\frac 18} \le {\sqrt 2}, }
 $$
 i.e. this functional is bounded  from above (and from below
 as being odd via $v \mapsto -v$). This defines the absolute
  maximum of the functional
\begin{equation}
 \label{F55}
  \tex{
  c_0= {\mbox{sup}}_{H_0} \, \int v^3 >0.
   }
 \end{equation}
 Hence, \ef{F09} gives  the first critical point $v_0 \in H_0$ such that 
 \be  
  \label{F91}
 \tex{
 c_0=\int v_0^3.
 }
 \ee
 For the
  functional in \ef{F09}, this means that
  $v_0$
  satisfies
\begin{equation}
   \label{F10}
   \tex{
  \mu(v_0^{(4)}+ v_0) - v_0^2=0 \,\,\,(\Longrightarrow \mu=\int v_0^3=c_0),
  }
 \end{equation}
where $\mu=\mu(v_0) \ne 0$ is the corresponding Lagrange's
multiplier (for the point $-v_0$, we have $\mu(-v_0)= -
\mu(v_0)$). Setting
\begin{equation}
  \label{F19}
  \tex{
 v_0 = \mu F_0, \quad v_0= \frac {F_0}
 {\sqrt{\int[(F_0'')^2+F_0^2]}}\in H_0,
}
 \end{equation}
 yields the first critical point $F_0 \ne 0$ of the functional \ef{F01}
 and hence a solution $F_0(x)$ of \ef{N1.1}.
It follows from \ef{F91} and \ef{F19} that 
the first
critical values $\Phi(r_0(v)v)$ and $\Phi(F)$  satisfy
\begin{equation}
 \label{F22}
 \begin{matrix}
 c_0= \int v_0^3 \equiv \frac{ \int
 v_0^3}{\{\int[(v_0'')^2+v_0^2]\}^{\frac 32}} \equiv 
  \frac{ \int
 F_0^3}{\{\int[(F_0'')^2+F_0^2]\}^{\frac 32}} \equiv \frac 1 {\sqrt{\int[(F_0'')^2+F_0^2]}}= \frac 1{\sqrt{\int F_0^3}}
 \\
 \mbox{and} \quad C_0 \equiv \Phi(F_0)=\Phi(\frac {v_0}{c_0})=  \frac 1{2c_0^2}-\frac {\int v_0^3}{3c_0^3}=
 \frac 1{6c_0^2}.
 \end{matrix}
 \end{equation}

\subsection{Numerical gluing of $F_0(x)$'s and other patterns}
As we have mentioned, various analytic and numerical examples of homoclinic orbits of several related Hamiltonian DSs (including sometimes directly \ef{N1.1}) can be found in \cite{BCT96}--\cite{ChK04} where further references and results can be found. We need to concentrate on some specific features of such nonlinear ODEs actually regardless there Hamiltonian nature.

In Fig. \ref{f7-1}, we begin with  the first (and the simplest)
variational solution of \ef{N1.1}. Here we use the {\tt MatLab}
{\em bvp4c} solver with typical
 $$
 \mbox{Tols}= 10^{-2} \quad \mbox{and} \quad  \mbox{Tols}= 10^{-3}
 $$
 or less for extra delicate and complicated combinations of possible simpler
  profiles. Note that solving \ef{N1.1} on a large interval $(x_0,x_1)$ we use the fact that
   $$
   \mbox{Dirichlet b.c.'s $F=F'=0$ at $x=x_{1,2}$ forbid  translational invariants of the ``discrete" ODE},
   $$
   so that numerical patterns $F(x)$ are  isolated, cf. Proposition \ref{TwoFund}.

\begin{figure}[htbp]
 \hfill   \hfill  \begin{minipage}[t]{0.45\textwidth} 
    \centering 
    \includegraphics[width=6cm,keepaspectratio]{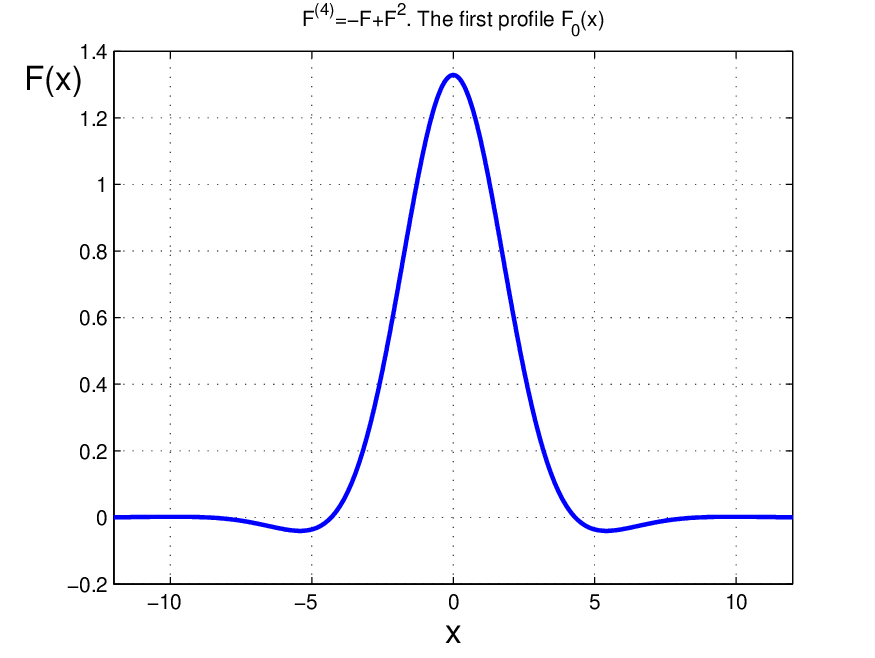}
    \caption{The first variational solution $F_0(x)$ of the ODE
problem \ef{N1.1}.}
\label{f7-1}
\end{minipage}
  \hfill   \hfill 
\begin{minipage}[t]{0.45\textwidth}
    \centering 
    \includegraphics[width=6cm,keepaspectratio]{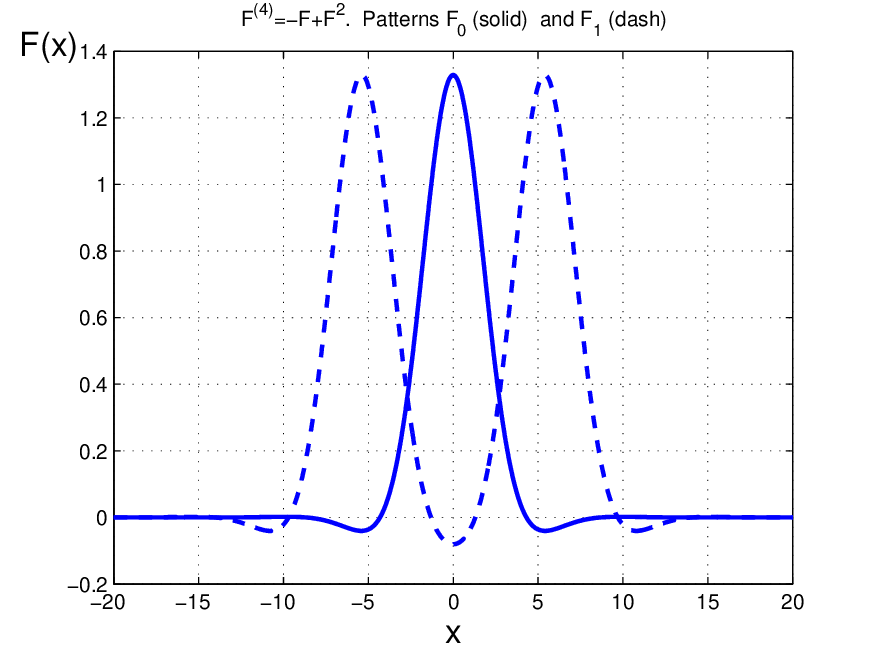}
    \caption{Two simplest  solutions $F_0$ and $F_1$  of the ODE
problem \ef{N1.1}.}
\label{f7F0F1}
\end{minipage}
\end{figure}

In Fig. \ref{f7F0F1}, we show just two basic
patterns $F_0$ and $F_1$ from an infinite countable family $\{F_k,
k=0,1,...\}$.

In Fig. \ref{f7-NN1}, we show how separated in space basic
profiles $F_0(x)$ can create more complicated patterns. A formal
procedure of such a gluing of an arbitrary pair and eventually of
an arbitrarily large finite number of isolated profiles will be
discussed below.

\begin{figure}[htbp]
 \hfill   \hfill  \begin{minipage}[t]{0.45\textwidth} 
    \centering 
    \includegraphics[width=6cm,keepaspectratio]{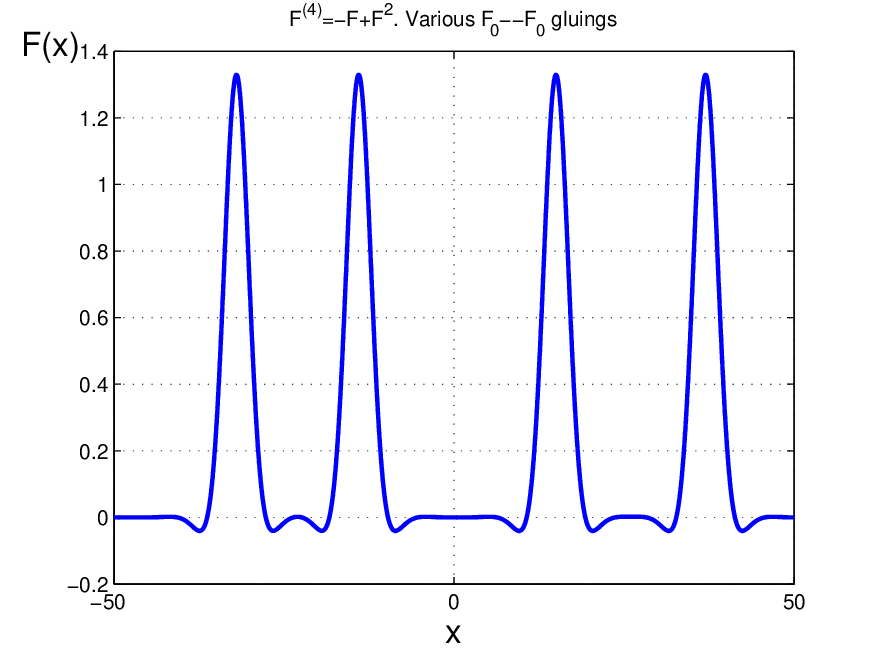}
    \caption{Various gluing together several basic profiles  $F_0(x)$
via exponential tails for the ODE problem
\ef{N1.1}.}
\label{f7-NN1}
\end{minipage}
  \hfill   \hfill 
\begin{minipage}[t]{0.45\textwidth}
    \centering 
    \includegraphics[width=6cm,keepaspectratio]{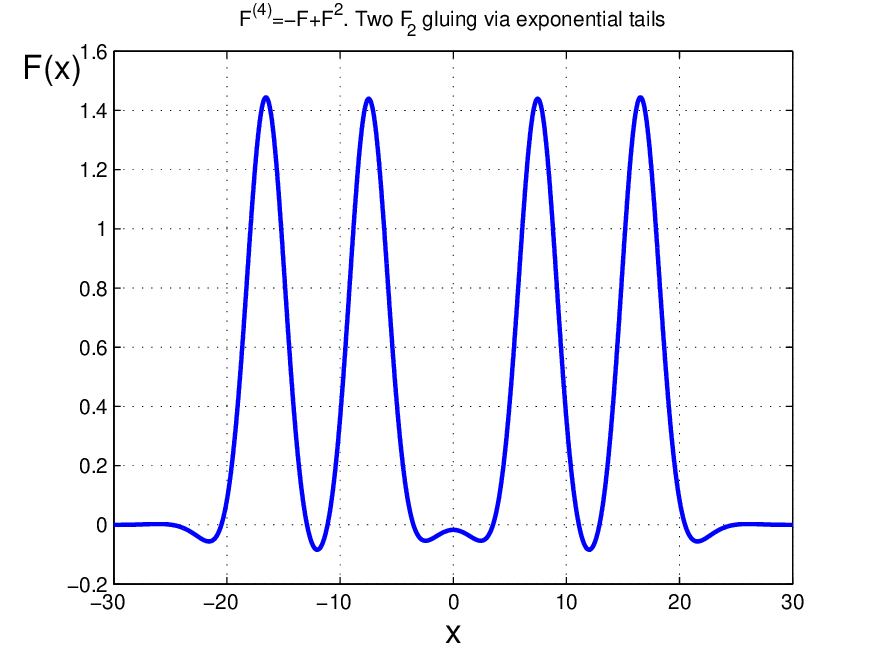}
    \caption{Two $F_1=F_0-F_0$ solutions of \ef{N1.1} matched at $x=0$
via asymptotic tails.}
\label{f7-2}
\end{minipage}
\end{figure}

{\sc The first basic family ${\mathcal F}_1=\{F_k\}$ via a periodic orbit $\Gamma_{\rm max}$.}
We next discuss more complicated solutions $F(x)$.
 Fig. \ref{f7-2}
shows two $F_2$ patterns obtained by the first gluing of the
elementary patterns $F_0-F_0$. In some sense, those double,
triple, {\em etc.,} patterns  $F_2$, $F_3$,...  look like a
standard L--S sequence of {\em minmax} critical points (of course, they are not). Recall
that for even functionals  those essentially  changing sign critical points are
constructed by using 
 the reflection $u \mapsto -u$ and characterize the
genus of each set which they belong to. Such critical points and
solutions are not available here for indefinite operators.
We will show that this basic family of patterns is organized as follows:
  \be 
  \label{max1}
  \begin{matrix}
{\mathcal F}_1= \{F_k\}_{\{k \ge 0\}} \,\,\, \mbox{is composed from  finite pieces of some periodic}
  \\
  \mbox{orbit $\Gamma_{\rm max}$ via gluing with exponentially decaying tails as $x \to \pm \iy$}.
  \end{matrix}
  \ee

\smallskip

{\sc The second basic family ${\mathcal F}_2=\{F_{+2l}\}$ via the periodic orbit $\Gamma_{\rm min}$.}
In Fig. \ref{f7-NN2}, we show the same first basic  profile
$F_0(x)$ (the solid line) and two new types of profiles called
$F_{+4}$ (the dash-line), which are matched at $x=0$. Those
$F_{+4}$ have exactly 4 intersections with the steady state of the
ODE \ef{N1.1} $
 F_*(x) \equiv 1.
  $
This second basic family of patterns  ${\mathcal F}_2=\{F_{+2l}\}$ is generated by another
 periodic orbit $\Gamma_{\rm min}$ in the same sense as in \ef{max1}.  
  
  \smallskip
  
  {\sc Further related patterns.}
According to our standard (simplified) classification, in Fig. \ref{f7-NN3}, we
observe a rather complicated profile of $F_{+4}-F_{+4}$ gluing,
which is characterized as
  \be
  \label{Fpr1}
 F_{+4,2,+4}.
  \ee
 The given 2 in between +4's  counts the number of zeros
 observed in between the two $F_{+4}$ profiles. Sometimes, we will use
 this number of zeros for measuring the number of minimum points around the
 matching area, when number of zeros is not sufficient to
 characterize the present ``geometry"
 of the pattern under consideration.

 In general, a classification of patterns via an index 
 like \ef{Fpr1} is not straightforward at all and moreover is not possible in such a simple way, so
 such indexes $\sigma$'  are used for convenience.
 For a correct way to ascribe indexes $\s$ to all the patterns, see 
  \cite{AEGII}.

\begin{figure}[htbp]
 \hfill   \hfill  \begin{minipage}[t]{0.45\textwidth} 
    \centering 
    \includegraphics[width=6cm,keepaspectratio]{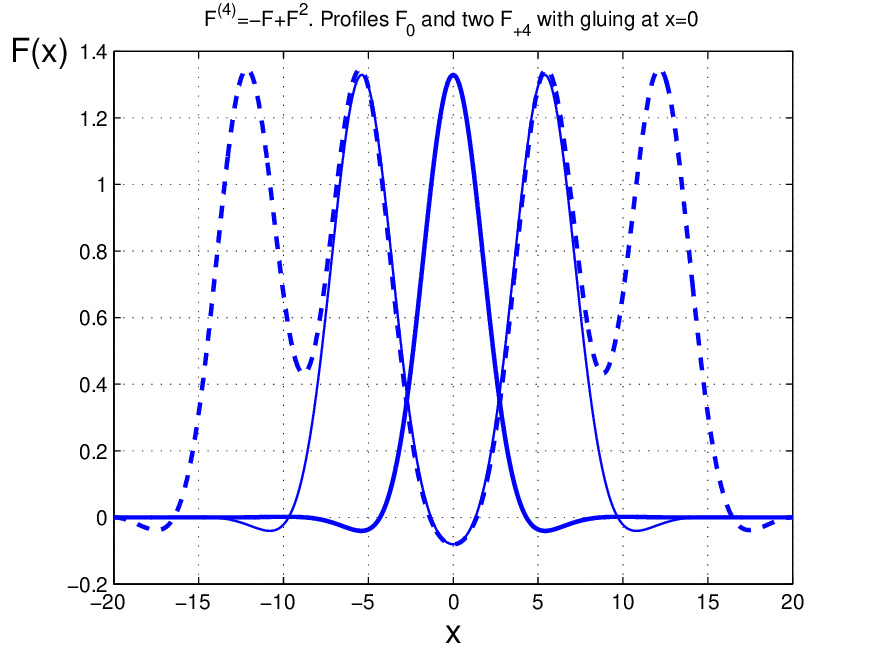}
    \caption{The first  $F_0(x)$ and two $F_{+4}$'s of the ODE problem
\ef{N1.1}.}
\label{f7-NN2}
\end{minipage}
  \hfill   \hfill 
\begin{minipage}[t]{0.45\textwidth}
    \centering 
    \includegraphics[width=6cm,keepaspectratio]{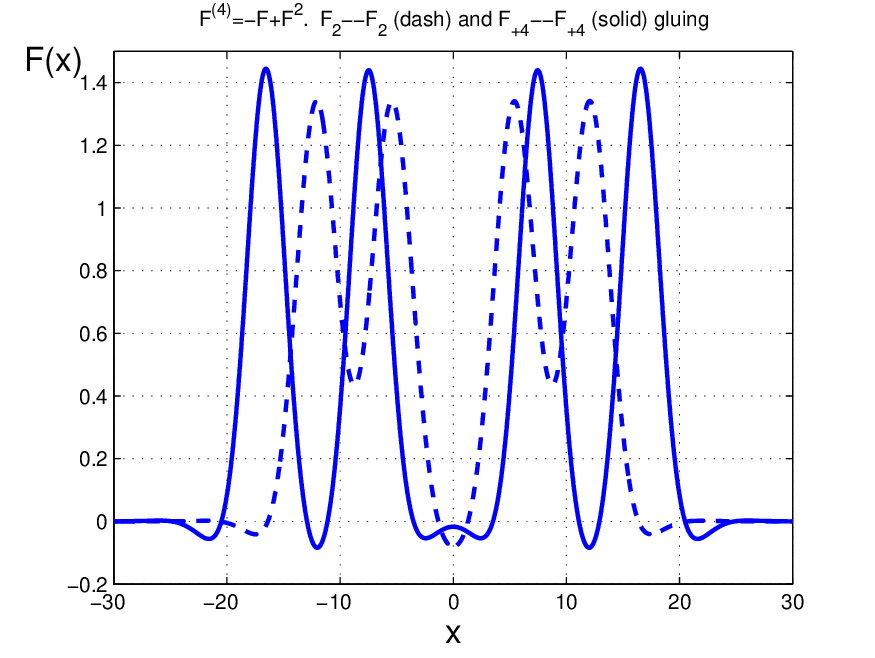}
    \caption{Gluing $F_2-F_2$ and $F_{+4}-F_{+4}$ of the ODE problem
\ef{N1.1}.}
\label{f7-NN3}
\end{minipage}
\end{figure}

In Fig. \ref{f7-NN4} and \ref{f7-NN5},  we present further
matching/gluing of $F_0$'s to create new ODE solutions. More
complicated solutions are presented in Fig. \ref{f7-NN4}, while
Fig. \ref{f7-NN5} shows enlarged zero and local minmax points of
such profiles.
 Observe that a clear evidence of the fact that matching/gluing
of various $F_0$, $F_1$ and other profiles happen pointwise
according to the period of decaying asymptotic tails of solutions
\ef{inf1N}, though the first gluing can be more nonlinear but
their discrete  nature is out of a reasonable discussion: matching
(gluing together) cannot occur at a continuum subset of points.

\begin{figure}[htbp]
 \hfill   \hfill  \begin{minipage}[t]{0.45\textwidth} 
    \centering 
    \includegraphics[width=6cm,keepaspectratio]{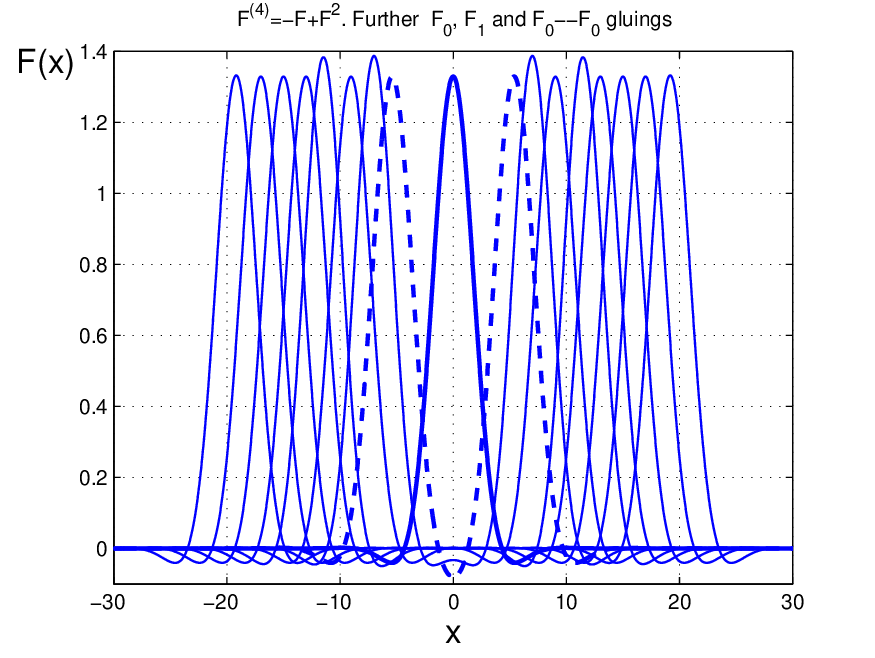}
   \caption{Complicated solutions  of the ODE problem
\ef{N1.1}.}\label{f7-NN4}
\end{minipage}
  \hfill   \hfill 
\begin{minipage}[t]{0.45\textwidth}
    \centering 
    \includegraphics[width=6cm,keepaspectratio]{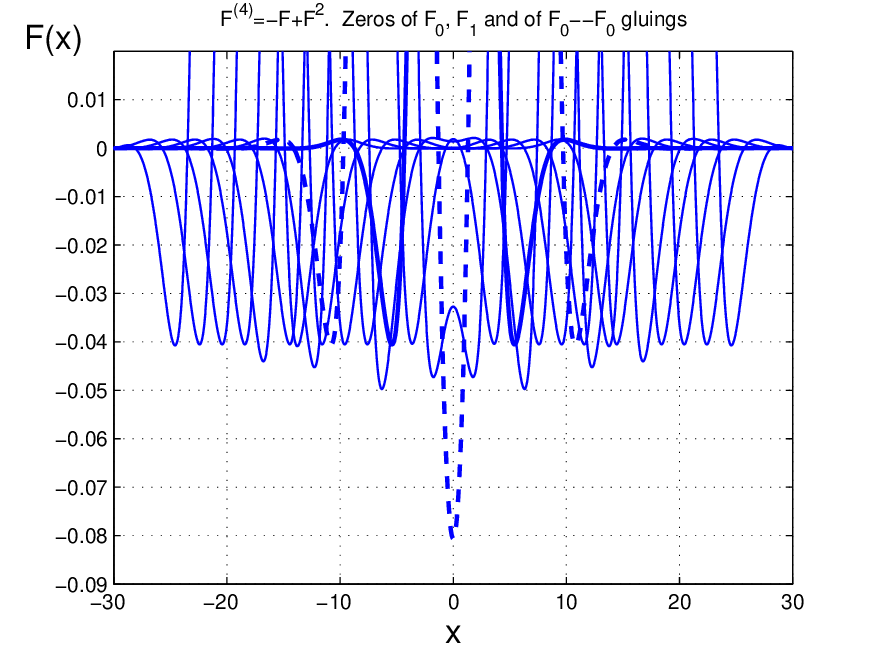}
   \caption{Enlarged zero structure of  solutions from Fig.
\ref{f7-NN4}.}\label{f7-NN5}
\end{minipage}
\end{figure}

\begin{figure}[htbp]
 \hfill   \hfill  \begin{minipage}[t]{0.45\textwidth} 
    \centering 
    \includegraphics[width=6cm,keepaspectratio]{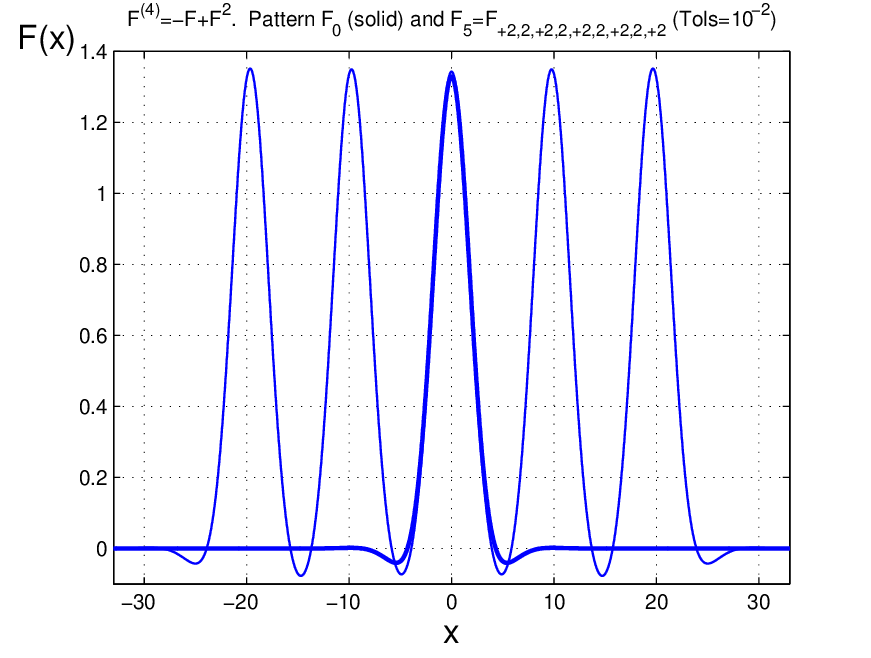}
  \caption{Pattern $F_5$.}\label{f7-NN21}
\end{minipage}
  \hfill   \hfill 
\begin{minipage}[t]{0.45\textwidth}
    \centering 
    \includegraphics[width=6cm,keepaspectratio]{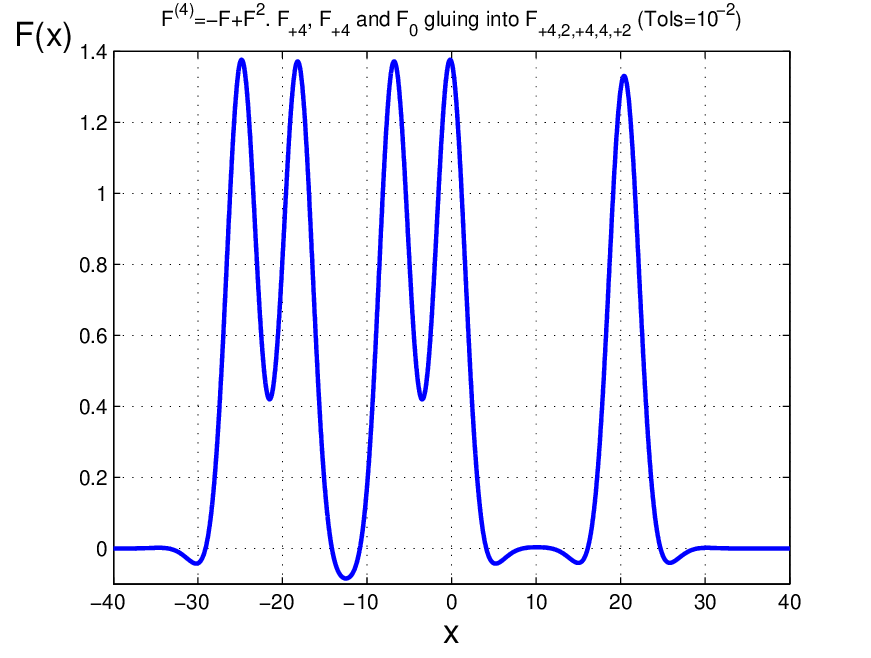}
  \caption{Matching two $F_{+4}$ and one  $F_0$.}\label{f7-NN22}
\end{minipage}
\end{figure}

Below, we present more Figures with typical properties of various
solutions of the ODE problem \ef{N1.1} obtained by gluing simpler
patterns.

\begin{figure}[htbp]
\begin{center}
\subfigure[Three $F_{0}(x)$, one $F_{+4}(x)$]{
\includegraphics[scale=0.45]{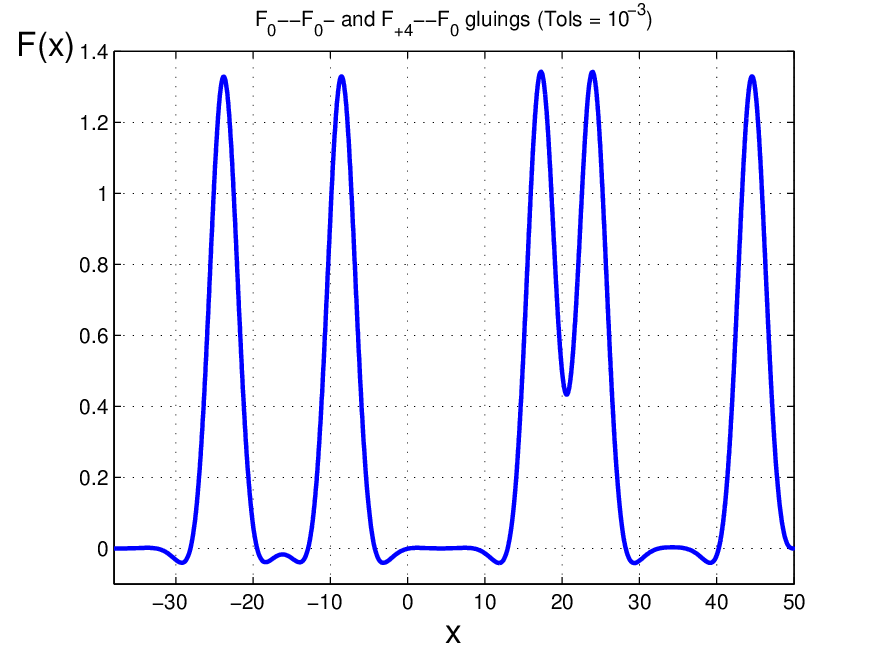} 
}
\subfigure[Zero set]{
\includegraphics[scale=0.45]{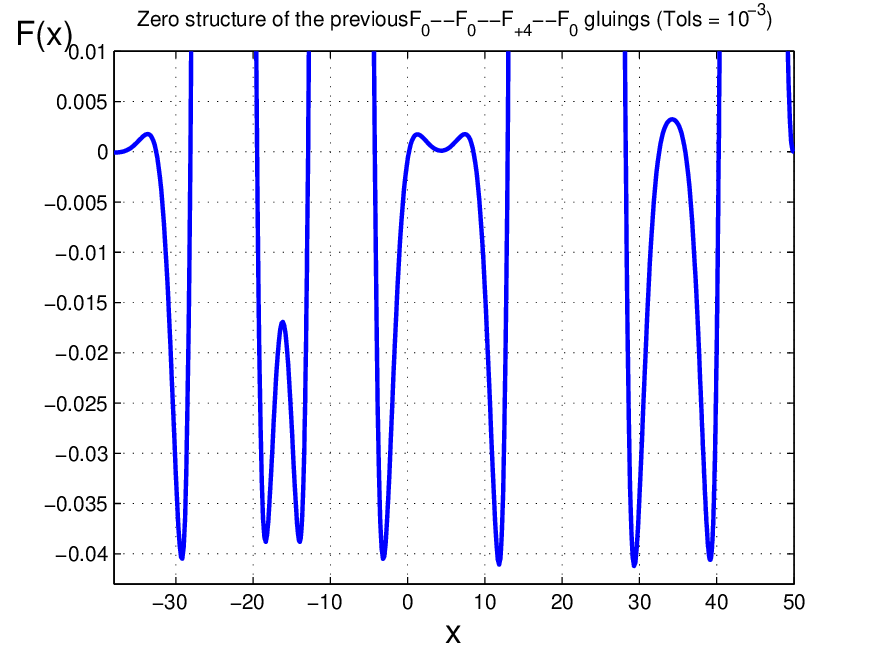}
}
\caption{Three $F_0$ and one $F_{+4}$ gluing together for  \ef{N1.1}
with zero-set.}
 \label{f4-NN7}
\end{center}
\end{figure}

\begin{figure}[htbp]
\begin{center}
\subfigure[Patterns  $F_{+6}(x)$, $F_{+8}(x)$]{
\includegraphics[scale=0.45]{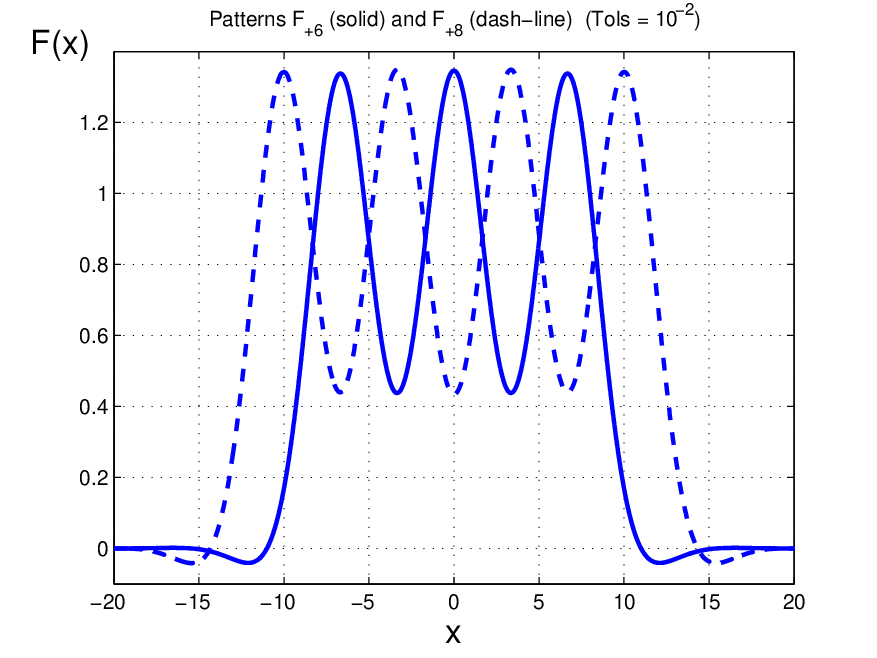} 
} \subfigure[Other patterns]{
\includegraphics[scale=0.45]{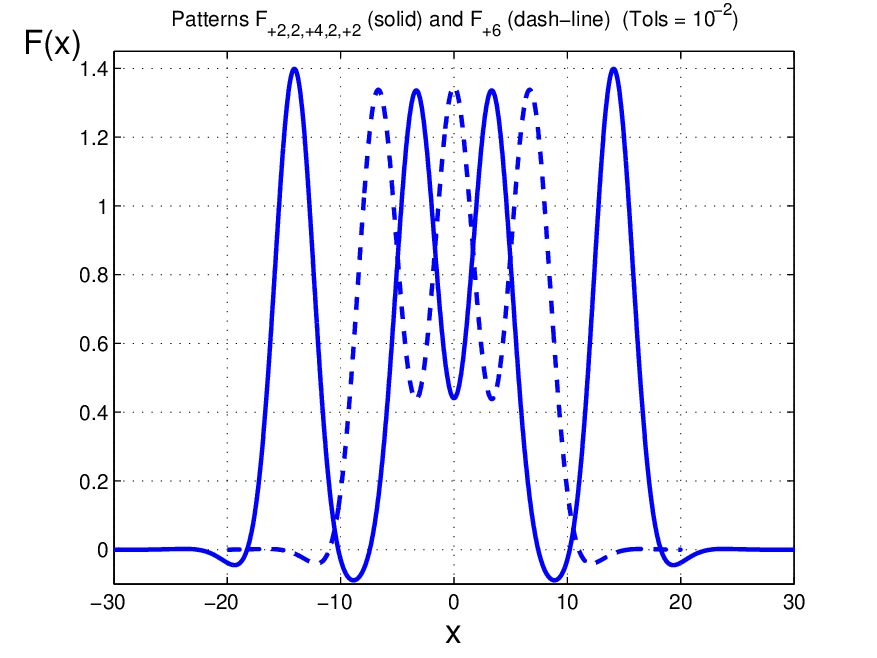}
}
\caption{Patterns $F_{+6}$, $F_{+8}$, and others.}
\label{f4-NN8}
\end{center}
\end{figure}

\begin{figure}[htbp]
\begin{center}
\subfigure[Transition $3F_0\to F_3\to F_{+6}$]{ 
\includegraphics[scale=0.45]{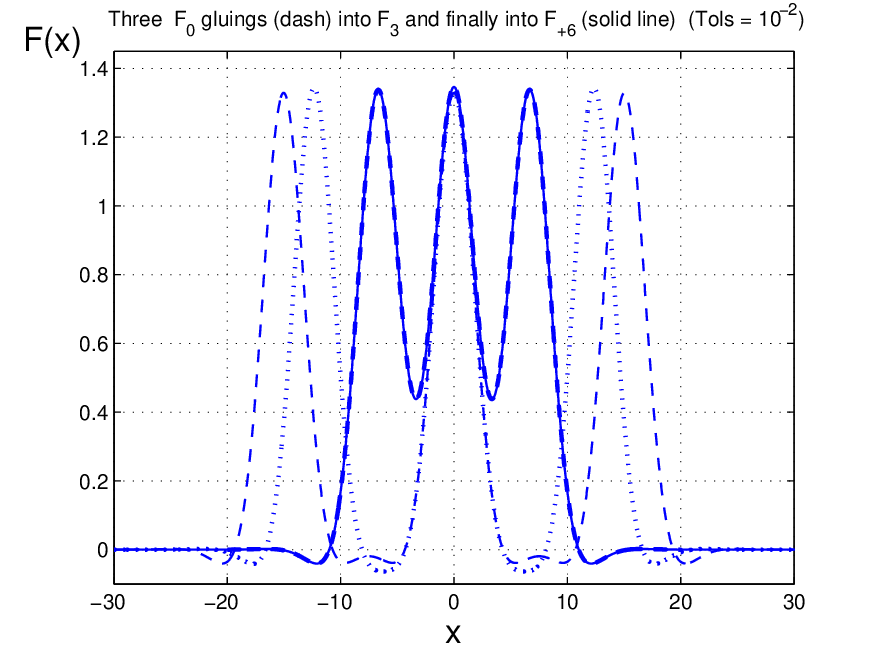} 
} \subfigure[$F_{+12}(x)$]{  
\includegraphics[scale=0.45]{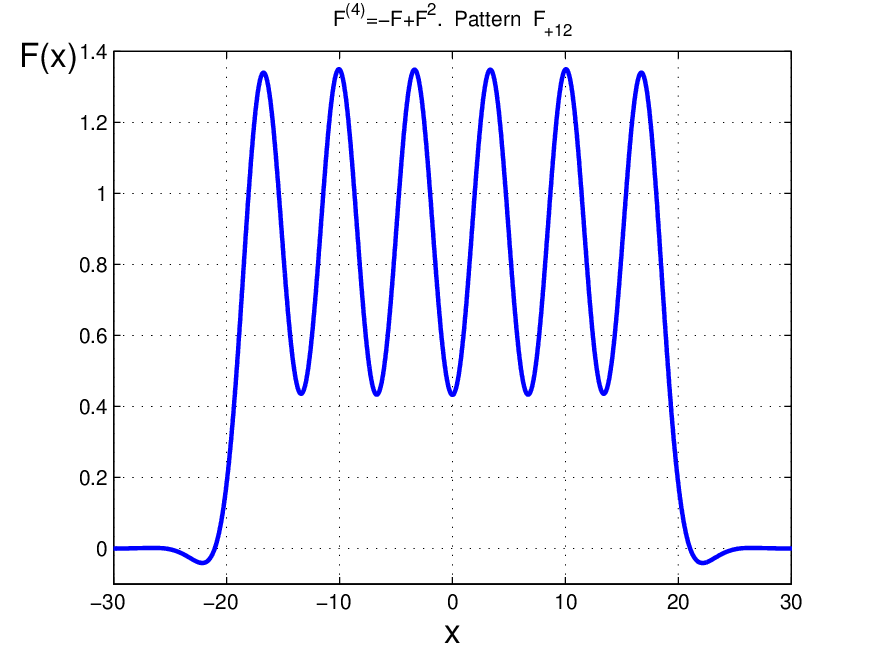} 
}
\caption{Patterns $F_{+2,4,+2,4,+2}$, $F_3$, $F_{+6}$ and $F_{+12}$.}
\label{f4-NN9}
\end{center}
\end{figure}

\

\subsection{More numerics: a direct shooting from the exponential tail}

We use now a shooting approach directly  from sufficiently small exponential tails and apply
 another solver {\em ode45} of {\tt MatLab} with, at least,
  $$
  \mbox{Tols}= 10^{-10} \,\,\,\mbox{or less}.
  $$
  Such a shooting is necessary and unavoidable  for more {\em divergent} models where a standard iteration procedure as in  the {\em bvp4c}
 often cannot deliver a result in view of  conservation laws, conditions, and more complicated subsets of patterns. As a typical example among similar others, this happens for already announced  degenerate equations like (this is an ``ugly"  but interesting equation from such a family and it is variational in $H^{-1}$!)
  $$
  ((F')^3)'''=-F-(F^5)''  \inB \re.
  $$ 
  
First, we shoot from the left-hand side by using the
equivalent representation
of patterns \ef{OneP1} via 
 their exponentially decaying tail as $x
\to -\iy$.
This  representation of the leading expansion term assumes that
the full 2D tail \ef{full1} with constants $C_{3,4}$ ($C_{1,2}=0$)
is recovered from \ef{OneP1} by the invariant translation in $x$.

In Fig. \ref{Anal1}, we show shooting from the left hand side by \eqref{OneP1}
with the data
  $$
  x_0=-30 \,\, (\mbox{the initial shooting point}) \andA B=0.0000828,
  $$
  when the pattern $F_0(x)$ appears. Continuing the same shooting,
  in the next Fig. \ref{Anal2},  we observe $\sim F_{+4}(x)$ and
  even a transition to $\sim F_{+6}(x)$ for
   $$
   B=0.0000830016.
   $$
 Further improving the $F_{+6}$-structure was beyond our accuracy (and
 patience).

\begin{figure}[htbp]
\begin{center}
\includegraphics[scale=0.52]{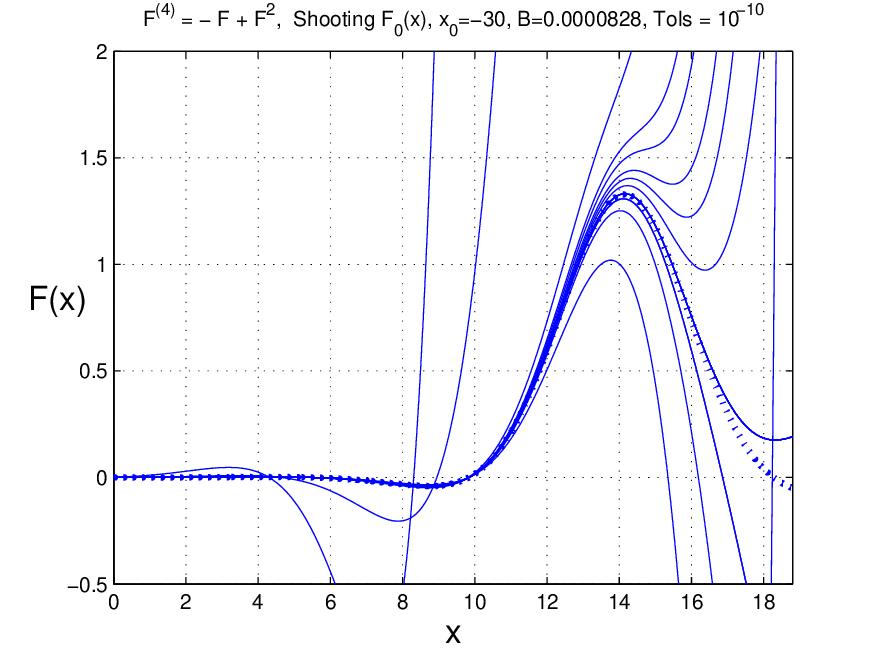}
\caption{Shooting $F_0(x)$.}\label{Anal1}
\end{center}
\end{figure}

\begin{figure}[htbp]
\begin{center}
\includegraphics[scale=0.52]{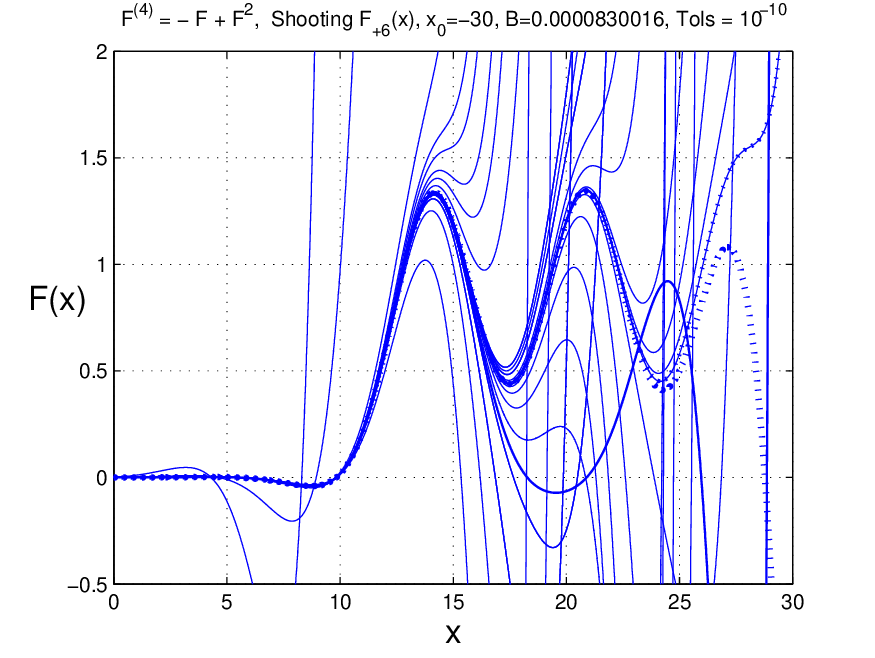}
\caption{Shooting $F_{+4}$ and trying  $F_{+6}$.}
  \label{Anal2}
\end{center}
\end{figure}

In order to match various structures at $x=0$, we take a sum of
two proper exponential expansions to create a symmetric even one and
 similar to  \ef{OneP1} as $x \to -\iy$ we have
 \be
 \label{BBB2}
 \begin{matrix}
 F(x)=\bar F_B(x) +O(B^2) \,\,\, \mbox{for small} \,\,\, |B|>0, \,\,\, \mbox{where}
 \\
\bar F_B(x) \equiv \frac 12 \,
 [F_B(x)+F_B(-x)]=B\cosh(\mu x)\, \cos(\mu x),
  \end{matrix}
 \ee
 so that we can match a pattern in a symmetric even way in both directions
 $x>0$ and $x<0$. In Fig. \ref{Anal3}, we show how to get $F_{+4}(x)$
 on the right-hand side for the shooting parameters
  $
  x_0=0 \andA B=-0.0518561<0.
  $
Reflecting this Figure by $x \mapsto -x$ yields the pattern
 $
 F_{+4,2,+4}(x).
  $
Similarly, in Fig.  \ref{Anal4}, for
 $
 x_0=0 \andA B=0.0022782>0,
 $
we see $F_0(x)$ on the right-hand side and, hence, by reflection
 $
 F_{+2,4,+2}(x).
 $

\begin{figure}[htbp]
\begin{center}
\includegraphics[scale=0.52]{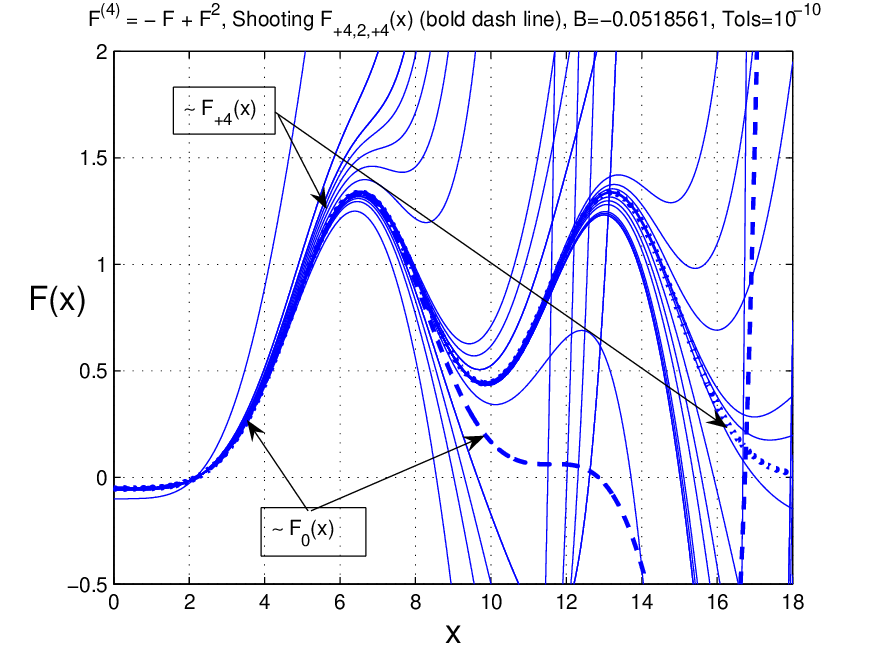}
\caption{Patterns $F_0$,  $F_{+4}$, and overall
$F_{+4,2,+4}$.}\label{Anal3}
\end{center}
\end{figure}

\begin{figure}[htbp]
\begin{center}
\includegraphics[scale=0.52]{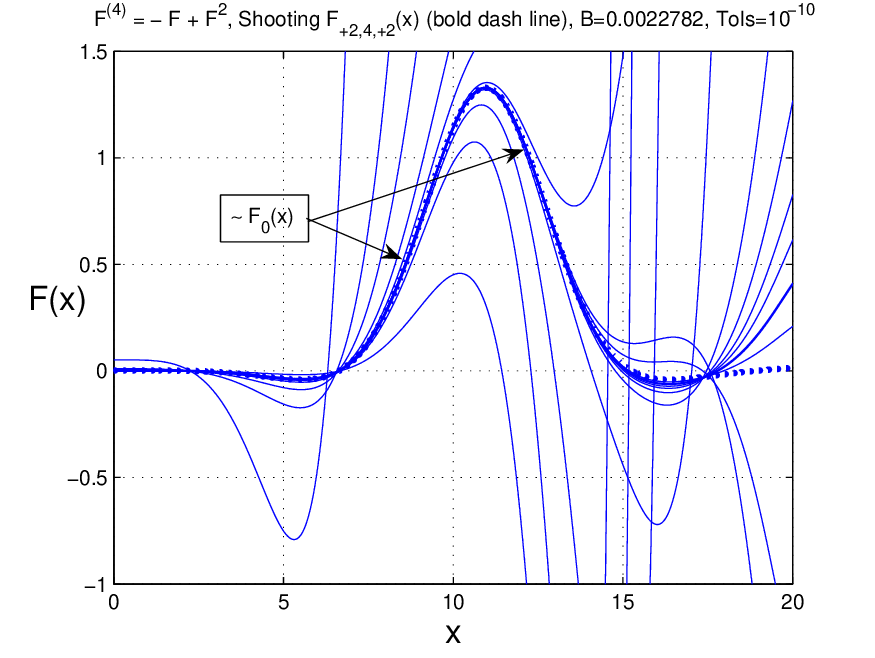}
\caption{Patterns $F_0$ and overall $F_{+2,4,+2}$.}\label{Anal4}
\end{center}
\end{figure}

In general, the {\it  ode45} solver is quite tricky to use for such
4th-order nonlinear ODEs, though it successfully  confirms the
previous results.

\subsection{Critical values of the functional by numerics}

The equalities  \ef{F22} allow us to understand numerically which
solutions of
 \ef{N1.1} deliver the main (extremal) critical value.
For convenience, in Fig. \ref{Crit23}, we present again the
profiles, for which we indicate their critical values.

First of all using notations from \ef{F22} we have that the first critical values corresponding to $v_0(x)=c_0 F_0(x)$
of $\Phi(r_0(v)v)$ 
and $F_0(x)$ of $\Phi(F)$ are extremal in the following sense:
 \be
 \label{c010}
 \fbox{$
 v_0: \quad c_0=0.4154... \, \mbox{is maximal}, \,\,\,\, F_0: \quad C_0=\frac 1{6c_0^2}=0.9659...\,\,\,\mbox{is minimal positive}. 
 $}
 \ee
This is important for us and serves as a first (numerical)
confirmation that $F_0(x)$ corresponds to the absolute critical
value of the functional \ef{F01}, \ef{F02}. It turns out that it
is even and  has the simplest geometric shape among other
patterns.
This is confirmed by symmetrization and level sets rearrangements
 for fourth-order problems though a full justification for nonlinear problems is not straightforward even in 1D; see
 \cite[Ch.~3]{GGS10} for details.
 
 Thus, according to \ef{F22} we will try to partially order our possible patterns as critical points according to there critical values: for $k \ge 0$,
  \be  
  \label{CrV1}
  \begin{matrix}
  \quad v_k: \quad c_k=\Phi(r_0(v_k)v_k) \equiv \int v_k^3 \to 0^+\,\,\,\mbox{is decreasing},
  \\ 
  F_k: \quad C_k=\Phi(F_k)= \frac 1{6c_k^2} \to +\iy \,\,\,\mbox{is increasing}, \quad\,\,
  \end{matrix}
  \ee
where the value $C_*=0$ of the trivial critical point $F_*=0$ can be attributed to the second sequence   as the first element.
Overall, those sequences in \ef{Tr1} look like belonging to a standard L--S type family of critical points but as we mentioned  in the present case of indefinite operators we cannot find definite traces of the L--S category theory
(but an ``L--S-type sequences" exist!).  
 
Anyway,  we do not hesitate to state this crucial for us
conclusion as follows:

 \ssk

{\bf Numerical Theorem ${\mathbf F_0}$.} \,\,{\em The first even
basic pattern} $F_0(x)$ {\em in Fig.} \ref{f7-1} ({\em given by the maximum of $\int v^3$ in}
\ef{F09} {\em on $H_0$})  {\em delivers the
absolute minimal positive critical value $C_0>0$ to the functional} \ef{F01}, \ef{F02}.

 \ssk

Note that according to \ef{CrV1}
$$
\tex{
c_k^2= \frac 1{6 C_k}, \,\,\, \mbox{i.e., minimal for $\{C_k>0\}$} \LongA \mbox{maximal for $\{c_k\}$}.
}
$$

Let us describe critical values of other patterns proving  
this ``theorem" where for convenience instead of the subscript ``$k$" we use the actual index ``$\{\cdot\}$" of the pattern. Thus, consider
$$
\mbox{the decreasing sequence critical values $\{c_k\} \to 0^+$ of critical  points $\{v_k\}$.}
$$
 It follows from \ef{F22} that, for any gluing $F_{+2,l,+2}$
 for large $l$ (and not that large, see below),
  \be
  \label{co11}
  \tex{
F_{+2,l,+2}: \quad c_{(\cdot)} \approx \frac {c_0}{{\sqrt 2}}
\approx 0.2938... \, .
 }
 \ee
 For a triple gluing with $l, \, m$ large
  \be
  \label{co12}
  \tex{
F_{+2,l,+2,m,+2}: \quad c_{(\cdot)} \approx \frac {c_0}{\sqrt 3}
\approx 0.2399... \, ,
 }
 \ee
etc. Other profiles also have smaller critical values, so cannot
deliver the absolute minimum (maximum) of the functional:
 \be
 \label{co13}
 F_1=F_{+2,2,+2}: \quad c_1=0.2941... \,,
 \ee
 \be
 \label{co14}
 F_{+2,4,+2}: \quad c_{+2,4,+2} =0.2937... \, ,
 \ee
 \be
 \label{c015}
 F_{+4}: \quad c_{+4}=0.2858... \, ,
 \ee
Observe that the critical values in \ef{co13} and \ef{co14}
 are very close to that in \ef{co11} indicating  that both belong
 to the family of $F_0-F_0$ gluings. In particular, in agreement  with
 \ef{co12},
 we have (see Fig. \ref{F36})
  \be
  \label{F333}
  F_2=F_{+2,2,+2,2,+2}: \quad c_{(\cdot)}=0.2403...\, .
  \ee

On the other hand, for $F_{+4}$ in \ef{c015}, this is not true, so
$F_{+4}$ belongs to another family $\{F_{+2l}, \, l=2,3,...\}$.
For instance (see Fig. \ref{F68})
 \be
 \label{co16}
 c_{+6}=0.2313... \, , \quad c_{+8}=0.1994... \, .
 \ee

\begin{figure}[htbp]
 \hfill   \hfill  \begin{minipage}[t]{0.45\textwidth} 
    \centering 
    \includegraphics[width=6cm,keepaspectratio]{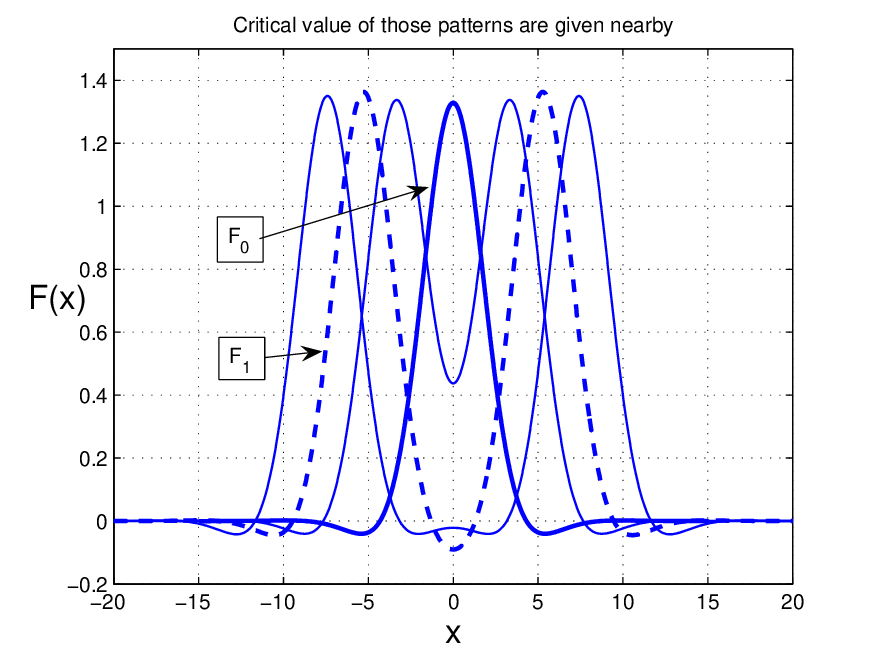}
   \caption{Solutions of \ef{N1.1} with critical values given above.}
    \label{Crit23}
\end{minipage}
  \hfill   \hfill 
\begin{minipage}[t]{0.5\textwidth}
    \centering 
    \includegraphics[width=6cm,keepaspectratio]{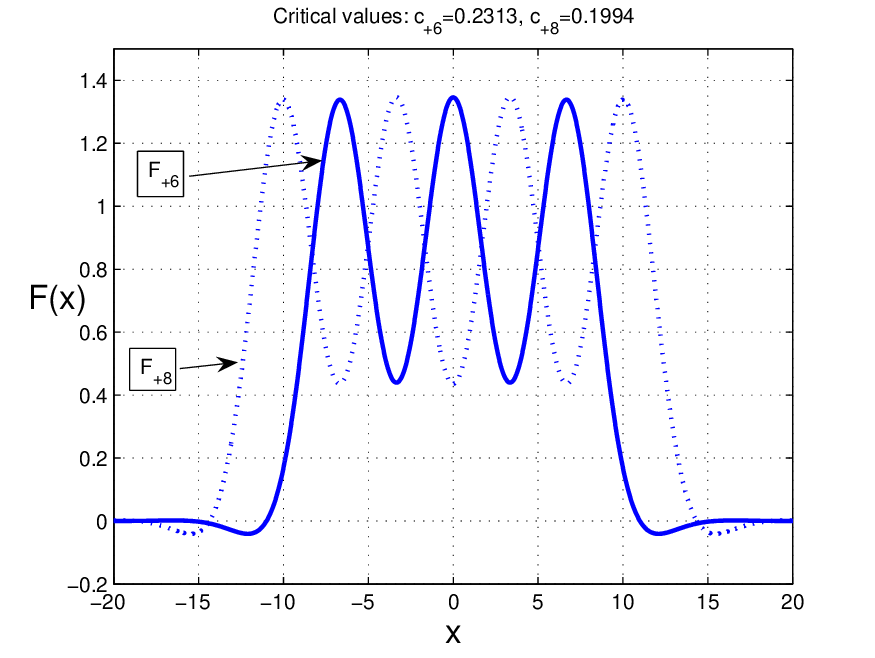}
    \caption{Solutions $F_{+6}$ and $F_{+8}$ of \ef{N1.1} with
critical values \ef{co16}.}
    \label{F68}
\end{minipage}
\end{figure}

\begin{figure}[htbp]
\begin{center}
\includegraphics[scale=0.42]{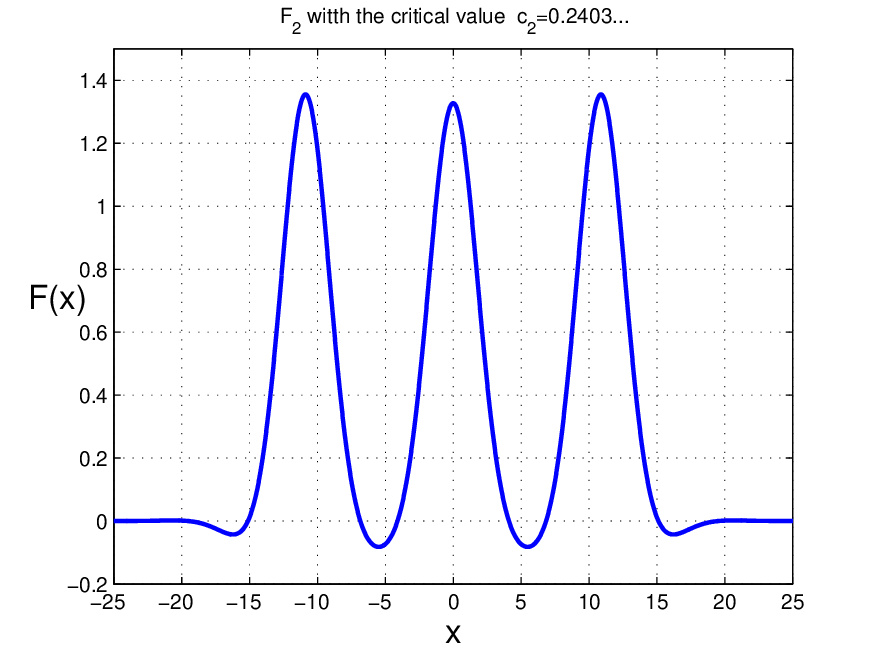}
\caption{The pattern $F_2=F_{+2,2,+2,2+2}$ of \ef{N1.1} with
critical values \ef{F333}.}
    \label{F36}
\end{center}
\end{figure}

\section{On countable subsets of patterns: first elementary
 arguments on a gluing/matching procedure}
\label{S2N}

After dealing with various fourth-order nonlinear ODEs and
obtaining some numerical evidence of what can be or cannot be with
their typical solutions, we now return to our basic model
\ef{N1.1} with the variational 
semilinear elliptic problem but not definite (with non-odd nonlinearities).
We recall again that here, in view of our next plans concerning other equations,  we cannot use any results of the Hamiltonian DSs theory which for such analytic ODEs can 
produce a lot of deep results and conclusions. We again refer to papers \cite{BCT96}--\cite{ChK04} where the most closed Hamiltonian (sometimes with asymptotic perturbations) were studied.

\subsection{The first basic pattern family ${\mathcal F}_1$} 

One of our first goals is to show that  \ef{N1.1} admits a
countable basic family of
patterns, which we have denoted above  by
 \be
 \label{LS12}
 {\mathcal F}_1=\{F_k(x)\}_{\{k \ge 0\}}.
  \ee
  On one hand, 
 each $F_k$ is produced by the simplest and the most ``dense" way of gluing
$k$ elementary first variational pattern $F_0$'s by properly
distribute them over the $x$-axis. On the other hand, $F_k$ represents a $k$-long piece of humps from the periodic orbit $\Gamma_{\rm max}$ where two border humps are glued to exponential tails going to $\pm \iy$.

\subsection{An $F_0-F_0$ asymptotic gluing: the universal expansion constant ${\mathbf C_1^0}$}

We begin with a preliminary
explanation how to match  by shifting in $x$
two patterns 
 $F_0$ via gluing their almost linear exponential oscillatory
 tails.
For convenience, we fix the simplest (and special, see general
setting of gluing below) case
  \be
  \label{f21}
 \mbox{$F_0-F_0$ gluings},
  \ee
  where we show how to  ``glue", in a symmetric even way, two basic
  variational patterns $F_0(x)$, which are shifted in space in such a way to make further
   ``nonlinear matching" possible.
  
Thus,  we are looking for a new 
pattern
approximately satisfying
 \be
 \label{mm1}
 F_\sigma(x) \approx F_0(x+a_n)+F_0(x-a_n), \quad \s=\{+2,l_n,+2\},
  \ee
  where $a_n \gg 1$ is a special sufficiently large shifting  parameter to be determined from an elementary
  algebraic expression to be presented.
In other words, we plan to do the following:
 \be
 \label{mm10}
 \mbox{to match two patterns $F_0(x+a_n)$ and $F_0(-x+a_n)$ at
 the origin $x=0$},
 \ee
 which requires specially  designed symmetry conditions at $x=0$
  to be studied below in detail. 

In the index $\sigma$ in
  \ef{mm1}, as usual, the first $+2$ represents two intersections
  with the steady  state $F_* \equiv 1$ in \ef{N1.1} of the first
  patterns $F_0(x+a_n)$ on the left-hand side  and the last describes the same number for
  $F_0(x-a_n)$, while the intermediate number measures, roughly
  speaking, the total number of zeros (and, sometimes, nearby
  $F_0$'s, extra minmax points) in between. 
  By the known
 oscillatory behaviour of decreasing exponential tails in
  \ef{full1} there holds
   \be
   \label{lk1}
   l_n \sim  n \forA  n \gg   1.
 \ee
 Due to the same asymptotic reasons,
 those shifting parameters
  can take only a discrete set of values $\{a_n\}_{n
  \ge 1}$. We then expect that, according to the half-period of
  oscillations in the tails \ef{full1} governed by $\cos( x/{\sqrt
  2})$,
   \be
   \label{mm2}
   a_n \sim \pi \sqrt{2}n \to +\iy \asA n \to \iy,
   \ee
   i.e., at $n= \iy$, we observe  no interaction between those
   shifted $F_0$-patterns
     at all, so they become independent, and,
   according to \ef{mm1} just move to $\pm \iy$ and eventually
   disappear at $\pm \iy$.

To make a further  approximation of  the equality \ef{mm1}, we use
the structure of the linearized stable (relative to the steady state $F_*
\equiv 0$) exponential tail in  \ef{full1} denoting it by
 \be
 \label{exp1}
  \tex{
F_0(x) \approx  W_0(x;C_1^0,C_2^0) = {\mathrm e}^{-\mu x}
\big[ C_1^0 \cos(\mu x)
+ C_2^0 \sin(\mu x)\big], \,\,\, x \gg 1.
 }
 \ee
Rewriting for convenience \ef{exp1} in the equivalent form
 \be
 \label{exp2}
  \tex{
F_0(x) \approx   W_0(x;C_1^0,C_2^0) =  C_1^0 {\mathrm e}^{-\mu x} \, \cos(\mu x + C_2^0), \,\,\,
x \gg 1,
 }
 \ee
Moreover, in view of the translational invariance of \ef{N1.1}, by
shifting in $x$, we always can make
 \be
 \label{CCC2}
 C_2^0=0,
 \ee
so that the final one-parametric expansion to be used later on
becomes
 \be
 \label{exp2N1}
  \tex{
F_0(x) \approx   W_0(x;{\mathbf C_1^0}) = {\mathbf C_1^0} {\mathrm
e}^{- \mu x} \, \cos( \mu x), \,\,\,
x \gg 1.
 }
  \ee
 This bold faced expansion
 constant ${\mathbf C_1^0}$ is an important characteristic of the basic pattern
  $F_0(x)$.
 
Bearing in mind the known regularity/analyticity properties of
uniformly bounded and decaying to zero solutions of \ef{N1.1}, we
have a similar asymptotically correct expression of the derivative
in the tail:
  \be
  \label{exp3}
  \tex{
  F_0'(x) \approx W_0'(x;
  {\mathbf C_1^0})=
  - \mu  \,
{\mathrm e}^{-\mu x} \,
  {\mathbf C_1^0} \big[ \sin(\mu x) +
 \cos(\mu x) \big].
 }
 \ee
 Therefore, since we want to reflect a properly perturbed
 structure $F_0(x+a_n)$ with respect to $x=0$ and to get such an
 $F_0-F_0$ gluing, we have to have the first symmetry condition to be
 approximately valid
  \be
  \label{SD1}
 F_0'(a_n) \approx W_0'(a_n;{\mathbf C_1^0})=0.
 \ee
 Indeed, by \ef{exp2N1} and  \ef{exp3}, we also achieve two  further
 matching conditions for these two glued/matched patterns
 \be
 \label{SD2}
 F_0(x+a_n) \approx W_0(x+a_n;{\mathbf C_1^0}) \andA
F_0(-x+a_n) \approx W_0(-x+a_n;{\mathbf C_1^0}).
 \ee
 Namely, by construction, we have  the asymptotically sharp
 tails $W_0(\pm x +a_n;{\mathbf C_1^0})$ as approximating  patterns $F_0(\pm
 x+a_n)$:
 \be
 \label{SD3}
 \begin{split}
\mbox{at $x=0$:} \,\,\,W_0(x+a_n;{\mathbf
C_1^0}) & =W_0(-x+a_n;{\mathbf C_1^0}), \\ 
   W_0'(x+a_n;{\mathbf
C_1^0}) & =W_0'(-x+a_n;{\mathbf C_1^0})=0,
 \\
    W_0''(x+a_n;{\mathbf
C_1^0}) & =W_0''(-x+a_n;{\mathbf C_1^0}).
 \end{split}
 \ee
Thus, the only problem is that we {\em principally} cannot match
the last remained symmetry condition since
 \be
 \label{SD4}
\mbox{at $x=0$:} \,\,\,W_0'''(x+a_n;{\mathbf C_1^0}) \not
=W_0'''(-x+a_n;{\mathbf C_1^0}).
 \ee
 Actually those third derivatives $W_0'''(\pm x+a_n)$ get the
 opposite non-zero values
 \be
 \label{SD77}
\mbox{at $x=0$:} \,\,\,W_0'''(x+a_n;{\mathbf C_1^0})= -
W_0'''(-x+a_n;{\mathbf C_1^0}) \not =0.
 \ee
  Obviously, there holds:
  \be
  \label{SD78}
  \begin{matrix}
  \mbox{for the linear ODE $F^{(4)}=-F$, $F(x) \not \equiv 0$, all four matching}
  \\
  \mbox{conditions
  cannot be valid by uniqueness.}
 \end{matrix}
  \ee

Indeed, in this linear construction only two  from all  four solutions of the fundamental system of $D^4_x+I$
are involved. Clearly,  dealing with all four small solutions a similar matching can be performed  using the corresponding linear tail presentation for $x \approx 0$ with general matching conditions at $x=0$: e.g., 
 \be
 \label{PR12}
 \begin{matrix}
 F_\s(x) \approx C_1 {\rm e}^{-\mu(x+a_n)}\cos(\mu(x+a_n))+ C_2 {\rm e}^{\mu(x-a_n)}\cos(\mu(x-a_n)), \,\,\,|C_{1,2}| \ll 1, 
 \,\,\,a_n \gg 1,
 \\
 \mbox{where} \quad [F_\s]=[F_\s']=[F_\s'']=[F_\s''']=0 \atA x=0.
 \end{matrix}
 \ee
 Observe that here $F(x) \sim O({\rm e}^{-\a_n})$ for $x \sim 0$, $C_{1,2}$ are arbitrary small, and the nonlinear term $F^2$ can supply a perturbation therein of the order $\sim O({\rm e}^{-2\a_n})$.

Therefore, a condition like \ef{SD4} for the above {symmetric} matching $F_0(\pm
x+a_n)$'s is the only one which cannot be achieved approximately,
so it can be satisfied under the presence of the nonlinear term
$+F^2$ in \ef{N1.1} and was observed in many reliable numerical
tests. Of course such a matching procedure is not a local one in a
neighborhood of $x=0$, but concerns a global structure  of a
pattern in a whole $\re$. It seems it cannot be solved by any
fixed
 point (or similar) arguments in view of a presence of a truly 2D
 unstable manifold around those patterns $F_0(\pm x+a_n)$
 involved. By the same reason, a proper application of various nonlinear analysis functional
 techniques to solve that problem is not  straightforward.

From our approximating condition \ef{SD1}
by using \ef{exp3} we obtain a simple (in this particular case) algebraic
equation for admissible values of  matching parameters $\{a_n\}$:
 \be
 \label{exp4}
 \tex{
 \tan \big( \frac{a_n}{{\sqrt 2}}\big) =-1,
 }
 \ee
 whence the following first approximation  of the shift parameters  (cf. \ef{mm2}):
  \be
  \label{exp5}
  \tex{
  a_n \approx  - {\sqrt 2} \big( \frac \pi 4 \big)
   + \pi {\sqrt 2} \, n \quad \mbox{for } \,\,\, n \gg 1.
 }
   \ee
  We have seen  such a phenomenon in various
 numerics for such not that large $n$'s, for which the
 exponentially small tails are sufficiently visible in practice.

   Therefore, the even matching is then performed at $x=0$, where we have
   impose the usual symmetry conditions \ef{bc1}, i.e., assuming such a perturbation
   of a true solution $F_\s(x)$ by shifting in $x$, such that in the
   asymptotic tail area (now, after $a_n$-shifting, at $x=0$)
   there hold:
    \be
    \label{kk1}
    F_\s'(0) = F_\s'''(0)=0,
    \ee
    together with the continuity of other values and derivatives.
Solving the problem \ef{N1.1}, \ef{kk1} for $\{a_n\}$
would lead to 
a countable number of new
solutions, see below.

Consider the approximation \ef{mm1}.
 For convenience, in view
of the even symmetry of such a pattern, we perform the shifting $x+a_n
\mapsto x$ (so that the matching point $x=0$ for a perturbed first
structure $F_0(x+a_n)$ is now at $x=a_n$), and consider a single
pattern $F_0(x)$ for $1 \gg x \le a_n$, with still unknown $a_n
\gg 1$, being a free parameter, at which there must hold, for the
sake of a further reflection,
 \be
 \label{kk98}
 F'_0(a_n)=0.
 \ee
 We thus require just a single ``linearized" condition for
 $F_0(x)$ in \ef{kk98} instead of two in \ef{kk1} ending up
 strictly this part of the problem.
 This allows us to reflect the pattern $F_0(x)$ relative to $x=0$, to get the
first approximation towards the required $F_\s$ gluing pattern. At
a further stage a nonlinear interaction (via the quadratic $F^2$-term), 
with the two shifted patterns $F_0$ involved, is required to create
the desired $F_\s$.

\smallskip

Overall, in vies of such a symmetric tails matching and \ef{PR12}
we can expect that 
a similar construction is available for many pairs of more arbitrary
patterns $F_\mu(x)$, $F_\nu(x), ...\,$, which should be shifted
 sufficiently from each other, to reveal their linearized approximately
 exponential tails, with different values of the expansion
 coefficient $C_1^\mu,\, C_2^\mu,..., \, C_1^\nu, C_2^\nu,...$ \, , leading
 to a discrete countable set of the appropriate shift-parameters
 $a_n^{\mu, \nu}$ for each successive pair of simpler patterns.
 Then we arrive at a situation when  simple higher-order ODE problems like
\ef{N1.1} and many others, semilinear or quasilinear, can admit an incredible countable set of
 possible solutions, which are ``localized" in $x$-space in the sense that all of
 them
  vanish, as $x \to \iy$, exponentially fast. For quasilinear degenerate equations this means  precisely {\em localized}, i.e., with finite interfaces and supports. E.g., this is true in the case of the leading thin-film equation: 
  $$ 
  (|F|F''')'=-F+F^4 \inB \re \LongA
  {\rm meas}\, {\rm supp} \,F(x)< \iy.
  $$

After a natural ``geometric" closure of this solution
subset, we can then claim that this subset is actually {\em chaotic}:
 omitting the boundary condition at infinity, $F(\iy)=0$,
 \be
 \label{exp6}
 \mbox{
 such solutions  $F_\s(x)$ in $\re$ exist for infinitely many 
 non-periodic indexes  $\s$},
  \ee
  with our not-that-strict and formal definition of the index $\s$
  explained above
  and involving when necessary the number of zero and/or minmax point
  in between the leading $F_k$-structures. We will return to this question in this section later on when discussing the geometric structure of the attractor $W^{2,\infty}$ composed from the pattern homoclinics in $\re^4$.
Another precise meaning of  chaotic orbits for fourth-order ODEs
with different coercive operators and non-oscillatory  tails of
patterns
 was used  in \cite[p.~198]{PelTroy}.

\subsection{A geometric justification  of multiple patterns via
 the matching/gluing procedure}

In general, we claim that existence of  proper oscillatory decreasing tails of solutions  at infinity having a 2D stable
manifold (a saddle node or $(2,2)$ defect indices on $\re_\pm$ in the linear case) is absolutely necessary and crucial for existence of a countable subset of patterns. As we have seen such oscillatory tails can be observed in various and quite unusual 4th-order models where as $x \to +\iy$, $F(x) \to 0$,
 \be  
 \label{AA1}
     F^{(4)}=-F^3 + O(F^4) \LongA F(x) \sim x^{-2} \phi_*(s),\,\,\, s=\ln x,
  \ee
   where   $\phi_*(s)$ is a periodic  oscillatory component satisfying some autonomous ODE, and, say,  
 \be 
 \label{AA2}   
     ((F'')^3)''=-F^3 +O(F^4) \LongA
     F(x) \sim {\rm e}^{ax}, \,\,\, 9a^8=-1, 
\ee
where $a$ has the maximal ${\rm Re}\, a_k<0$ among all the eight roots $\{a_k\}$ of this characteristic equation.
  This list of difficult and not fully proved nonlinear asymptotic expansions leading to special kinds of ``homoclinics" can be extended but computations become cumbersome.

Using our convenient canonic model \ef{N1.1} with the simplest explicit exponential tails  and the above experience, we can complete our study of the existence of a countable subset of patterns. 
It consists of a few steps:

\begin{enumerate}
\item[{\bf (i)}] A nontrivial variational pattern denoted here by $F_0(x)$ exists being the critical point  of the absolute extremum of the functional.

\item[{\bf (ii)}] $F_0(x)$ has the known asymptotic tails (\ref{inf1N}) as $x \to \pm \iy$ and hence the behaviour (\ref{OneP1}) for some $B=B_0 \not = 0$. We denote local solutions for a given $B$ by $\hat F_B(x)$. Thus, with $\mu = \frac 1{\sqrt 2}$:
 \be 
 \label{Pr1}
 \begin{matrix} 
 \hat F_0(x)= B_0 {\rm e}^{\mu x}\cos(\mu x)+ O({\rm e}^{2 \mu x}), \quad x \to -\infty,
 \\
  \hat F_0(x)= C_0 {\rm e}^{-\mu x}\cos(\mu x+ \varphi_0)+ O({\rm e}^{-2 \mu x}), \quad x \to +\infty,
   \end{matrix}
   \end{equation}
 where $C_0$, $\varphi_0$ are some constants, i.e., the last line means that  the asymptotics of $\hat F_0$ belongs to a true 2D stable subspace shown in (\ref{inf1N}). Without loss of generality we assume that
  $
  C_0>0.
  $
 
\item[{\bf (iii)}] According to the strategy formally discussed above in detail,  we have to find some $B \approx B_0$ for which
  \be 
  \label{Pr2}
  \exists \,\, x_B \gg 1: \quad \hat F'(x_B)=F'''(x_B)=0,
  \ee
so that by reflection about $x=x_B$ we get a new pattern by gluing $\hat F_B(x)$ and $\hat F_B(2 x_b-x)$ which is even relative to $x=x_B$. Using  (\ref{Pr1}) for $x \gg 1$ we obtain:
 \be 
 \label{Pr3} \begin{array}{l}
 \hat F_0'(x)= - \mu  C_0 {\rm e}^{-\mu x}\big[\cos(\mu x+ \varphi_0)+ \sin(\mu x+\varphi_0)\big]
 + O({\rm e}^{-2 \mu x})=0
 \\
 \hat F_0''(x)= 2\mu^2  C_0 {\rm e}^{-\mu x}\sin(\mu x+\varphi_0)
 + O({\rm e}^{-2 \mu x}),
 \\
 \hat F_0'''(x)= 2 \mu^3  C_0 {\rm e}^{-\mu x}\big[-\sin(\mu x+ \varphi_0)+ \cos (\mu x+\varphi_0)\big]
 + O({\rm e}^{-2 \mu x}).
 \end{array}
 \ee
 such that from the first equality we deduce that $\cos(\mu x+ \varphi_0)+ \sin(\mu x+\varphi_0) \approx 0$ and, then, $\mu x + \varphi_0 \approx  - \frac {\pi}4 + \pi k$, with $k \gg 1$. 
It follows from the first and the last lines in (\ref{Pr3}) that for all large $k$
 \be 
 \label{Pr4}
 \begin{matrix} 
 \mbox{at  maxima} \,\,
 x_k^+ + \varphi_0 \approx - \frac {\pi}4 + 2\pi k: \,\,\,\hat F_0'=0, \,\,\,\hat F_0''' = C_0 {\rm e}^{-\mu x_k^+}
 +O({\rm e}^{-2\mu x_k^+}) >0,
 \\
   \mbox{at  minima} \,\,
 x_k^- + \varphi_0 \approx  \frac {3\pi}4 + 2\pi k: \,\,\,\hat F_0'=0, \,\,\,\hat F_0''' = - C_0 {\rm e}^{-\mu x_k^-}
 +O({\rm e}^{-2\mu x_k^-}) < 0.
 \end{matrix}
  \ee

\item[{\bf (iv)}] Since any solution including $F_0$ ((\ref{OneP1}) prevents translation in $x$)
 is isolated by the analyticity (actually, this is not necessary, a sufficiently smoothness is enough),
 patterns do not exist for small $|B-B_0|>0$. Hence for any $B \approx B_0$, $B \not = B_0$ the functions $F_B(x)$ is not a homoclinic so that $F_B(x)$ is large enough for $x \gg 1$, 
 in the sense that it cannot be uniformly close to zero therein due to appearing unstable modes $\sim {\rm e}^{+\mu x}$ in (\ref{full1}). By the continuity in $B$
  \be
  \label{Pr11}
  \hat F_B(x) \to F_0(x) \,\,\,\mbox{as $B \to B_0$ uniformly in any interval $(- \iy, L]$ for any $L \in \re$}.
   \ee

 \item[{\bf (v)}] Then either such $F_B(x)$ blows up at a finite $x \to x_B^-< +\infty$ (see below how it happens and it does if $F_B > \frac 32$)) or remain sufficiently positive or negative large for $x \gg 1$ ($x <x_B$ in case of blow-up) simply meaning that the linearized asymptotic tail of $F_B(x)$ as $x \to +\iy$ must be destroyed for any $B \not = B_0$
  being arbitrarily closed to $B_0$. We do not specify more details on such a matter, but note that finally 
  due to the above geometric min/max properties of $ \hat F_B(x)$, such an essential destruction of the tail for $x \gg 1$ will inevitably lead to the appearance of new more complicated patterns obtained just by reflection over symmetry points.

\item[{\bf (vi)}] Therefore, changing $B \approx B_0$ we inevitably find an interval for $x \gg 1$ such that
 \begin{itemize}
 \item either  $\hat F_B(x) \gg  \sup |F_0(x)|>0$ or;
 
\item on the contrary $\hat F_B(x) \ll - \inf |F_0(x)|<0$.
  \end{itemize}
  In both cases such a $B$-evolution destroying the min/max point in (\ref{Pr4}) would lead 
 to an inflection point where
  \be 
  \label{Pr5} \hat F'_B= \hat F_B''=0 \,\,\, \mbox{and}: \quad \mbox{\rm  (a)}\,\, \hat F_B'''
   \ge 0, \quad \mbox{\rm  (b)}\,\,\hat F_B'''\le 0.
    \ee 
These correspond to a generic disappearance of:

  \begin{itemize}
 \item a positive maximum moving down  to sufficient negative values, when such a deformation goes from the right-hand side while the left-hand one remains  more untouched by the continuity in $B$, i.e., by (\ref{Pr11});
 
 \item  a negative  minimum moving up  to sufficient positive values again due to (\ref{Pr11}).
  \end{itemize}
\item[{\bf (vii)}] Finally, we observe that {\rm  (a)} corresponds to a $B$-evolution of a minimum point with $F_0''' < 0$ by 
  (\ref{Pr4}) to $\hat F_B''' \ge 0$ at this inflection in (\ref{Pr5}) and vice versa for {\rm  (b)}:  a maximum point 
  with $F_0''' > 0$ by 
  (\ref{Pr4}) to $\hat F_B''' \le 0$ at this inflection in (\ref{Pr5}).
  
  Therefore, there exists a $B \approx B_0$ (actually, $B$ can be arbitrarily close to $B_0$) such that (\ref{Pr2}) holds. Moreover, any points $x_k^\pm$ in (\ref{Pr4}) for large $k$
  can be used in such a manner, i.e., there exists an infinite number of different gluings provided those min/max points
  with sufficient accuracy belong to a true asymptotic tail.
  \end{enumerate}
  
  \smallskip
  
  {\bf Remark 1: oscillatory sign-changing asymptotic tails are crucial.}
  The above analysis shows that the key ingredient of the construction of new patterns is:
  \be 
  \label{Pr13}
 \exists F_0\,\, \mbox{and asymptotic tails are oscillatory} \Longrightarrow \exists \,\, \mbox{infinitely many patterns.}
  \ee
 This explains why everyone can find a lot of patterns in practically any reasonable nonlinear higher-order ODEs of a Cahn-Hilliard type including those we study here.
  
  \smallskip

  {\bf Remark 2: {\em a priori} existence of an $F_0$ is not necessary.} Indeed, the same geometric argument as in {\bf (vi)} and {\bf (vii)} applies if we consider a semi-orbit satisfying \ef{Pr1} as $x \to -\iy$ with some $B_0$ which blows up (a generic behaviour if, by the assumption,  no homoclinics exist). Then changing $B_0$ will inevitably lead to a similar deformation of some min/max points to a symmetry point, i.e., to the birth of a pattern.

\subsection{On a general principle of matching/gluing}

We have discussed a particular case of matching based on the following geometric observation:
 \be 
 \label{Infl1}
 \begin{matrix}
 \mbox{min/max points with $F'=0,\,\,F'''>0$ or $<0$} \Longrightarrow
 \\
  \mbox{inflection points with
 $F'=F''=0, \,\,F''' \le 0$ or $ \ge 0$}
 \end{matrix}
 \ee
 implying that along such a local $B$-evolution there appears a point $(B,x_B)$ at which $F'=F'''=0$ so the reflection
 at $x=x_B$ yields a desired (even) patterns in $\re$.
 
 It is easy to a state the most general matching/gluing conditions, which in general become useless without deep involvement 
 of other preliminary properties of the flow. Those are: we shoot from both sides  with two solutions
  \be 
  \label{Infl2}
  \hat F_{B_\pm}(x), \quad \hat F_{B_\pm}(x) = F_{B_\pm}(x)+... \equiv B_\pm {\rm e}^{\pm  \mu x} \cos(\mu x) +...
\,\,\,\mbox{as} \,\,\,x \to \pm \iy,
 \ee 
 and we need to solve the following matching system including two point of matching $x_\pm \in \re$:
 \be
 \label{Infl3}
(B_\pm, x_\pm): \quad \hat  F^{(l)}_{B_-}(x_-)= \hat F_{B_+}^{(l)}(x_+), \,\,\,l=0,1,2,3.
 \ee
This gives four equations with four unknowns. For any solution, the resulting patterns is given by translation
 \be 
 \label{Infl4}
 F(x)=
 \left\{
 \begin{matrix}
 \hat F_{B_-}(x), \,\,\, x \le x_-, \qquad \qquad  \quad 
 \\
 \hat F_{B_+}(x_+  - x_- + x), \,\,\, x \ge x_-.
 \end{matrix}
 \right.
  \ee
According to our experience partially presented in this paper, in most of the cases this system admits  a countable subset of solutions provided that the Fr\'echet derivative at $F=0$ of the nonlinear operators has a proper  defect
index $(2,2)$ (and $(m,m)$ for the leading operator $(-D_x^2)^{m}$).

It is seen that in above analysis we have mainly concentrated on the cases where we get an even pattern (\ref{Infl4}) relative to $x=x_-$ via a geometric scrutiny as in (\ref{Infl1}). Indeed, via that construction, the system (\ref{Infl4}) is reduced to a single equation in the present variable:
 $$
 \hat F_{B_-}'''(x_-)=0 \,\,\,(\mbox{the continuity of $F,F'=0,F''$ is guaranteed}),
 $$ 
which can be reasonably analyzed.

\subsection{Some  comments on
 the corresponding parabolic gradient system}

Here we try to explore another parabolic evolution tool to deal with the matching approach.
Denote by $\hat F_0(x)$ the following even function, which we have
dealt with asymptotically to predict those $x$-shifting for a
possible gluing:
 \be
 \label{evol1}
 \hat F_0(x) = \{F_0(a_n+x) \forA x \le 0; \quad F_0(a_n-x) \forA x
 \ge 0\},
 \ee
 i.e., as usual, we reflect $F_0(a_n+x)$ relative the $y$-axis.
 Now the values of $a_n$ are derived by using the same principles as before, 
 at $x=0$:
\begin{enumerate}
 \item[(i)] $\hat F_0(x)$ is continuous, $\hat F_0(0^-)=\hat
 F_0(0^+)$,

\item[(ii)]  $\hat F_0'(x)$ is continuous, $\hat F_0'(0^-)=\hat
 F_0'(0^+)=0$,

\item[(iii)] $\hat F_0''(x)$ is continuous, $\hat F_0''(0^-)=\hat
 F_0''(0^+)$ (reflection gives an even extension),

 \item[(iv)]  $\hat F_0'''(x)$ is discontinuous, $\hat F_0'''(0^-)=-\hat
 F_0'''(0^+)=d_3(n) = F_0'''(a_n) \not = 0$.
\end{enumerate}
Note that, in our construction (iv) is inevitable, otherwise the
continuity of $\hat F_0'''(x)$ at $x=0$ would contradicts the
uniqueness for the Cauchy problem: with given data at $x=0$, there
exist two solutions for $x>0$: $F_0(x)$ and $F_0(a_n-x)$.

We next consider the Cauchy problem for the corresponding
semilinear parabolic equation for a function $w=w(x,t)$:
 \be
 \label{evol2}
 w_\t = {\mathcal L}(w) \equiv -w_{xxxx}-w+w^2 \inB \re \times \re_+, \quad w(x,0)= \hat
 F_0(x) \inB \re.
 \ee
 By the classical parabolic theory, \ef{evol2} has a unique local
 in time
 classical analytic (for $t>0$) solution. 
 
 Since this PDE is a gradient system
 (the operator is variational), we can
  multiply \ef{evol2}
 by $w_\t$ in $L^2(\re)$ to get the required monotonicity
  \be
  \label{evol3}
  \tex{
  \frac{\mathrm d}{{\mathrm d}\t} \Phi(w)(\t) = - \|w_\t(\t)\|_2^2 \le
  0 \forA t >0,
  }
  \ee
  where $\Phi(w)$ is the functional \ef{F01}.
 Moreover, under the assumption
  \be
  \label{evol4}
  \mbox{the CP has a global uniformly bounded solution}
  \ee
  (we have to impose that since blow-up may occur, while proving
  non-blow-up is not straightforward),
 we obtain two  main gradient system conclusion:
\begin{enumerate}
 \item[(i)] $\Phi(w(\t)$ is monotone decreasing on evolution orbits, and

 \item[(ii)] integrating \ef{evol3} ensures the
  boundedness of $w_\t$ in $L^2(\re \times (1,+\iy))$:
 \be
 \label{evol5}
 \tex{
 \int_1^\iy \|w_\t(\t)\|_2^2 \,{\mathrm d}\t < \iy.
 }
 \ee
 \end{enumerate}
 Then, the omega limit set of the orbit (all partial limits along
 any
 sequence $\{\t_n\} \to +\iy$) satisfies:
  \be
  \label{evol6}
  \mbox{$\omega(\hat F_0)$ consists of stationary solutions,
  $\{F(x)\}$}
 \ee
 and is also compact and connected. In other words, there exists a
 solution of \ef{N1.1} and a sequence $\{\t_n\} \to \iy$ such that
  \be
  \label{evol7}
  w(\cdot,\t_n) \to F(x) \,\,\mbox{uniformly in $\re$} \asA n \to
  \iy.
  \ee
  Thus, here  $F$ is a desired stationary solution.
  We expect that such an $F(x)$ must stay in a small neighbourhood of
  the initial data $\hat F_0(x)$ at least for large $n \gg1$, since the only ``defect"
of the initial data \ef{evol1} is expressed by a small discrepancy of $\hat F_0'''(0)$: 
$$
d_3(n)=F_0'''(a_n) \to 0 \asA n \to \iy.
$$  
However, we cannot complete such a proof, which seems require an extra subtle estimate on solutions
for $\t \gg 1$. A general and quite common estimates like \ef{evol5} are not sufficient.

On the other hand, any unstable mode generated by such a perturbed
stationary profile $\hat F_0(x)$ could lead again to blow-up in
\ef{evol2} in view of the superlinear combustion-like  term $u^2$
according to the basic reaction-diffusion model
 \be
 \label{evol8}
 \tex{
 u_t=-u_{xxxx} + u^2, \quad \sup_x|u(x,t)| \to +\iy \asA t \to
 T^-
 < \iy,
 }
 \ee
 i.e. even the boundedness \ef{evol4} is under a threat. We will
 pay some attention to blow-up in \ef{evol8} later on in this section.

\ssk

 Let us again point out that we observe a measure in
the action of our operator in \ef{evol2} on the data in
\ef{evol1}:
 \be
 \label{evol10}
 {\mathcal L}(\hat F_0)(x)= -2d_3(n) \d(x) \quad \hbox{in}\quad \re,
 \ee
 where by construction on the exponential tail,
  \be
  \label{evol11}
  d_3=d_3(n)= {\mathcal O}({\rm e}^{-\mu a_n}) \to 0 \asA a_n \to \iy.
 \ee
 Hence, the measure in \ef{evol10} can be made arbitrarily small. 
This allows an extra control of the time derivative  $v=w_\t$. It
follows from \ef{evol10} that, since the initial function in
\ef{evol2} is ``almost" stationary except at $x=0$, the only
singularity of $w_\t(x,\t)$ for $\t \approx 0^+$ occurs at the
origin. To show that it disappears  quickly,
 consider a parabolic equation for $w_\t$
by differentiating \ef{evol2} in $\t$:
 \be
 \label{vvt1}
 \tex{
  {v_\t}=-{v_{xxxx}} - v + 2w(x,\t)v \,\,\, \mbox{in} \,\,\, \re \times \re_+, \quad v(x,0)= -2d_3 \d(x)
  \inB \re.
 }
 \ee
 Recall that according to \ef{evol11}, the initial measure can be
 made arbitrarily small for $n \gg 1$.
 Since by the assumption $|w(x,\t)| \le C$, as an elementary scaling argument shows, the behaviour
 of the solution to \ef{vvt1} in a neighbourhood of $(x,\tau)=(0,0^+)$ is
 governed by the bi-harmonic operator in \ef{vvt1}, that gives the
 fundamental solution of  $D_t+D_x^4$:
  \be
  \label{Fu21}
  \tex{
  w_\t(x,\t)=v(x,\t) \approx -2d_3(a_n) \t^{- 1/4}b_0(y),
  \quad y=\frac x{\t^{ 1/4}},
  }
  \ee
  where $b_0(y)$ is the fundamental kernel. It follows that the
  singularity  disappears in time
   \be
   \label{Fh22}
    \tex{
    w_\t(0,\t_n) \sim 1 \LongA \t_k \sim d_3^4(a_n) \ll 1 \forA n
    \gg 1,
    }
    \ee
    and until then is concentrated on small intervals $\{|x| \le  C \t^{1/4}\}$.
Overall, this means that a weak singularity occurring at $(x,\tau)=(0,0^+)$
in the problem \ef{evol2} disappears in a sufficiently fast manner, 
leading to a standard problem on the $\omega$-limits for a
parabolic gradient system discussed above. However, we cannot justify 
that a required stationary solution eventually remains  very close
to $\hat F(x)$.

\subsection{A discussion around the countability of solutions}

Obviously, in view of the analytic nature of \ef{N1.1} and many other equations under consideration with 
 defect indices of linear operators $(2,2)$ in $\re_\pm$ ($(m,m)$ for proper $2m$th-order operators) the patterns subset cannot be more than countable. Note that in the nonlinear operator theory the fact that fixed points are always isolated 
does not require analyticity and a some regularity is usually enough.

As usual, in view of further applications to non-smooth and non-variational ODEs, we cannot rely on any results of the advanced Hamiltonian DSs theory.
However, for convenience, we summarize some preliminary discussions on this
subject and present some new comments. In fact, this conclusion is
a rather straightforward for analytic equations. Indeed, since by
the expansion \ef{full1} of arbitrary proper solutions it is
required by shooting to match a stable (as $x \to - \iy$) 2D
manifold in \ef{full1} with some nontrivial pair $\{C_3,C_4\}$ with
the stable (as $x \to +\iy$) 2D manifold in \ef{full1}, where
another nontrivial pair $\{C_1,C_2\}$ is required. Other unstable
2D manifolds must be vanished  as $x \to \pm \iy$ but of course play
a necessary role in the matching in between. In fact, the shooting
has the main role of fully destroying the 2D unstable counterpart.

Going back to the expansion \ef{full1}, where
 we denote the expansion coefficients by $C_l^\pm$, $l=1,2,3,4$,
we note that it means that those appear in the corresponding asymptotics as $x
\to \pm \iy$.
 Therefore, in what
follows, by shooting from the left-hand side (i.e. from $x=-\iy$)
with a pair $\{C_3^-,C_4^-\}$ in \ef{full1} (the rest of $C^-$'s
are zero), in order to create a proper solution $F(x)$ in $\re$,
we require the unstable manifold as $x \to +\iy$ to vanish. This
leads to a couple of analytic algebraic equations
\be
\label{exp9} \left\{
\begin{matrix}
 C_3^+(C_3^-,C_4^-)=0,
 \\
  C_4^+(C_3^-,C_4^-)=0.
  \end{matrix}
  \right.
 \ee

Thanks to our matching/gluing patterns construction we are able to state the following result.
\begin{proposition}
 \label{Pr.Anal1}
The ODE problem $(\ref{N1.1})$
 admits at most a discrete countable family of different  solutions.
  \end{proposition}

 It is also natural  to suggest that  any limiting point of this countable patterns subset
 should be
  at
 infinity in sense that their possible  infinite number cannot be
 somehow ``concentrated" on a compact interval in $\re$.
 Otherwise, those patterns must eventually grow without bound in
 $L^\iy$ close to the limiting point, but this  contradicts a
 clear blow-up character of solutions with big initial data on a
 bounded interval. Proving that assumes a more detailed and
 involved study of blow-up for the parabolic flow \ef{evol2},
which is not our goal here.

 Recall that for the translationally invariant ODE \ef{N1.1} the
  dimensions of manifolds can be reduced to one. In addition to
  that,
 since due to the fact that \ef{N1.1} is autonomous, it can be reduced to a
 third-order  ODE, but this does not help.

Let us also mention that some classical techniques do not apply
here. For instance (as a naive suggestion) if for the linear operator
 $$
 {\mathcal M} F=F^{(4)}+F \quad ({\rm ker}\, \mathcal M=\{0\}),
 $$
 using a proper Green's function, we rewrite the \ef{N1.1} problem
in the integral form 
 \be
 \label{JJ1}
  F= {\mathcal M}^{-1} F^2,
  \ee
  applications of various fixed point theorems and/or Schauder-type
  compact operator theory, homotopic vector fields, rotations/degree theory, {\em etc.,}  do not guarantee any result in view of
  a possible huge variety of patterns already observed. In any
  case, applications of standard classical techniques do not
  prevent obtaining the already known (and probably the most
  stable) basic variational solution $F_0(x)$, or even the trivial
  one $F \equiv 0$.
  Therefore, a delicate almost ``local" (actually, global)
  analysis of possible points of gluing/matching of patterns seems
  unavoidable to reveal a whole abundance of the set of
  solutions.

\ssk

Note also that, formally, the above proposition does not guarantee that
the set of nontrivial solutions is not empty.
Fortunately, this is not the case for the most of such problems {\color{blue} similar to \eqref{JJ1}} but, indeed, this can happen for some
specially designed equations. Of course a general possibility of
such a solvability crucially depends on the structure of
stable/unstable manifolds available on the plane of the equation.

\subsection{1D stability subspace of the equilibria $F =1$}

Consider the behaviour close to the nontrivial
steady state
 $
 F_* \equiv 1
 $
  for the ODE in \ef{N1.1},
 By linearization, we have
  \be 
  \label{111}
  \begin{matrix}
  F=1+Y \LongA Y^{(4)}= Y 
 \LongA
  Y(x) = {\mathrm e}^{\l x}
 \\
   \LongA \l^4=1 \LongA \l^2 = \pm 1 \LongA \l \in \{\pm 1,\pm{\mathrm i}\}.
 \end{matrix}
  \end{equation}
This means that in both directions $x \to \pm \iy$, there exist a
1D stable manifold and a 3D unstable. Therefore, to pose a
condition for  \ef{N1.1} like
 $$
 F(x) \to 1 \,\, \mbox{as} \,\, x \to \pm  \iy
  $$
  makes no sense since such patterns $F(x) \to 1$ (instead of $0$) as $x \to \pm \infty$ do not exists except $F(x) \equiv 0$. Moreover, 
  we will prove that even a heteroclinic $\{0\} \to \{1\}$ does not exist. Indeed, the same shooting from $-\iy$ by using a one-parametric expansion
  $$
  F(x) = 1 +C_1^- \eee^x +... \asA x \to -\iy, \,\,\, C_1^- \in
  \re,
   $$
  to vanish the 3D unstable manifold
  \be 
  \label{112}
 F(x)=1+C_1^+ \eee^x + C_2^+ \cos (x+\varphi_1)
 + C_4^+ \eee^{-x}+..., \quad t \to +\iy
  \end{equation}
 would lead to an overdetermined system of two
 algebraic analytic
 equations (like \ef{exp9}) with a single unknown $C_1^-$:
 \be
 \label{PrNN}
 \left\{
\begin{matrix}
 C_1^+(C_1^-)=0, 
 \\
  C_2^+(C_1^-)=0.
 \end{matrix}
  \right.
 \ee
 In general,
it is not easy  to find an ODE with the right-hand side of the type of
that in \ef{N1.1}, for which such a problem \ef{PrNN} has a
solution (this problem is doable but not that interesting or challenging).

\subsection{Finite $x$-blow-up of the ODE trajectories}

As rather usual in such semilinear ODE problems, a general
geometry of stable/unstable manifolds essentially can change when
the trajectories can disappear in a blow-up process at a finite,
say, $x_0^+$ and appear again (maybe, in a different manner) at
$x_0^+$.
We show that this is precisely the case for \ef{N1.1}, where in
view of the even fourth derivative operator, the mechanism of post
blow-up appearing of orbits remain practically the same. We also
show that disappearance/appearance of orbits at any finite blow-up
point does not change a positive dominant tendency of geometry of manifolds, which is
necessary for the solvability. More or less, finite $x$ blow-up is
no different from a natural exit/appearance of unbounded orbits at
$x = \pm \iy$.

Anyway, we need to consider such a blow-up in a greater detail. It
follows from \ef{N1.1} that, assuming now for convenience shooting
from the right-hand side, blow-up at $x=0^+$ happens according to
the equation
 \be
 \label{BB1}
 F^{(4)}=F^2, \quad F(x) \to \iy \asA x \to 0^+.
  \ee
  The real blow-up ``envelope" is easily obtained explicitly:
  \be
  \label{BB2}
  F_{\rm e}(x)= C_0 x^{-4} \ge 0, \quad C_0= 840.
  \ee
The linearized analysis is also straightforward (but not that
convincing, see below):
 \be
 \label{BB3}
 F(x)= F_{\rm e}(x)(1+Y(x)) \LongA (x^{-4} Y)^{(4)}= 1680 x^{-8} Y.
  \ee
For this Euler's-type equation, the characteristic equation is
   \be
   \label{BB3NN}
    Y(x) = A x^\a \LongA (\a-4)(\a-5)(\a-6)(\a-7)= 1680.
    \ee
    Clearly, this polynomial has  two real
   roots
    \be
    \label{BB4}
    \a_+ >7, \quad \a_0=0
    \ee
    and two complex ones. The positive root $\a_+>7$ corresponds to
    a 1D stable manifold of orbits blowing up according to
    \ef{BB2}. It seems that  the unstable linearized  manifold with $\a_0=0$
    is  possibly connected 
    with 
    orbits blowing up at some $x \not =0$, i.e., corresponds to the $x$-translational symmetry of the ODE \ef{BB1}.
    Possible  other ones as the manifolds consisting of some
     oscillatory functions corresponding to  complex-valued  $\a$'s require a further
      delicate nonlinear analysis.

     Namely, we perform in \ef{BB2} a ``blow-up" change of variables
      \be
      \label{BB5}
      F(x)={\mathrm e}^{-4 s} \varphi(s), \quad s=\ln x \to - \iy \asA
 x \to 0^+, 
 \ee
 where $\varphi(s)$ satisfies the following semilinear ODE:
  \be
  \label{BB6}
  L(\varphi)\equiv \varphi^{(4)} -22 \varphi''' + 179 \varphi'' - 638 \varphi' +
  840 \varphi-\varphi^2 =0 \forA s \ll -1.
  \ee
  In this form, it is easier to study the stability of its steady
  solutions, which now take the form
   \be
   \label{BB7}
    \varphi_1(s) \equiv C_0=  840 \andA \varphi_0=0,
    \ee
    where the first one is responsible for the stability of the
    envelope \ef{BB2}. Beyond those easy linearized stability
    approaches, there appears a nonlinear one concerning existence
    of a nontrivial oscillatory-like {\em periodic} solution $\varphi_*(s)$ of
    \ef{BB6} which is called the {\em oscillatory component}.

    First numerical attempts to solve \ef{BB6}
 show a huge instability as $s \to + \iy$ again with a fast
 blow-up according to
 \be 
 \label{Bl12}
\varphi^{(4)}=\varphi^2 (1+o(1)) \asA s \to s_0^- \,\, (\varphi(s) \to \iy).
 \ee
 As $s \to -\iy$, the behaviour is less unstable though blow-up
 according to \ef{Bl12} is also available. Actually, these numerics confirm the fact that
  such a nontrivial periodic solution $\varphi_*$  does not exists) and blow-up always occurs in a monotone style of \ef{BB2} and along its stable manifold. We will prove that later on that the blow-up behaviour in $\{F> \frac 32\}$ is always monotone for those orbits we are interested in.

Overall, we claim that the possible blow-up of local solutions of
\ef{N1.1} does not affect a general solvability of our ODE and cannot
 diminish the overall variety of solutions. Indeed. why cares what happens to semi-orbits which are not homoclinics, do these blow-up or just remain large enough away from a homoclinic range? 
One can see  that in the ODE
\ef{N1.1} blow-up via the {\em positive} mechanism in \ef{BB1} can occur if
 $$
 F^2-F \quad \mbox{is sufficiently positive, so that for large $F> 1$ (guaranteed for  $F > \frac 32$)}.
 $$

\section{A DS view: pattern-homoclinics, two basic and infinitely many other periodic and chaotic orbits,  an unstable  ``blow-up"  $W^{2,\iy}$-attractor in $\re^4$}
\label{SAttr}

At the moment, after achieving some better understanding of  the countable variety of patterns, we can try to use a standard 
approach to such a 4th-order  dynamical system (DS) and describe a whole geometric structure in $\re^4$ of orbits of our DS. Indeed, using standard DS theory we show the existence of two periodic orbits 
$\Gamma_{\rm max}$ and $\Gamma_{\rm min}$, which might be constructed from our gluing/matching of patterns technique, showing as well that there many other orbits in between those two. 



\subsection{A standard representation of a bounded  unstable ``blow-up"  attractor with two leading and infinitely many other periodic orbits}
We deal with a simple looking 4D DS: denoting $(F,F',F'',F''')^{\rm T}=(v_1,v_2,v_3,v_4)^{\rm T}=V$ we have
 \be 
 \label{DS9}
 F^{(4)}=-F+F^2 \,\,\, \,\, \LongA
 \left\{
 \begin{matrix}
v_1'=v_2, \qquad \quad  \\
v_2'=v_3, \qquad \quad  \\
v_3'=v_4, \qquad \quad \\
v_4'=-v_1+v_1^2,
\end{matrix} 
\right.
 \ee  
 and we begin to discuss its simple geometry on the basis of our previous results:

\begin{enumerate} 
\item[{\rm (i)}] Each pattern $F_\s(x)$ of any complexity finite index $\s$ in the admissible domain $\{F=v_1 \le \frac 32\}$ (see the next proposition) becomes a  {\em homoclinic} of the origin $O=(0,0,0,0)$. 
Let us recall again that  a single finite index  $\s$ cannot cover all the peculiarities 
of patterns and orbits, {\em etc}, but successfully describe some their key features important for a visual presentation.

\item[{\rm (ii)}] Using a fruitful comparison with the 2-wings (periodic) structure of Lorentz's classic strange attractor \cite{Lor63}, we next need to identify all periodic solutions of \ef{DS9} 
and their stability pattern subset. Periodic orbits  are not homoclinics though we always can find an orbit which  can approach O as close as necessary.
\end{enumerate}
We first claim that:
 $$
 \mbox{there exist (again) {\bf two} key main periodic orbits denoted by $\Gamma_{\rm max}$ and $\Gamma_{\rm min}$}.
 $$
The first main one $\Gamma_{\rm max}$ is the naturally largest periodic orbit  composed of positively dominant single hump waves 
  (a slightly negative only) $\{\hat F_0(x)\}^{x \in \re}$, i.e.,  consisting of the infinite series of single  humps  $\hat F_0 \approx F_0$ 
 continuously distributed over $\re$.
As we know, it is constructed via most densely $x$-compressed patterns $\hat F_0$  which cover the whole $x$-axis.
Note that by itself $\hat F_0(x)$ does not have exponential tail connection with $x= \iy$ and, unlike $F_0(x)$,  is not a pattern.
The  basic patterns $\{F_k\}_{k \ge 0}$ composed from $k+1$ neighbouring ones being close to the 
first variational pattern $\approx F_0(x)$ (again mostly densely compressed with a minimal 2 transversal zeros in between).
So that those are $\sim$ any finite pieces of $\Gamma_{\rm max}$.
Thus, the first periodic orbit $\Gamma_{\rm max}$  can actually be classified as $\{F_{\iy}(x)\}$ in $\re^4$ and is given by a single wave periodic pattern  $\hat F_0(x)$, $x \in [-\frac {\hat T}2,\frac {\hat T}2]$:
 \be 
 \label{DS7} 
 \tex{
\Gamma_{\rm max}=\{(v_1= \hat F_0(x),\, v_2=\hat F_0'(x), \, v_3= \hat F_0''(x), \, v_4=\hat F_0'''(x))^{\rm T}\}.
 }
 \ee 
 Any  pattern orbit starting with an infinite number of rotations around the origin, as $x \to - \iy$,
  according to a typical elliptical  projection on the $\{F,F'\}$-plane:
   \be 
   \label{sp2}
  \tex{
  \big(F-\frac 1{\sqrt 2} F'\big)^2 +  \frac 12 \, F^2= \frac 12 B^2 \eee^{2\mu x} + O(\eee^{3 \mu x}), \quad x \to -\iy,
  }
   \ee
describing a spiral out behaviour with the clock-wise rotation from $x=-\infty$.   
After \ef{sp2} the homoclinic partially approaches  $\Gamma_{\rm max}$ (or others) and can rotate near $\Gamma_{\rm max}$ an arbitrary number of times before 
 returning to O according again to \ef{sp2} with now $x \to +\iy$ and $\mu \mapsto - \mu$. On the other hand, after a finite spiral in rotations via \ef{sp2} it can return to $\Gamma_{\rm max}$ again and again and ends up with an infinite number of converging rotations around the origin according to \ef{sp2}, i.e.,  the exponential tail behaviour as $x \to +\iy$.

 \ssk
 
The second key periodic orbit in $\re^4$ denoted by  $\Gamma_{\rm min}$ exists creating the second family $\{F_{+2l}\}_{l \ge 2}$, generated by a double hump  positive hump $\hat F_{+4}$, with the same homoclinic properties. 
Note that the characteristic equation \ef{111} with $\l= \pm {\rm i}$ and the linearized periodic patterns \ef{112}
with pure $\cos x$ somehow indicate existence of a positive periodic solution about $F \equiv 1$ and even the linearized period $2\pi$ is close to the actual one. We do not prove the existence of such a solution though it is crucial for forming of the second wing of the attractor. Recall that for the corresponding  PLP (\ref{PL111}) this solution is elementary:
 $$
 \tex{
 \Gamma_{\rm min}: \quad F_{\rm min}(x)= 1 + \frac 12 \cos x.
 }
 $$ 
$\Gamma_{\rm max}$ for the PLP can be also calculated algebraically.

 Thus, 
 $\Gamma_{\rm min}$ corresponds to the periodic orbit $F_{+2l}|_{l=\iy}(x)$, and is  a strictly positive periodic solution oscillating about $F=1$ (geometrically, around $F \sim \frac 74$). 
$\Gamma_{\rm min}$ has a similar accompanying  subset of homoclinics, which starting and finishing with the spirals like \ef{sp2} can make any number of $x$=disjoint rotations around it as many times as we want with shorter or longer asymptotic tails in between.

A schematic picture of key periodic orbits  $\Gamma_{\rm min/max}$ is shown in Fig. \ref{FigW2}; note that even in $\re^4$
$\Gamma_{\rm min}$ is actually ``embedded" inside $\Gamma_{\rm max}$.

\begin{figure}[htbp]
\begin{center}
\includegraphics[scale=0.32]{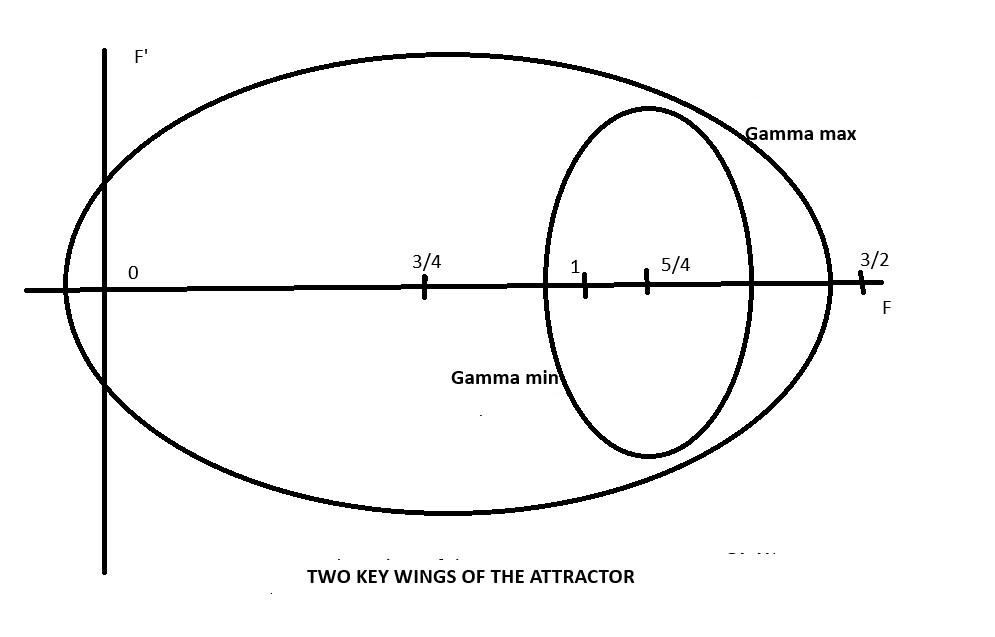}
\caption{Two key wings $\Gamma_{\rm max/min}$ of the attractor $W^{2,\infty}$ (the quadratic ODE).}
 \label{FigW2}
\end{center}
\end{figure}

\ssk

A further countable number of periodic orbits $\{\Gamma_\s\}$ can be formed from any given pattern $F_\s(x)$ with any finite index $\s$ being periodically extended over $x \in \re$, denoted by $\Gamma_\s$. Each one has its ``stable" subset of pattern orbits which can mimic and arbitrarily close approach any sufficiently long finite part of their periodic structure in $x \in \re$. Roughly speaking, all those periodic orbits $\Gamma_\s$ composed from various pieces of two above key periodic $\Gamma_{\rm max/min}$ with an arbitrarily large ``discrete" (via zeros of tails $\sim \cos(\frac x{\sqrt 2})$) distances.

\ssk

Projecting for simplicity this variety of periodic orbits and homoclinics on the $(F,F')=(v_1,v_2)$-plane 
we have that:
 
 \begin{itemize}
 \item $\Gamma_{\rm max}$ represents a deformed circle/ellipse geometrically centered at $\sim (\frac 34,0)$ of the
 radius $\hat  R \sim \frac 34$ (the actual oscillations of $\hat F_0(x)$ is about the steady state $F \equiv 1$). It has a small $F$-negative part. Rotations are clock-wise as all the others orbits/patterns if we start shooting from $x=-\iy$.
 
 \item $\Gamma_{\rm min}$ then is smaller, $F>0$ and lies almost completely inside $\Gamma_{\rm max}$. It also has a circular/elliptic deformed form with a formal geometric center at $\sim (\frac 74,0)$ (the actual oscillations are again about $F \equiv 1$).

\item other periodic orbits as invariant sets fill the space between those two key $\Gamma_{\rm max/min}$ and partially inside 
$\Gamma_{\rm min}$. Recall that some periodic orbits can infinitely many times approach O as close as possible (but do not converge to O).

\item the rest of this projected 2D phase-space is partially covered by a countable number of our pattern homoclinics of O.
\end{itemize}

Projections on other planes (say on $\{F',F''\}$) look similar, simpler, and more symmetric.
In $\re^4$, this plane structure will get an extra 2D volume but structurally remains geometrically the same.

\ssk

Thus we observe:
\begin{enumerate}
\item[{\rm (a)}]  the attractor in $\re^4$ has two key special  periodic orbits $\Gamma_{\rm max/min}$ somehow embedded, but overall
contains  an {\em infinite countable subset of invariant periodic wings},
which we denote by $W^{2,\iy}$-one (so that, surprisingly, $2$ as in the Lorentz's one and $\iy$ as the actual number of periodic orbits involved. Though the Lorentz's attractor can also have infinitely many periodic orbits);

\item[{\rm (b)}] patterns $F_\s$ create a countable subset of various homoclinics of O partially approaching any finite pieces (a collection of humps) of those periodic orbits $\Gamma_\s$;

\item[{\rm (c)}]  $W^{2,\iy}$ is unstable and, moreover, in a {\em blow-up} manner: other local orbits of a given  $F(x)$, starting at O ($x=-\iy$),  blow-up as $x \to x_F< \iy$, corresponding to $|V| \to \iy$ (see more  below);

\item[{\rm (d)}] there is an almost ``isomorphic" formal resemblance between this $\{\s\}$ geometric structure and  
real numbers on the interval $(0,1)$ according to:

\begin{itemize}
\item rational numbers correspond to periodic orbits in a 1-to-1 manner,
  
\item our patterns (homoclinics of O)
are simple rational numbers having a finite writing form $0.a_1a_2...a_{|\s|}00...$ via finite $\s$'s (and, say,  ``zeros" afterwards, though this is not necessary to ascribe);  
  
\item a  subset of the irrational numbers (say, transcendent) giving global solutions of \ef{DS9} which are not patterns and do not have any finite periodic structure inside;

\item the rest of non-transcendent irrationals  corresponds to principally local orbits blowing-up at finite $x$'s; by continuity, this is a true {\sc uncountable subset}. 
\end{itemize}
\end{enumerate}

\subsection{$W^{2,\iy}$-attractor: numerology}

One can reduce once the order of the ODE in \ef{DS9} by multiplying this divergent operator by $F'$ and integrating
  over $(-\iy,x)$ (a Hamiltonian property):
 \be 
 \label{DS1}
 \tex{
F^{(4)}=-F+F^2 \LongA 2F'F'''-(F'')^2=p_0(F) \equiv F^2(\frac 23 F-1),
 }
 \ee
 where the constant of integration is zero since we deal only with a 1D subset of orbits satisfying $F(-\iy)=F'(-\iy)=...=0$. According to the known expansion (then the invariant $x$-translational parameter cancelled)
  \be 
  \label{DS2}
  \tex{
  F(x)=B \eee^{\mu x} \cos (\mu x) + O(\eee^{2\mu x}) \asA x \to  - \iy, \,\,\,B \not=0; \,\,\,  \mu = \frac 1{\sqrt 2}.
   }
   \ee
   Note that for any pattern $F(x)$, $B \in [1,{\rm e}^{2 \pi})$, see Proposition \ref{TwoFund}.

 Next, for convenience, we present a  ``number theory" of $W^{2,\iy}$ in terms of the single parameter $B$ in \ef{DS2} generating a local or global orbit $F=F_B(x)$.
 
 \begin{proposition}
 \label{PrNum}

\begin{enumerate}
\item[{\rm (i)}] $\{B\}$ can be reduced to $B \in {\mathcal B}=(1,{\rm e}^{2\pi})$ plus $B=1$.

\item[{\rm (ii)}] ${\mathcal B}_{\rm hom}=\{B: \,\, F_B(x)\,\,\mbox{is a homoclinic of O}\} \subset {\mathcal B}$  is countable.

\item[{\rm (iii)}] ${\mathcal B}_{\rm glob}=\{B: \,\,F_B(x) \,\, \exists \,\,\mbox{in $\re$}\}=
{\mathcal B}_{\rm bound}=\{B: \,\,|F_B(x)| \le C \,\,\mbox{in} \,\,\re\}$.

\item[{\rm (iv)}] ${\mathcal B}_{\rm blow}=\{B: \,\, \mbox{$F_B(x)$ {blows up at some} $x=x_0^-<+\iy$}\}$  {is uncountable}. 

\item[{\rm (v})] ${\mathcal B}={\mathcal B}_{\rm glob} \cup {\mathcal B}_{\rm blow}$.
\item[{\rm (vi)}] The set ${\mathcal B}_{\rm blow}$ is a unity of a  sequence  of maximal disjoint
 open non-empty  intervals 
 $(a_j,b_j)$
  \be
  \label{Dis44}
 {\mathcal B}_{\rm blow}= \cup_{j=1}^\iy (a_j,b_j),
  \ee
  {such that} 
  \be 
  \label{Dis45}
[a_i,a_j]\cap [a_j,b_j]=\{\emptyset\} \,\,(\mbox{i.e.,} \,\,a_i \not = b_j) \,\, \mbox{for all} \,\,i \not = j \,\, \mbox{and}\,\,
 a_i,b_i \in {\mathcal B}_{\rm glob} \,\,\,\mbox{for all} \,\,\,i=1,2,...\, .
 \ee

 \item[{\rm (vii)}] ${\mathcal B}_{\rm glob}={\mathcal B}_{\rm hom} \cup {\mathcal B}_{\rm per} \cup {\mathcal B}_{chaot}$,
 where in the last two subsets orbits $F_B(x)$ with a periodic or chaotic (non-periodic) behaviour  as $x \to +\iy$ are meant.

 \item[{\rm (viii)}] { For any}
  $\hat B \in {\mathcal B}_{\rm per} \cup {\mathcal B}_{\rm chaot}$ {there exists a sequence} $\{B_j\}\subset {\mathcal B}_{\rm hom}$
   {such that} 
   $\{B_j\} \to \hat B$.
\end{enumerate}

 \end{proposition}
 
 \begin{proof}
 \begin{enumerate}
\item[{\rm (i)}]  Proposition \ref{TwoFund}.

 \item[{\rm (ii)}] Thanks to the analyticity of the flow and $(2,2)$ deficiency indices of the linear operator.
  
   \item[{\rm (iii)}]  Since $F_B(x) \le \frac 32$ the behaviour $F_B(x) \to -\iy$, along a sequence $\{x_j\} \to + \iy$, with the corresponding  controlled growth of  derivatives, is forbidden by the equation
$$F^{(4)} =-F + F^2 \to +\iy,\quad \hbox{for}\quad x <x_j,$$ 
 guaranteeing fast blow-up for large $j$.

    \item[{\rm (iv)}] 
 If $F_{\hat B}(x)$ blows-up as $x \to x_0^-(\hat B)$ then by the continuous dependence of $x_0(B)$ on the parameter  $B$   blowing-up happens for all $B \approx \hat B$.
 
     \item[{\rm (v)}] By {\rm (iii)} and {\rm (iv)}.

      \item[{\rm (vi)}] By definition, ${\mathcal B}_{\rm blow}$ is open. Since each $(a_j,b_j)$ is maximal, the end points $a_j, b_j \not \in {\mathcal B}_{\rm blow}$ and hence belong to ${\mathcal B}_{\rm glob}$ by {\rm (v)}. 
       
       \item[{\rm (vii)}] Just a classification.
       
       \item[{\rm (viii)}] By the matching/gluing procedure.
       \end{enumerate}
       \end{proof}
 
{\bf Example.} Consider a random covering of the unite interval ${\mathcal B}=(0,1)$ by a sequence $\{(a_j,b_j)\}_{j=1}^\iy$ of open intervals with the property \ef{Dis45} and such that for ${\mathcal B}_{\rm blow}$ given by  \ef{Dis44} 
$$ 
 \tex{
 {\rm meas}\,{\mathcal B}_{\rm blow}=
\sum_{j=1}^ \iy (b_j-a_j)= {\rm meas}\,{\mathcal B}=1.
 }
 $$
 Then {\rm (viii)} holds  for any $\hat B \in {\mathcal B} \setminus {\mathcal B}_{\rm blow}$.
 
 \smallskip

As usual, the attractor $W^{2,\iy}$   is indeed a ``strange" one in the sense there exist homoclinic orbits with an arbitrary number of rotations around $\Gamma_{\s}$ and the origin. 
  A sequence of such successive rotations and its length are also arbitrary, and this eventually  creates chaotic orbits. Note that with our previous results in hand, we can guarantee that a homoclinic with  any prescribed finite number and sequences of rotations {\em actually exist}, so that we now precisely are familiar with  the internal structure of such a chaotic object.


\subsection{Reducing the ODE order  and the admissible subset $\{F \le \frac 32\}$}

\begin{proposition}
\label{PrDS1}
{\rm (i)} Any local orbit starting at $x=-\iy$
 crossing for the first time
 the level $\{F= \frac 32\}$ at a finite $x=x_0$ cannot return to the admissible subset $D_{\rm hom}=\{F \le \frac 32\}$.

{\rm (ii)} A heteroclinic orbit connecting equilibria $\{0\} \to \{1\}$ does not exist.

\end{proposition}

\begin{proof}
 
{\rm (i)} By \ef{DS1}  any min/max points of the orbits for $F > \frac 32$ are impossible:
\be 
\label{DS3}
\tex{
F'=0 \LongA -(F'')^2=F^2(\frac 23 F-1)>0 \quad(\mbox{same for ``3-inflections", $F'''=0$}).
 }
 \ee
It follows that if $F(x_0)=\frac 12$ and
  $F(x)$ changes sign in any neighbourhood of $x=x_0$, then by analyticity 
$F'(x)>0$ for all small $x-x_0>0$. Hence,  $F'(x)>0$ for  $x>x_0$, i.e., $F(x)$ is strictly increasing, and
it is natural to expect that those local orbits for all such $B$'s then are attracted to a blow-up at some $x_1>x_0$, as explained earlier. Such orbits correspond to the monotone blow-up, while others (now from the manifold  \ef{DS2}) can be oscillatory according to the structure of the periodic oscillatory component and can partially appear in the admissible subset $D_{\rm hom}$. Actually those essentially sign changing orbits can represent a connection between any pair of
blow up points $(x_-,x_+)$ in such a way that $F(x) \to \iy$, as $x \to (x_\pm)^\mp$.

{\rm (ii)} 
If $F(x) \to 1$ as $x \to + \iy$ in $D_{\rm hom}$, then (\ref{DS1}) yields $0=- \frac 13$.

\end{proof}

The next step is to further reduce the order of \ef{DS1} to a non-autonomous 2nd-order ODE 
 \be
 \label{DS4} 
 \tex{
 F'=\psi(F), \,\,\,F''= \psi \psi'_F \LongA
  \psi''= -\frac {1}{2 \psi}(\psi')^2 + \frac 1{3 \psi^3} F^2(F-\frac 32).
   }
   \ee  
Let us describe some of its  properties translated from the homoclinic language used above.

\ssk

{\rm (i)} The asymptotic exponential tail decay \ef{DS2} yields the behaviour near the origin
 as $x \to - \iy$
 \be 
 \label{DS5}
  \tex{
F \to 0, \,\,\,\psi=F' \to 0, \,\,\, \mbox{and} \,\,
 \psi'_F= \frac {F''}{F'} \sim - \sqrt 2 \, \frac {\tan (\mu x)}{1+\tan(\mu x)}.
 }
 \ee 
For patterns, a similar behaviour is expected to exist, as $x \to -\iy$, so that
by \ef{DS5} we see a homoclinic of the $\psi'$-axis: 
 \be
 \label{DS6}
 H_{\rm hom}= \{F=\psi=0,\psi' \in \re \}.
  \ee 
On projection on the $\{F, \psi\}$-plane, it is viewed as a standard homoclinic of the origin.

\smallskip

{\rm (ii)} The key periodic orbits $\Gamma_{\rm max/min}$ become quite tricky and unbounded in the present variables. We have that
 \be 
 \label{DS7NN} 
 \tex{
\Gamma_{\rm max}=\{F= \hat F_0(x),\, \psi=\hat F_0'(x), \, \psi'= \frac {\hat F_0''(x)}{\hat F_0'(x)}, \,\, x \in 
[-\frac T2, \frac T 2]\} \,\,\,(T \,\,\mbox{is its period}),
 }
 \ee 
 so that $\psi'$ take infinite values when $F' \to 0^\pm$ and $\psi'' \not \to 0$.
 The pattern orbit can rotate near $\Gamma_{\rm max}$ as many times as possible, before 
 returning to the axis \ef{DS5} with an infinite number of rotations around.
 In general, any periodic orbit is situated on two ``bottles" which are infinite along the $\psi'$-axis.
 
 A similar description is true for $\Gamma_{\rm min}$  attracting the second family $\{F_{+2l}\}_{l \ge 2}$, with the same homoclinic properties. As we know, $\Gamma_{\rm min}$ corresponds to the periodic orbit $F_{2 \times \iy}(x)$, i.e., a  periodic solution in $\re$ (not a pattern) oscillating about $F=1$. The same holds for any of those periodics $\Gamma_\s$ composed by a periodic extension in $\re$ of any given pattern $\sim F_\s$ with spatial gaps organized according to the semi-period $\pi \sqrt 2$ of the exponential tail in \ef{DS2},  {\em etc}.

 \ssk

\subsection{A note on the cubic nonlinearity: $W^{5,\iy}$-attractor}

For the same-type equation with the cubic nonlinearity
 \be 
 \label{Cu10}
 F^{(4)}=-F+F^3,
 \ee 
the 1D exponential manifold \ef{DS2} remains identical, so that
 the admissible subset is a bounded one: if $F'=0$ (or $F'''=0$), then
  \be 
  \label{FF3}
  \tex{
  2F'F'''-(F'')^2 =
2\int_{-\iy}^x (-F+F^3)F' {\mathrm d}x
= \frac 12  F^2(F^2-2) \le 0 
\Rightarrow D_{\rm hom}= \{|F| \le \sqrt 2\},
 }
  \ee
  so that oscillations with local  min/max points are possible in $D_{\rm hom}$ only.
  We also immediately observe that for (\ref{Cu10}) by \ef{FF3}
   $$
   \tex{
   \not \exists \,\,\mbox{heteroclinics of equilibria $\{0\} \to \{\pm 1\}$}\quad  (0 \not =-\frac 14).
   }
   $$
   
  Orbits intersecting  the boundary  $\{|F|= \sqrt 2\}$ never come back and seem all blow-up at a finite $x=x_0^-(B)$. A further convergence to steady states $\pm 1$ is also impossible. 
Since the operator in \ef{Cu10} has odd nonlinearities and the functional is even (the L--S genus variational approach then applies), the overall geometric structure of the blow-up attractor is slightly more complicated
and consists of (see Fig. \ref{FigW2VV}):
\begin{enumerate}
\item[{\rm (i)}] Two $\pm \Gamma_{\rm max}$ and two $\pm \Gamma_{\rm min }$ corresponding to reflection $F \mapsto -F$.

\item[{\rm (ii)}] In addition, there exists a new main key periodic orbit denoted by $\Gamma_{\rm L-S}$ corresponding  
to the ``final"  L--S pattern  formally having an infinite L--S genus (category) and hence having an infinite length of dominated positive/negative humps. As usual,  any finite collection of  finite pieces of this periodic orbit can create a homoclinic, according to the same scenario as above for $\Gamma_{\rm min/max}$. 
Thus, the whole new L--S countable subset of patterns
 $\{F_k^{\rm L-S}\}_{k \ge 0}$ (note that $F_0^{\rm L-S}=F_0$), each one consisting of $k+1$ humps successively  changing signs
 (and perturbed ones by moving out any single pattern) is involved forming extra countable subsets of periodic orbits.   
On the $\{F,F'\}$-projection plane $\Gamma_{\rm L-S}$ looks more symmetric, circular,  and is centered at $(0,0)$.

\end{enumerate}

Overall, the infinite countable subset of invariant wings $W^{5,\iy}$ consists of  $2+2+1=5$ main key periodic orbits forming the same two ``double-layer" wings in both $\{F \ge, \le  0\}$ and a new  central L--S one; all are creating other periodics and  catching homoclinics in the same manner. 

\begin{figure}[htbp]
\begin{center}
\includegraphics[scale=0.32]{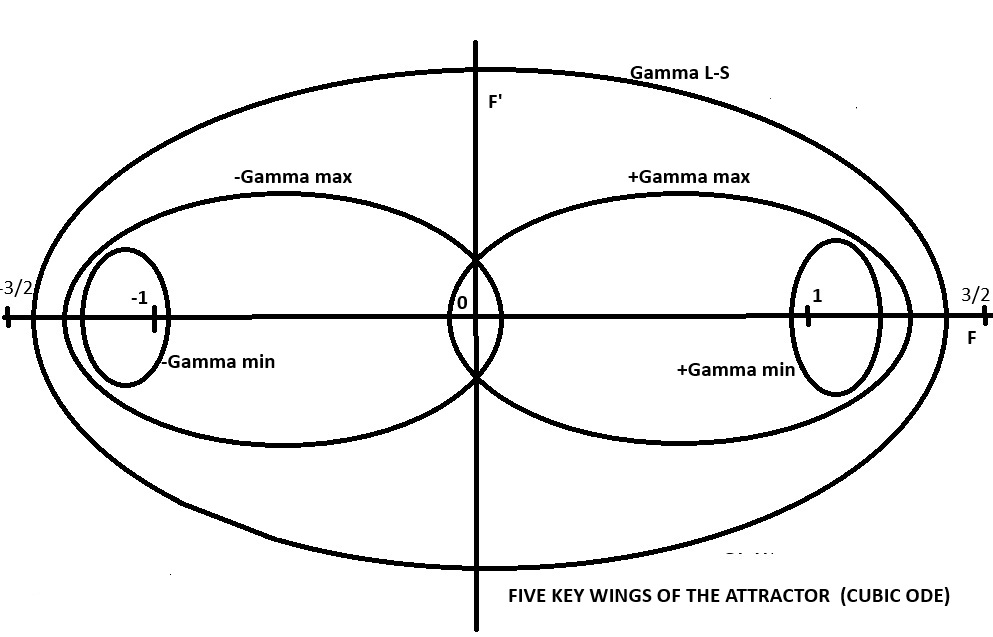}
\caption{Five key periodic  wings of the attractor $W^{5,\infty}$ for (\ref{Cu10}).}
\label{FigW2VV}
\end{center}
\end{figure}

\ssk

We do not attempt to mathematically justify all the above properties and to carefully study the given DS's.
Overall, as we mentioned, we are trying to stress the attention to the fact that similar nonlinear phenomena  
exist for a much wider and higher-order DS including
 $$
 (-1)^m D_x^{(2m)}F=-F + F^2 \quad (...+F^3, \,\,\,\mbox{\em etc.}) \,\,\mbox{for any} \,\, m \ge 3,
 $$
 where a standard DS-type detailed analysis of more complicated unstable blow-up $W^\iy$-attractors  does not look promising. Other semilinear and quasilinear non-variational DSs with similar properties can be introduced.

\ssk

Instead, in the second half of this paper \cite{AEGII}, we  present an algebraic description
of all the patterns (homoclinics of O) by using a {\em piece-wise linear} (P-L) approximation of the quadratic ODE:
 \be 
 \label{hhh1}
  \tex{
 F^{(4)}=-F+F^2 \,\,\Leftrightarrow \,\, F^{(4)}=|F- \frac 12|- \frac 12, \,\,\,F(\iy)=0 \,\,\,(\mbox{(PLP): a P-L Problem}).
   }
  \ee
 This PLP  admits a full patterns classification via fundamental systems of solutions of linear operators involved by matching/gluing  those $C^4$-solutions at the level $\{F= \frac 12\}$ leading to finite transcendent equations. The main difficulty is then to construct a proper continuous (a ``homotopic" in a natural sense if possible) $\e$-deformation for $\e \in [0,1]$ connecting ODEs in \ef{hhh1}, $\e=0 \leftrightarrow \e=0$, without losing Morse and/or rotations of vector fields indices (degrees of operators)  of all the suitable pairs of patterns available.

\section{On a formal expansion towards the first pattern $F_0(x)$: typical transcendent equations}
\label{SExpan}

 In a formal fashion, we briefly explain
in our basic quadratic model how  patterns can appear by an elementary expansion technique.
Similar but more difficult manipulations can be done using different even somehow implicitly prescribed structures of asymptotic tails as in \ef{AA1} and \ef{AA2}.  A final conclusion will remain the same:
a nontrivial ``mass-concentration" via changing the parameter of  1D asymptotic tails leads to a possible first pattern.
Of course  in the simplest case of \ef{N1.1} leading to manipulations with elementary functions  this is most convenient and easy.
Another  our goal here by presenting such trivial calculations with pure exponential and polynomial   is just to give some insight into 
$$
\mbox{typical transcendent manipulations occurring in P-L ODEs},
$$
a class of equations with piece-wise nonlinearities admitting a full algebraic pattern classification.
We have presented some already to be studied more in an attempt to classify the patterns for some non P--L models.

\subsection{A quadratic extension from $x=-\iy$}

{\sc (i)} Starting again from $x=-\iy$ according to \ef{DS2}, the first step is purely linear:
 \be
  \label{Z1}
  F^{(4)}=-F+... \LongA F(x) = F_B^{(0)}(x)= B \eee^{\mu x} \cos(\mu x) +... \, .
   \ee

{\sc (ii)} The second step involves the quadratic term $+F^2>0$ for $F \not = 9$ again in a linear fashion:
 \be 
 \label{Z2}
 \begin{matrix}
 F(x)=F_B^{(0)}(x)+ Y(x), \quad Y(x)=o(\eee^{\mu x}) \LongA Y^{(4)}=-Y + (F_B^{(0)}(x))^2+...=
 \\
 =-Y+ B^2 \cos^2(\mu x)+...= -Y + \frac 12 B^2(1+\cos(2 \mu x))+...
   \LongA
 \\
 Y(x) = F_B^{(1)}(x)= \frac 1{48} ^2 x^4 +  \frac 18 B^2 \cos (2 \mu x) - \frac 1{16}B^2 \cos(2\mu x)+... \quad (\mu = \frac 1{\sqrt 2}),
  \end{matrix}
  \ee
where for simplicity we present a part of this solution just to fix how the quadratic term can affect the linear expansion to add some clear positive influence to finally move the linear expansion up to create something looking like the first
pattern (but with a clear violation of its symmetry, cf. an improved  even expansion  below). In the last line of \ef{Z2}
we keep a single non-negative term $\sim x^4$ from the corresponding polynomial since it mostly affect the increased positivity of the improved linear expansion.
This is enough to visualise  a first $B \not =0$ for which a double zero appears (a tangent point at the $x$-axis) for this two terms approximation
 \be 
 \label{XX1}
F(x) \approx \tilde  F_B(x) \equiv  F_B^{(0)}(x) + F_B^{(2)}(x) \LongA \,\,(B_0,x_0): \,\,\,
 \left\{
 \begin{matrix}
 \tilde F_B(x)=0, 
 \\
  \tilde F_B'(x)=0.
  \end{matrix}
  \right.
 \ee  
Using \ef{Z1}, \ef{Z2} yields the algebraic system
  \be 
  \label{XX2}
  \left\{
  \begin{matrix}
  \cos(\mu x) + \frac 1{48} B x^4 + \frac 18 B \cos(2\mu x)=0, \quad 
  \\
  -\mu \sin(\mu x) + \frac 1{12} B x^3 - \frac{\mu}4 B \sin(2 \mu x)=0.
 \end{matrix}
 \right.
  \ee 
Substituting $B$ from the first equation
 \be 
 \label{XX3}
 \tex{
  B= - \frac{ \cos(\mu x)}{ \frac 1{48} x^4 + \frac 18 \cos(2 \mu x)}  
 }
  \ee
into the second one yields a single transcendent equation for the tangent point $x=x_0$:
 \be 
 \label{XX4}
 \tex{
x_0: \,\,\, \frac {\mu}4 \, \tan(\mu x)= -  \frac{x^3- \mu \sin(2\mu x)}{x^4+ 6 \cos(2 \mu x)}.
  }
  \ee 
We do not study this equation but note that the existence of a first minimal $x_0$ is seen by using an elementary balance  of the terms of different signs in \ef{XX1}.

A further enhancement  of such linearized expansions can be continued but computations then become much more cumbersome and do not clarify any principle properties. In particular, it is difficult then to see a birth of the second pattern $F_2 \sim F_0-F_0$, by gluing two $F_0's$, i.e., having  
 two positively dominant humps.


\subsection{An even extension centered at $x=0$} The linear approximation if $F(0) \not = 0$ is then even as well as  the second quadratic term (and other ones if necessary), and these are:
 \be
 \label{XX5}
  \begin{matrix}
  F_B^{(0)}(x)=B \cosh(\mu x) \cos (\mu x) \LongA F(x)=F_B^{(0)}(x) + Y, 
 \\
 F_B^{(1)}: \quad Y^{(4)}=-Y+ \frac 14 B^2(1+\cosh(2\mu x))(1+ \cos(2\mu x)) \LongA
 \\
 F_B^{(2)}(x) = \frac 1{96} B^2 x^4 + \frac 1{16} B^2 \cosh(2 \mu x) + \frac 1{16} B^2 \cos(2 \mu x)+... \, ,
 \end{matrix}
  \ee 
where we omit the last term containing $\sim \frac 1{16} B^2 \cosh(2\mu x) \cos(2 \mu x)$ which is of the same order and even smaller than the third term. Acting as around \ef{XX1}, for the first touching  of  
the axis $x$ at the point $(B_0,x_0)$ we obtain a similar algebraic system
 \be 
  \label{XX21}
  \left\{
  \begin{matrix}
 \cosh(\mu x) \cos(\mu x) + \frac 1{96} B x^4 + \frac 1{16} B \cosh(2\mu x)=0, \qquad \qquad \qquad \qquad 
  \\
  \mu [\sinh(\mu x)\cos(\mu x)-\cosh(\mu x) \sin(\mu x)] + \frac 1{24} B x^3  \frac{\mu}8 B \sinh(2 \mu x)=0.
 \end{matrix}
 \right.
  \ee 
  This yields a transcendent equation for $x_0$:
   \be 
   \label{XX22}
   \tex{
    B= - \frac{\cosh(\mu x) \cos(\mu x)}{ \frac{1}{96}x^4 - \frac 1{16} \cosh(2 \mu x)} \,\,\Longrightarrow \,\,
     \mu [\tanh (\mu x)- \tan(\mu x)]=4 
      \frac  { x^3 + 3\mu  \sinh(2 \mu x)}{x^4- 6 \cosh(2 \mu x)}.
       }
       \ee
 Again the structure of functions in \ef{XX5} shows that such a first tangency point $(B_0,x_0)$ exists.
  Moreover, the positive terms like $\frac 1{96}B^2 x^4$ and the next one in \ef{XX5}, make it possible 
  to lift up the second negative hump  in $F_B^{(0)}(x)$ to create a touching point. Then,  we can see that
  the length of the resulting positive hump will be close to the full  period of $\cos(\mu x)$ (i.e., the distance between two consecutive minima of its graph), i.e., to
   $
   T_1= 2 \mu \pi= 2 \sqrt 2 \pi = 8.8857...\, .
   $
  This is the typical length we have observed in all numerics associated with multi-hump patterns composed somehow from $\sim F_0(x)$'s.  Thus smaller oscillatory 
  terms having $\cos(2\mu x)$ create  oscillatory tails around $x_0$ with the half period
    $
    T_2= \frac 12 T_1= \sqrt 2 \pi = 4.4429...\,,
    $ 
    and the actual distance between neighbouring zeros is two time small and is 
     $
     \sim \frac \pi{\sqrt 2} \sim 2.2214...\, ,
     $
      which we also can see in the related figures. 
  Such  elementary but rather bulky manipulations given above partially remind us those rigorous ones 
  which are necessary and actually unavoidable for a complete classification of all the patterns of the PLP as in \ef{hhh1}. 
  In fact, those algebraic computations get more difficult if we really pretend to perform a kind of algebraic classification of all the patterns to appear.

\section{Pattern formation in several related models}
 \label{SMod}

 
 In this section we show several higher order ODE models for which our matching/gluing performance provides 
 some pattern formation.
    
\subsection{Not a $C^2$-nonlinearity: $F^{(4)}=-F -(|F|F-F)''$}
Here we briefly introduce  patterns for the ODE (\ref{Mod33}) with a $C^{1,0}$ nonlinearity.
The corresponding functional is calculated in the metric of $H^{-1}(\re)$: applying  $D_x^{-2}$
to \ef{Mod33}) yields
 $$
 \tex{
 F''=(-D_x^{2})^{-1}F - (|F|F-F) \LongA \Phi(F)= \frac 12 \int[(F')^2+F^2+(D_x^{-1} F)^2] - \frac 13 \int |F|^3,
 }
 $$
 in $H^1(\re)$.
The operator has odd nonlinearities, the functional is even and sufficiently regular, so the L--S sequence of critical points is guaranteed.
 Note that in this differential expressions we observe a  discontinuity:
 $$
 (|F|F)''= |F|F''+ {\rm sign}F \, (F')^2.
 $$
 For less regular operators allowing  similar patterns, there appear  measures: e.g.,
  $$
  F^2 \mapsto |F| \Longrightarrow (|F|)''= {\rm sign}F F'' + 2 \delta (F) (F')^2,
  \,\,\,\mbox{\em etc.}\,,
  $$
which makes no essential difference in a possible matching and other analysis.  
  
As usual, we also  show  non L--S patterns.
The shooting to the left point is again $x_0=-20$ with the same standard leading term linear expansion for $x \ll -1$:
$
F(x) \approx B {\rm e}^{\mu x} \cos(\mu x)$, $\mu = \frac 1{\sqrt 2}$.
See Fig. \ref{Mod1}--\ref{Mod6}.

\begin{figure}[htbp]
 \hfill   \hfill  \begin{minipage}[t]{0.45\textwidth} 
    \centering 
    \includegraphics[width=7cm,keepaspectratio]{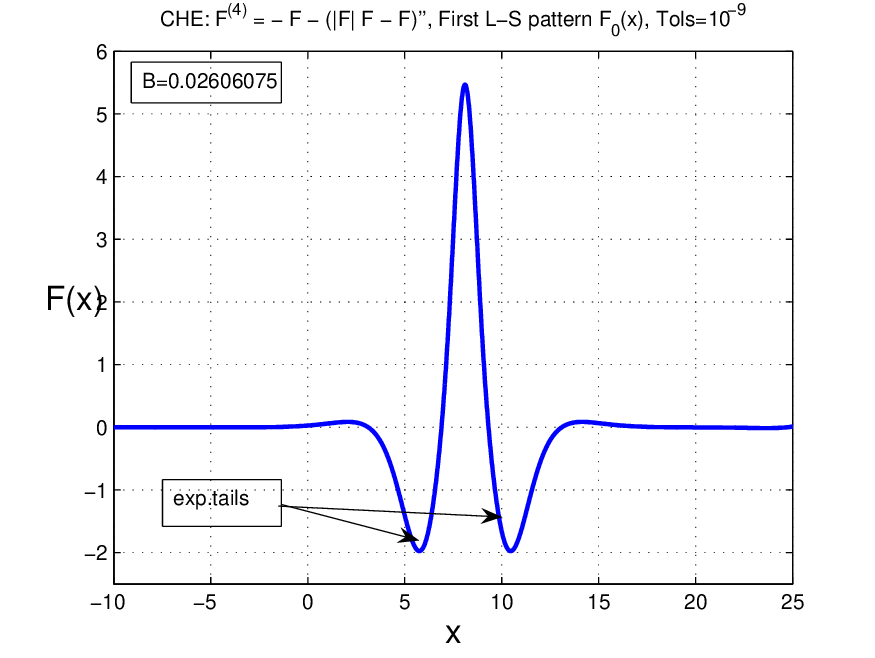}
   \caption{$F^{(4)}=-F -(|F|F-F)''$:
 the first basic pattern $F_0(x)$ (the absolute minimum point of the functional $\Phi(F)$), $B=0.02606075$.}
    \label{Mod1}
\end{minipage}
  \hfill   \hfill 
\begin{minipage}[t]{0.5\textwidth}
    \centering 
    \includegraphics[width=7cm,keepaspectratio]{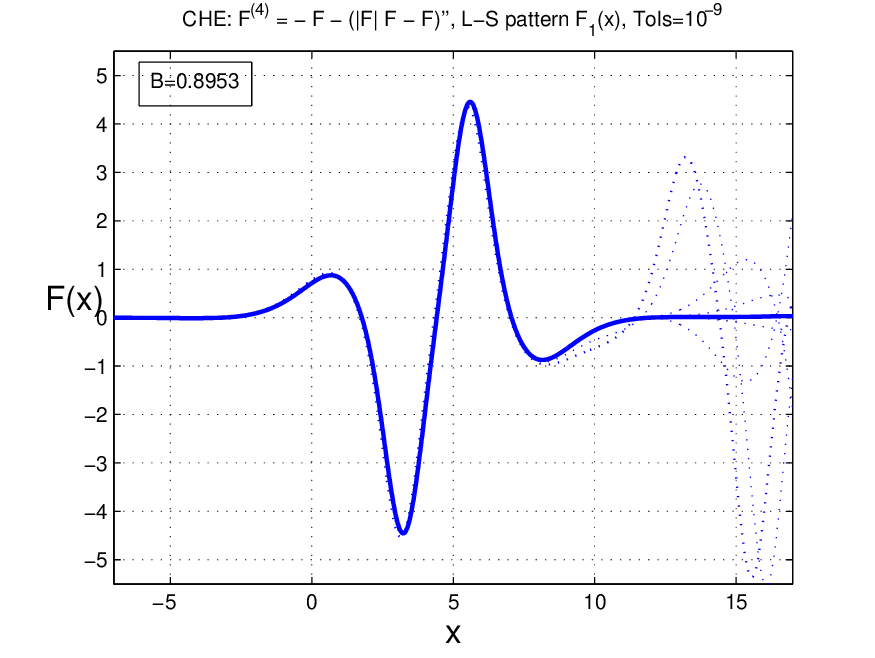}
    \caption{$F^{(4)}=-F -(|F|F-F)''$:
 the second basic L--S pattern $F_1(x)$, $B=0.8953$.}
    \label{Mod2}
\end{minipage}
\end{figure}

\begin{figure}[htbp]
 \hfill   \hfill  \begin{minipage}[t]{0.45\textwidth} 
    \centering 
    \includegraphics[width=7cm,keepaspectratio]{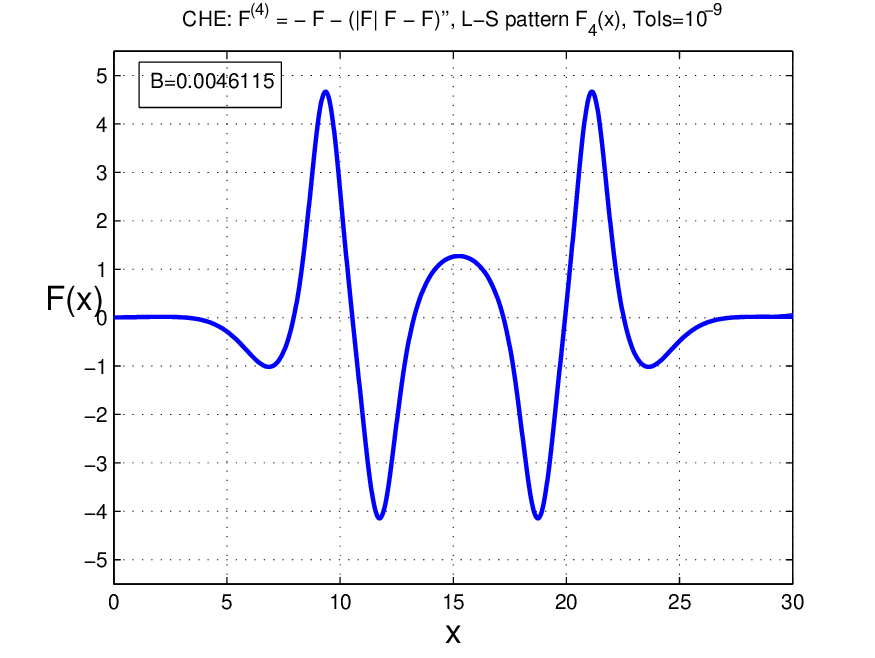}
   \caption{$F^{(4)}=-F -(|F|F-F)''$:
 the third basic L--S pattern $F_2(x)$, $B=0.0046115$.}
    \label{Mod3}
\end{minipage}
  \hfill   \hfill 
\begin{minipage}[t]{0.5\textwidth}
    \centering 
    \includegraphics[width=7cm,keepaspectratio]{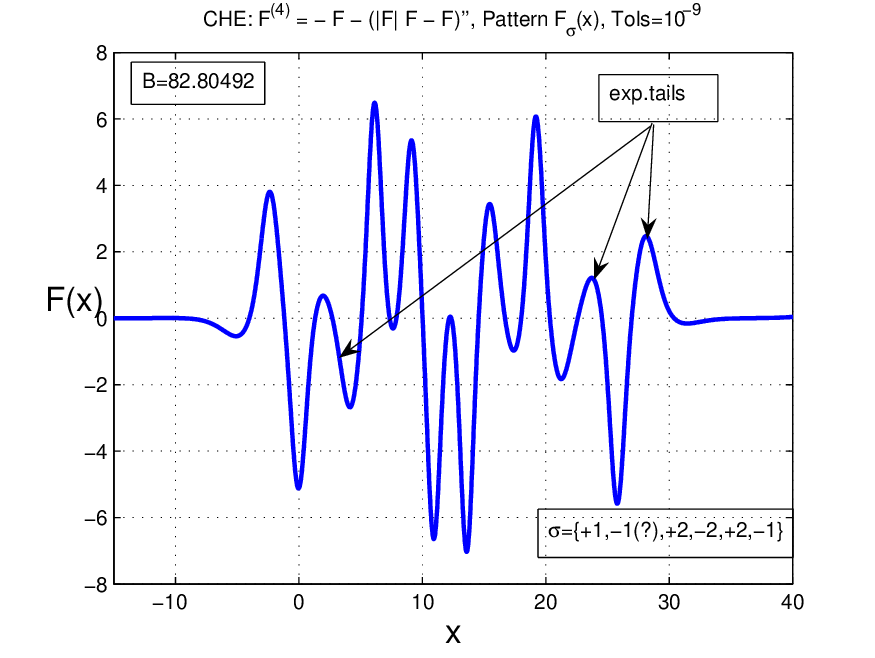}
    \caption{$F^{(4)}=-F -(|F|F-F)''$:
 a pattern $F_\sigma(x)$,
 $\sigma \sim \{+1,-1, ({\rm tail}),+2,-2,+2,-1\}$,
  $B=82.80492$.}
    \label{Mod4}
\end{minipage}
\end{figure}

\begin{figure}[htbp]
 \hfill   \hfill  \begin{minipage}[t]{0.45\textwidth} 
    \centering 
    \includegraphics[width=7cm,keepaspectratio]{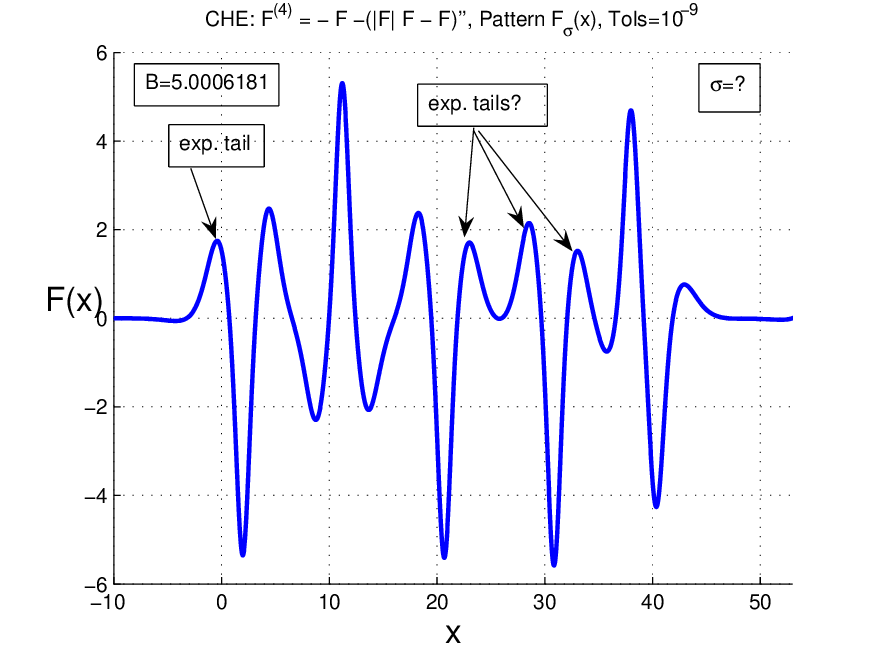}
   \caption{$F^{(4)}=-F -(|F|F-F)''$:
 a complicated pattern $F_{\sigma_1}(x)$, $B=5.00081781$.}
    \label{Mod5}
\end{minipage}
  \hfill   \hfill 
\begin{minipage}[t]{0.5\textwidth}
    \centering 
    \includegraphics[width=7cm,keepaspectratio]{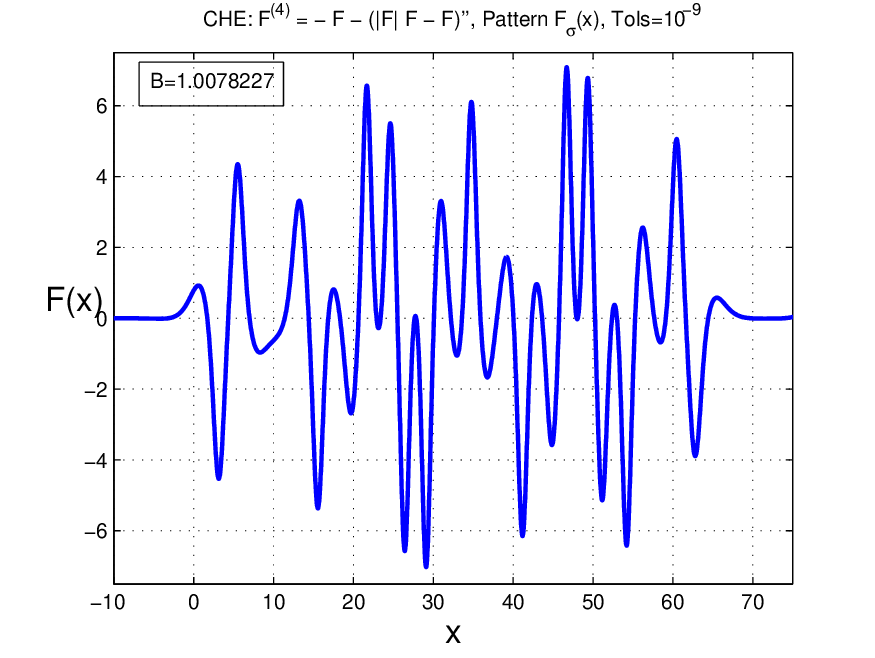}
    \caption{$F^{(4)}=-F -(|F|F-F)''$:
 another complicated pattern $F_{\sigma_2}(x)$, $B=1.0078227$.}
    \label{Mod6}
\end{minipage}
\end{figure}

\subsection{Patterns for a 
 H\"older continuous non-variational ODE}
  \label{SHold}

Finally, let us briefly discuss the exotic  ODE (\ref{Hold1}).
Patterns are not exponentially decaying as
$x \to \iy$ and are compactly supported with the following
behaviour at the support end points:
 \be
 \label{Hol2}
 {\rm supp} \,F(x)=(x_0,x_1), \,\,F(x)= (x-x_0)^8
 \varphi_*(\ln(x-x_0)+s_0)+..., \,\, x \to x_0^+, 
   \ee
  where $s_0 \in \re$ is a translational constant for solutions of a nonlinear autonomous ODE for $\varphi(s)$.  
 A similar behaviour occurs at the second end-point of the support, as $x \to x_1^-$.
In \ef{Hol2}, the {\em oscillatory component} $\varphi_*(s)$ is a periodic
function, so that \ef{Hol2} represents a 2D stable manifolds with
two parameters $x_0,s_0 \in \re$. This kind of oscillatory
behaviour at  finite interfaces replacing the exponential decay
for smooth $C^1$-nonlinearities were studied in \cite{EGK1,EGK2},
see also  \cite[Ch.~1]{GMPBook}, so 
we present a few figures.

 The oscillatory sign-changing  kind of behaviour allows to apply the same as above 
 procedure to a ``straightening" to an inflection point of asymptotic positive/negative humps 
 of the  expansion. Since the expansion (\ref{Hol2}) is nonlinear and is not known rigorously,
 a proof is difficult to justify. However,
basic patterns, more or less, are shown to remain of the
same geometric structure, see Figs \ref{FigHol1}--\ref{FigHol3} below. 
For any pair of patterns with disjoint supports one can see an uncountable subsets of patterns by  moving them
independently  without any supports overlapping.   

In numerics, we do not mimic a complicated manifold 
\ef{Hol2} and  at $x=x_0=-20$ we put
 $$
 F(x_0)=F'(x_0)=F''(x_0)=0, \andA F'''(x_0)=B,
 $$
 where $|B| \ll 1$ is an analogy of the constant $B$ in the
 exponential tail \ef{OneP1} for the analytic decay.
 A better use of the envelope $\sim (x-x_0)^8$
 for these boundary conditions (say at some $x=x_0+10^{-2}$)
 is not necessary and does not give any improvements.

 \begin{figure}[htbp]
 \hfill   \hfill  \begin{minipage}[t]{0.45\textwidth} 
    \centering 
    \includegraphics[width=7cm,keepaspectratio]{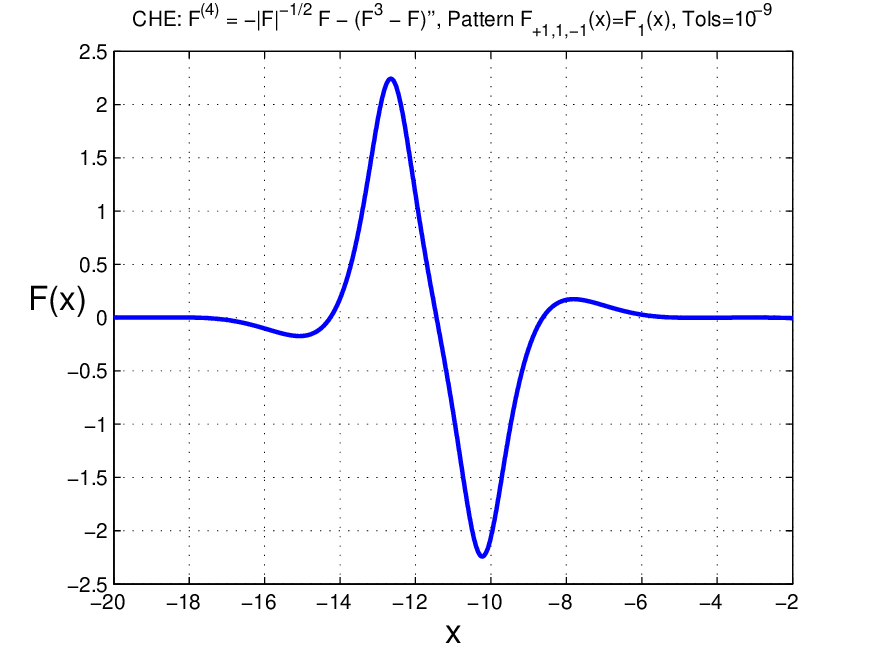}
   \caption{Equation (\ref{Hold1}):
  pattern
$F_1(x)= F_{+1,1,-1}(x)$, $B=0.003977866$.}
    \label{FigHol1}
\end{minipage}
  \hfill   \hfill 
\begin{minipage}[t]{0.5\textwidth}
    \centering 
    \includegraphics[width=7cm,keepaspectratio]{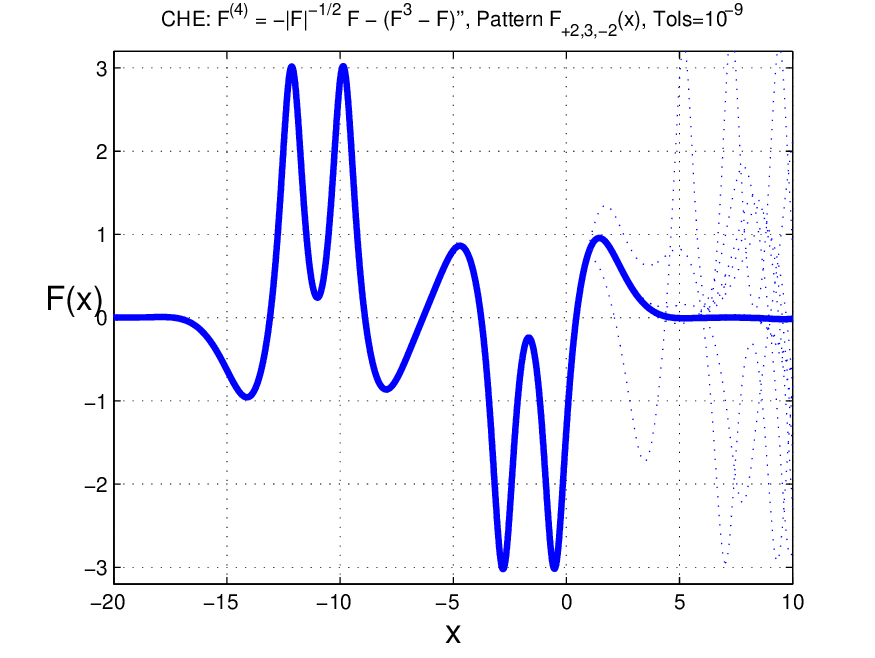}
    \caption{Equation (\ref{Hold1}):
 pattern
$F_{+2,3,-2}(x)$, $B=0.0100004165$.}
    \label{FigHol2}
\end{minipage}
\end{figure}

\begin{figure}[htbp]
\begin{center}
\includegraphics[scale=0.52]{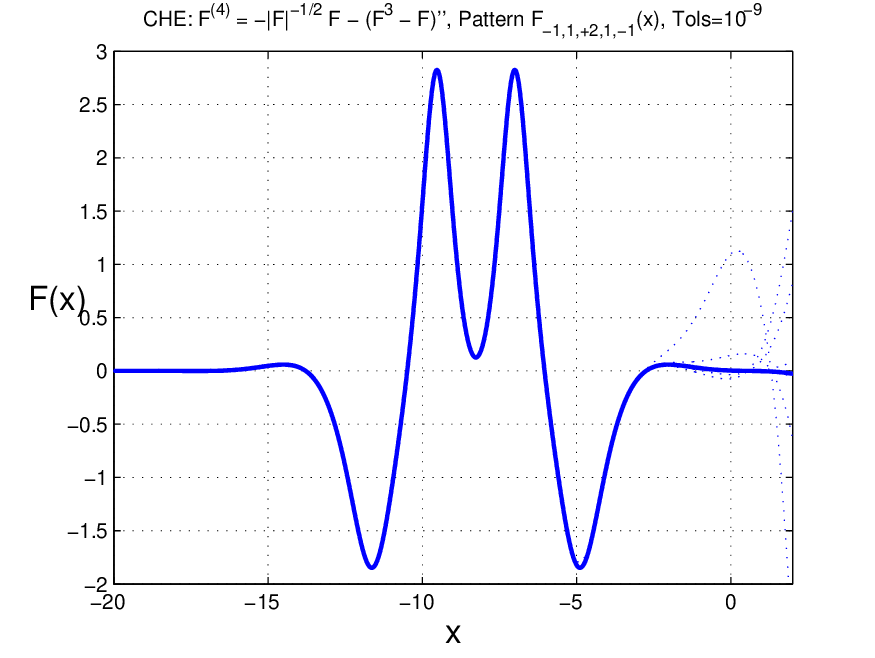}
\caption{Equation (\ref{Hold1}):
pattern
$F_1(x)= F_{-1,1,+2,1,-1}(x)$, $B=0.01584537$.}
  \label{FigHol3}
\end{center}
\end{figure}

{\bf Remark: small deformations of $-F$ keep the exponential tail.} Such a curious example of a H\"older continuous nonlinearity
in \ef{Hold1} instead of the standard $-F$ in  the ODEs  which changes  exponential tails into  the oscillatory \ef{Hol2} at finite interface points inspires  another (positive) example.
For instance, consider the ODE with another  perturbed odd term $-F$:
 \be 
 \label{CH123}
 F^{(4)}= -\kappa(F) +F^2, \quad \mbox{where} \quad \kappa(F)=F|\ln|F||\,\,\mbox{for small} \,\,|F| \ge 0,
 \ee
 where $\kappa(F)$ is locally H\"older continuous, for any exponent $\a \in (0,1)$, and $\kappa'(0)= \iy$.
   However  transformed oscillatory  exponential tails exist and are given by
  $$
  \tex{
  F(x) \sim {\rm e}^{-\nu |x|^{4/3}} \cos(\nu |x|^{4/3}) \asA x \to \iy, \quad \nu= \frac 1{\sqrt 2}(\frac 34)^{4/3},
  }
  $$
 so that matching of various patterns is allowed almost in the same manner as for the standard quadratic ODE.
 It is convenient  that the leading term of such tail expansions  is given explicitly.

\section{On extensions to other indefinite operators}
 \label{SectExp4}

\subsection{Fourth-order ODE with an indefinite exponential
nonlinearity}

 We continue to present similar gluing/matching patterns for other ODEs with indefinite
 operators beginning with the following exponential one:
 \be
 \label{exp41}
 F^{(4)}=-F+F^2 {\mathrm e}^{F-1} \quad (\exists \,\mbox{equilibria} \,\, F_*=0\,\, \mbox{and} \,\,1).
  \ee
Fig. \ref{FigF1} shows standard $F_0$, $F_2$, and $F_{+4}$
patterns for \ef{exp41}.

\begin{figure}[htbp]
\begin{center}
\includegraphics[scale=0.52]{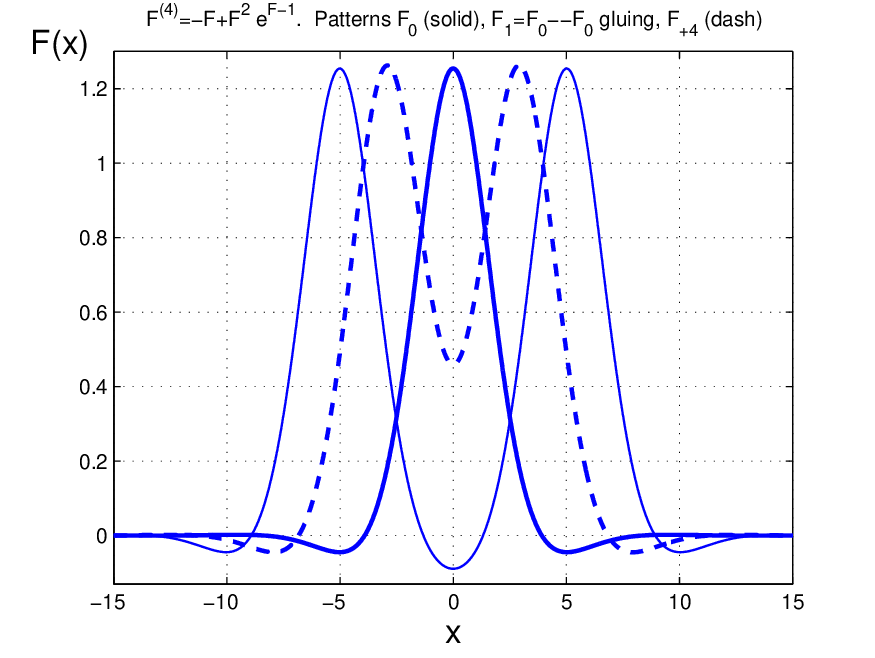}
\caption{Three typical solutions of \ef{exp41}.}
\label{FigF1}
\end{center}
\end{figure}

\subsection{On a ``polynomial" equation $F^{(4)}=\pm F^2+F^3$ and similar}

As we have seen from the ODE \ef{Hold1} with a H\"older continuous term,  
existence of various patterns by matching requires oscillatory tails of possible solutions at the end points
or at least a 2D asymptotic manifold.
Here we consider two cases of simple polynomial-like nonlinearities without a linear term $-F$ in the ODEs.

\smallskip

{\sc  (i) Patterns for $...=-F^2$.}  
Obviously,   a local nonexistence  asymptotic result happens if
 \be 
 \label{non78}
 \mbox{a 2D manifold of solutions $F(x) \to 0$ as $x \to \iy$ is unavailable.}
  \ee
For instance, in the case of a quadratic term $=-F^2+...$ instead of the standard $=-F+...$:
 \be
 \label{QQ1}
 \tex{
 F^{(4)}= -F^2+F^3 \quad(\Longrightarrow \int_\re F^3=\int_\re F^2>0),
 }
 \ee
 most initial data lead to  BVP solutions.
In Figures \ref{f7-24} and \ref{plusF2} we present
typical numerical samples for \ef{QQ1}. Note that \ef{QQ1} admits an algebraic decay via
$$
F(x)=x^{-4}\varphi(s)+...\,, \,\,\, s=\ln x, \,\,\,x \to +\iy,
$$
where $\varphi(s)$ is a periodic or any suitable global solutions in $\re$ of an autonomous ODE, see below.

{\sc  (i) Patterns for $...=F^2$.}   On the other hand, Figure \ref{plusF2} also shows  a pattern for the opposite sign in the term, $+F^2$:
\be
\label{plus11}
\tex{
F^{(4)}=F^2+F^3\quad(\Longrightarrow \int_\re F^3=-\int_\re F^2 < 0),
}
\ee
where a suitable asymptotics is available and obviously
\be  
\label{plus12}
F(x) \sim 840 \, x^{-4}+... \asA x \to +\iy.
\ee
This manifold is 2D (recall the translation $x \mapsto x+x_0$) and hence can be used to match/glue other patterns into more complicated ones in $\re$. Note that the expansion $-F(x)$ in \ef{plus12} works for \ef{QQ1}.

In Figure \ref{plus22} we show the first pattern $F_0(x)$ in $\re$ for \ef{plus11} so that $-F_0(x)$ satisfies \ef{QQ1}.

\begin{figure}[htbp]
\begin{center}
\includegraphics[scale=0.42]{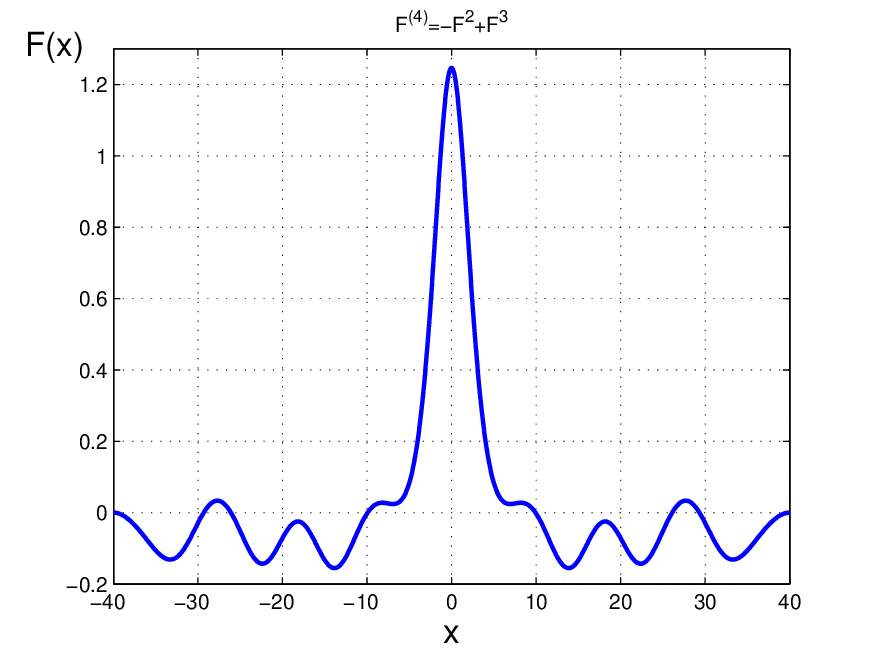}
\caption{A typical non-decaying solution of
\ef{QQ1}.} \label{f7-24}
\end{center}
\end{figure}

\begin{figure}[htbp]
\begin{center}
\includegraphics[scale=0.32]{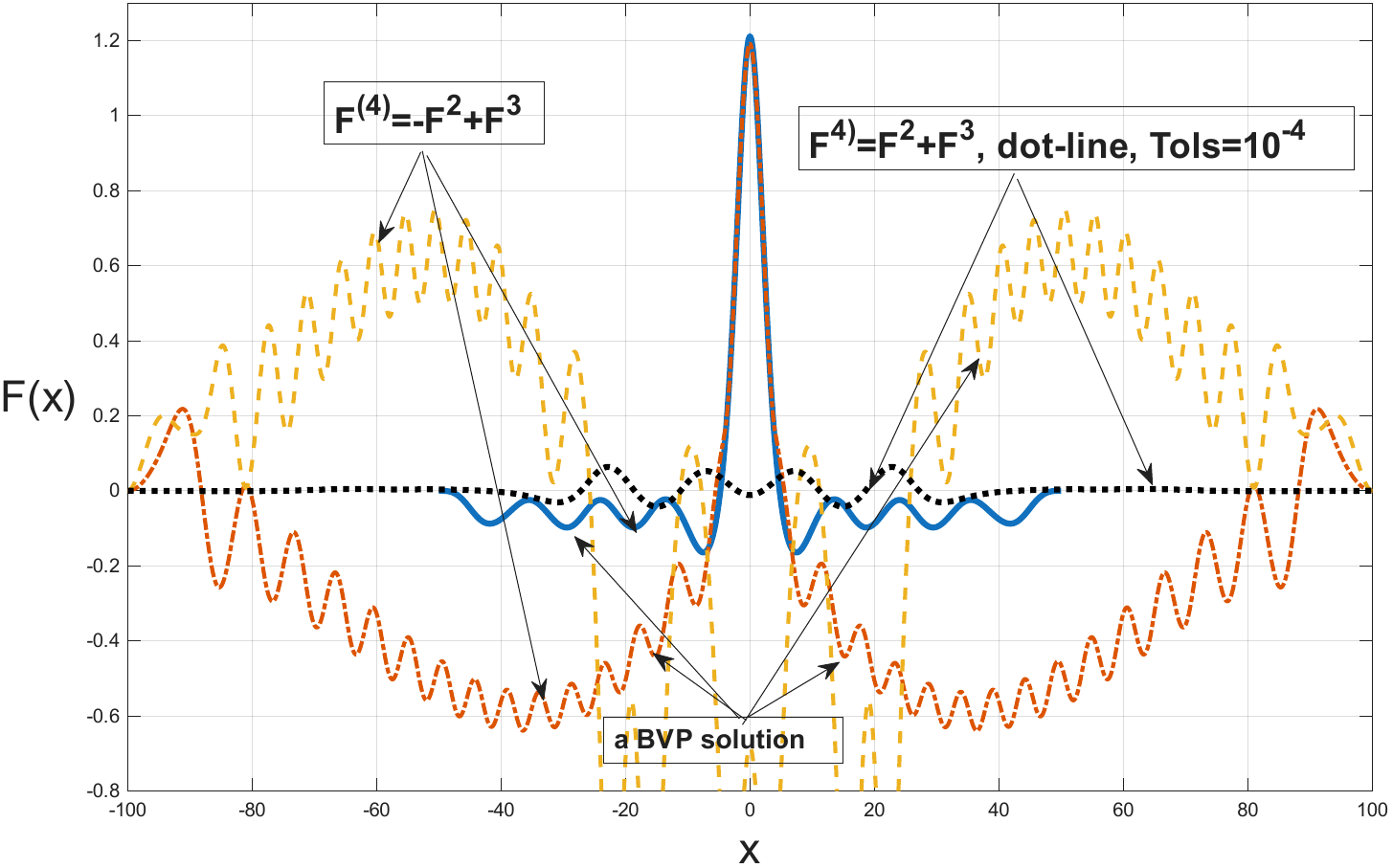}
\caption{Typical solutions of $F^{(4)}=\pm F^2+F^3$.}
\label{plusF2}
\end{center}
\end{figure}

\begin{figure}[htbp]
\begin{center}
\includegraphics[scale=0.32]{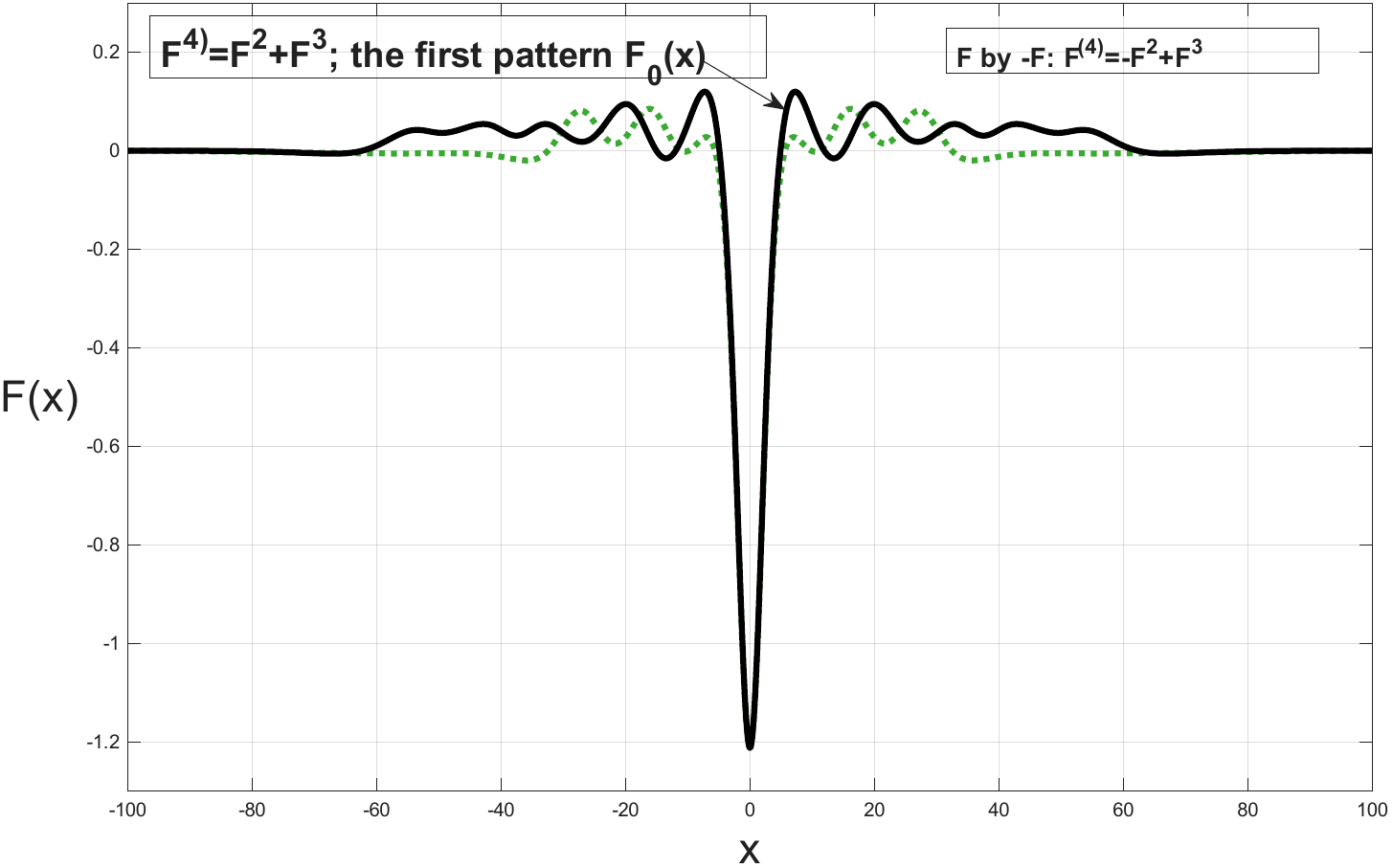}
\caption{The first pattern $F_0(x)$ for  $F^{(4)}=F^2+F^3$.}
\label{plus22}
\end{center}
\end{figure}

A similar ``asymptotic nonexistence" (for some initial data) is presented for a 
fifth degree polynomials:
\be
 \label{QQ2}
 F^{(4)}= -F^2+F^5 \quad \mbox{and} \quad  F^{(4)}=-F^4+F^5;
 \ee
see Figures \ref{f7-25} and \ref{f7-26} where clearly BVP patterns appear.

\begin{figure}[htbp]
 \hfill   \hfill  \begin{minipage}[t]{0.40\textwidth} 
    \centering 
    \includegraphics[width=7cm,keepaspectratio]{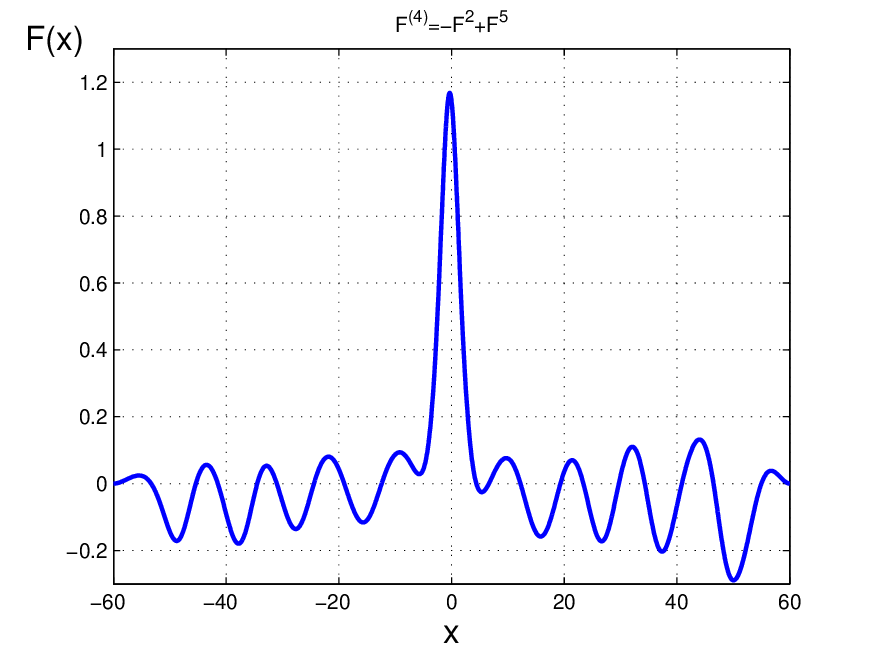}
   \caption{A typical non-decaying (BVP) solution of the first ODE in
\ef{QQ2}.}\label{f7-25}
\end{minipage}
  \hfill   \hfill 
\begin{minipage}[t]{0.5\textwidth}
    \centering 
    \includegraphics[width=7cm,keepaspectratio]{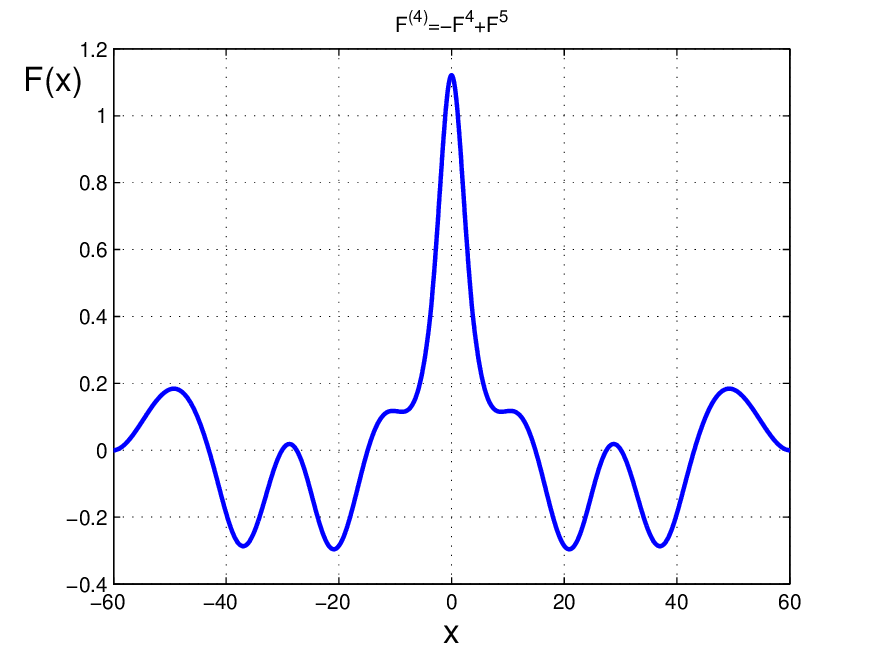}
   \caption{A typical non-decaying (BVP) solution of the second ODE in
\ef{QQ2}.}\label{f7-26}
\end{minipage}
\end{figure}

\smallskip

Observe that  most of the local orbits of \ef{QQ1} and \ef{QQ2} blow-up at some finite $x \to x_0^-$ and $F(x)$ becomes unbounded nearby:
$$
\begin{matrix}
F^{(4)}=F^3+... \LongA F(x) =\pm (x-x_0)^{-2} \varphi_*(\ln(x-x_0)+s_0),
\\
F^{(4)}=F^5+... \LongA F(x)= \pm  (x-x_0)^{-1} \varphi_*(\ln(x-x_0)+s_0), \quad x_0, \, s_0 \in \re,
\end{matrix}
 $$
 where $\varphi_*(s)$ is a periodic sign-changing solution of the corresponding ODEs of the type \ef{BB6} (or  \ef{Mod68} below) is derived similarly.

\subsection{Existence:  equations with nonlinear algebraically decaying tails}

We show that proper patterns exist in the case of polynomial nonlinearities under the condition that
 \be 
 \label{Lead33}
 ...=-F+... \,\,\, \mbox{is replaced by a {\em monotone} function as $F \to 0$}.
 \ee 
  
  \begin{enumerate}
\item[(I)] {\sc  Existence for $-|F|F$.}
Replacing $-F^2$ by a monotone quadratic $-|F|F$
yields 
 \be 
 \label{Mod66}
F^{(4)}=-|F|F+ F^4 \quad  \mbox{in} \quad \re, \quad F(\iy)=0 \quad (\mbox{a ``bi-quadratic" ODE}).
 \ee
 The leading term, as $F \to 0$, creates the following algebraic tails of  solutions, as $x \to +\iy$:
  \be 
  \label{Mod67}
 F^{(4)}=-|F|F \LongA F(x) = x^{-4} \varphi(s+s_0) \to 0, \quad s= \ln x \to +\iy, \,\,\, s_0 \in \re,
  \ee
  where the {\em oscillatory component} $\varphi(s)$ solves
    \be 
    \label{Mod68}
   \varphi^{(4)} -22 \varphi''' + 179 \varphi'' - 638 \varphi' +
  840 \varphi = - |\varphi|\varphi.
   \ee
  Compare it with \ef{BB6} for blow-up phenomena containing  similar operators but having different existence conclusions.   
  As we have mentioned,  ODEs like \ef{Mod68} are known to admit a periodic sign-changing solution $\varphi_*(s)$ \cite{EGK1,EGK2}, unlike \ef{BB6} in our blow-up analysis
   in Section \ref{S2N}.
  Therefore the oscillatory character of asymptotics (\ref{Mod67}), though not given explicitly, technically allows us 
 to apply matching/gluing approaches to create similar countable subsets of various patterns.
  
  The operator in \ef{Mod66} is indefinite and  the L--S family of critical points is nonexistent. 

\item[(II)] {\sc Algebraic tails for $-F^3$.} A similar tail occurs for the cubic nonlinearity
\be 
\label{mCU1}
F^{(4)}=-F^3 + F^4 \LongA F(x) \sim x^{-2} \varphi(s+s_0) \to 0, \quad s= \ln x \to +\iy, \,\,\, s_0 \in \re,
 \ee
where a periodic  $\varphi(s)$ solves a similar to \ef{Mod68} ODE with $-\varphi^3$ on the right-hand side. 
 
 Thus, both ODEs \ef{Mod66} and \ef{mCU1} and many others with such  ``fully nonlinear" operators 
satisfying \ef{Lead33} admit a similar variety of homoclinic orbits of O creating a similar attractor $W^{2,\iy}$.
The only difference is that the deficiency $(2,2)$ (a ``saddle-node")  linearized behaviour close to the origin O 
is replaced by the nonlinear one, as in \ef{Mod67}, and \ef{mCU1}. A full extension of pattern formation requires a more careful analysis of the existence of a nontrivial oscillatory  sign changing periodic solutions $\varphi_*(s)$ of related ODEs like \ef{Mod68}, and then we arrive at a two-parametric family of algebraically decaying tails
$\{F(x+x_0), x_0 \in \re\}$ containing also the second  parameter $s_0 \in \re$ of the $\ln x$-translation. In other words, we keep the values $(2,2)$ of the ``nonlinear" deficiency indices, so that these 2D asymptotic oscillatory  tails allow  
us to perform a similar classification of countable pattern subsets.  
 \end{enumerate}

\subsection{Algebraic oscillatory interface tails for a quasilinear degenerate ODE}

As a typical but a difficult example, we take the first ODE in \ef{TFE2} in the form (it is $L^2$-variational)
 \be 
 \label{Nn1}
 (|F''|^n F'')''=-F + F^2, \quad \mbox{where $n >0$ is a fixed parameter.}
\ee 
Since the differential operator is degenerate at $\{F''=0\}$, \ef{Nn1} is understood in a  weak sense and solutions are linear functionals in $C_0^\iy(\re)$, $|F''|^n F'' \in L^1_{\rm loc}(\re)$ (an estimate to be derived separately), {\em etc.} We first  need to reveal the oscillatory properties and simultaneously their actual regularity at the end points of solutions finite support. Let $F(x)$ be compactly supported on an interval $(x_0,x_1)$. Then checking its behaviour as $x \to 0^+$ we perform the same blow-up scaling for small $|F| \ge 0$, i.e., as $x \to x_0^+$: 
\be 
\label{Nn2}
\tex{
(|F''|^n F'')''=-F \Longrightarrow F(x)= (x-x_0)^a \varphi(s), \,\, s = \ln (x-x_0) \to - \iy, \,\, a= \frac {2(n+2)}n,
  }
 \ee
 where the oscillatory component $\varphi(s)$ satisfies a harder than usual  $s$-autonomous ODE
  $$
  {\rm e}^{-s}\{ {\rm e}^{-s}[{\rm e}^{(a-2)(n+1)s} | \varphi''+(2a-1)\varphi'+a(a-1)\varphi|^n 
   (\varphi''+(2a-1)\varphi'+a(a-1)\varphi)]'\}'=-{\rm e}^{a s},
 $$
 since the exponential multipliers  $\sim {\rm e}^{s}$ cancel each other: $(-1)+(-1)+(a-2)(n+1)=a$. Such ODEs for $\varphi$ are known to admit an oscillatory sign-changing periodic solution $\varphi_*(s)$; see similar thin film ODEs in \cite{EGK1,EGK2}. Moreover, in these papers, it is shown that the  necessary uniqueness of $\varphi_*$ can be traces out by passing to the limit $n \to 0^+$. Namely, we observe an actual convergence to the linear exponential tails for $n=0$: using a singular boundary layer approach, in a natural rescaled sense, as $n \to 0^+$,
  $$ 
   F(x)=(x-x_0)^a \varphi_*(\ln (x-x_0)+s_0),\,\, x \approx 0^+
 \quad \mbox{``approaches"} \quad  C_0 {\rm e}^{\mu x} \cos(\mu x + C_1), \,\,  x \ll -1,  
   $$ 
 where as always $\mu = \frac 1{\sqrt 2}$.
   Both manifolds  are 2D: two arbitrary constants $(x_0,s_0)$ in the former and two $(C_0,C_1)$ in the latter
 (by $x$-translation these are 1D). Since for $n=0$ the uniqueness of the periodic orbit is given and obvious, we expect that $\varphi_*(s)$ is unique for $n>0$ by continuity. For small $n>0$ this can be proved by a perturbation argument and a further extension to $n=1$ is expected but is not established rigorously. Thus, as above, we observe a 2D oscillatory manifolds as $n \to 0$ (a ``nonlinear deficiency indices" $(2,2)$). This allows us to apply matching/gluing methods to reveal pattern subsets and eventually the corresponding  chaotic  attractor $W^{2,\iy}$ of homoclinics of \ef{Nn2}.
  
  Curiously, it follows from the generic expansion \ef{Nn2} near finite interfaces, that those weak compactly supported weak solutions are actually the  classic ones therein if
   $$
   \tex{
 F \in C^4: \quad   a= \frac{2(n+1)}n > 4, \quad \mbox{i.e., for any} \quad n \in (0,2).
   }
   $$
At other degenerate points where $F''=0$ but $F \not =0$ the regularity is worse: 
$$
\tex{
F \in C^{3,\d}\quad \hbox{with}\quad  
\d=\frac {1-n}{1+n} \in (0,1)\quad \hbox{for}\quad n \in (0,1).
}
$$

\section{Semilinear six-order ODE operators}
\label{SSix}

\subsection{Sixth-order ODE with an indefinite exponential
nonlineairty}

The corresponding sixth-order ODE with an exponential non-odd
nonlinearity is
 \be
 \label{exp42}
 F^{(6)}=F-F^2 {\mathrm e}^{F-1} \quad (\exists \, F_*=0, \, 1).
  \ee
The corresponding complicated pattern glued from five  variously
$x$-distributed  $F_0$'s and a single $F_{+2}$ is shown in Fig.
\ref{FigF2}. 
 Fig. \ref{FigF3} shows a typical basic $F_{+6}$ pattern, while Fig. 
 \ref{FigF4} reminds a possibility to glue more extended
 patterns, whose zero sets are shown in Fig. \ref{FigF5},

\begin{figure}[htbp]
 \hfill   \hfill  \begin{minipage}[t]{0.40\textwidth} 
    \centering 
    \includegraphics[width=7cm,keepaspectratio]{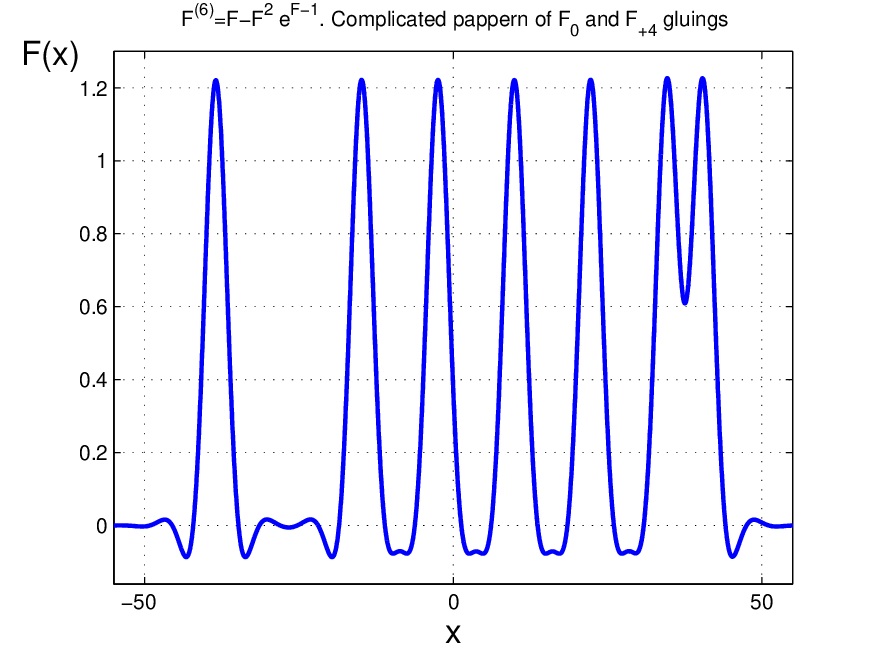}
   \caption{A complicated  solution of \ef{exp42}.}
\label{FigF2}
\end{minipage}
  \hfill   \hfill 
\begin{minipage}[t]{0.5\textwidth}
    \centering 
    \includegraphics[width=7cm,keepaspectratio]{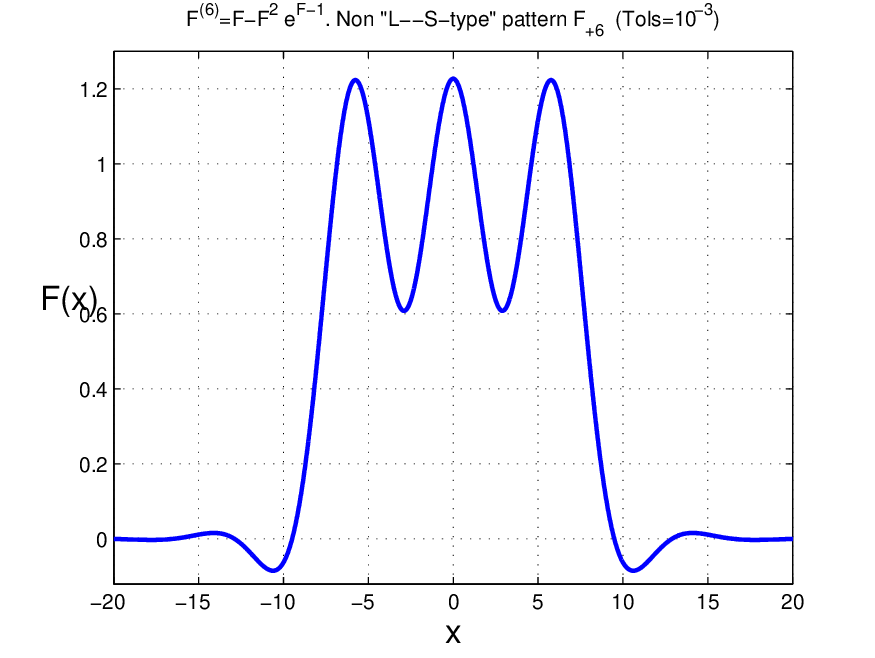}
  \caption{$F_{+6}$  solution of \ef{exp42}.}
\label{FigF3}
\end{minipage}
\end{figure}


\begin{figure}[htbp]
 \hfill   \hfill  \begin{minipage}[t]{0.40\textwidth} 
    \centering 
    \includegraphics[width=7cm,keepaspectratio]{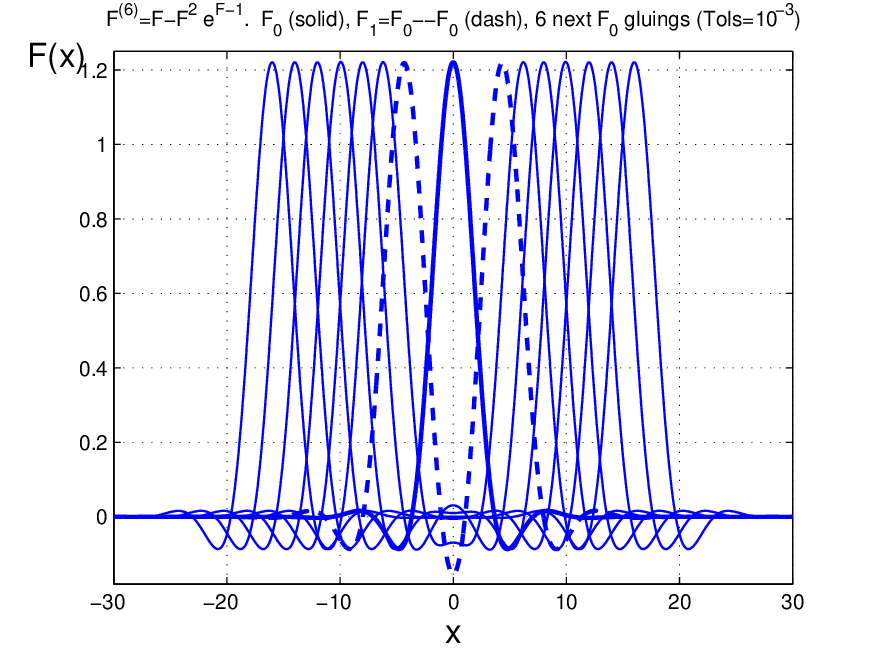}
    \caption{Three different  solutions of \ef{exp42}.}
\label{FigF4}
\end{minipage}
  \hfill   \hfill 
\begin{minipage}[t]{0.5\textwidth}
    \centering 
    \includegraphics[width=7cm,keepaspectratio]{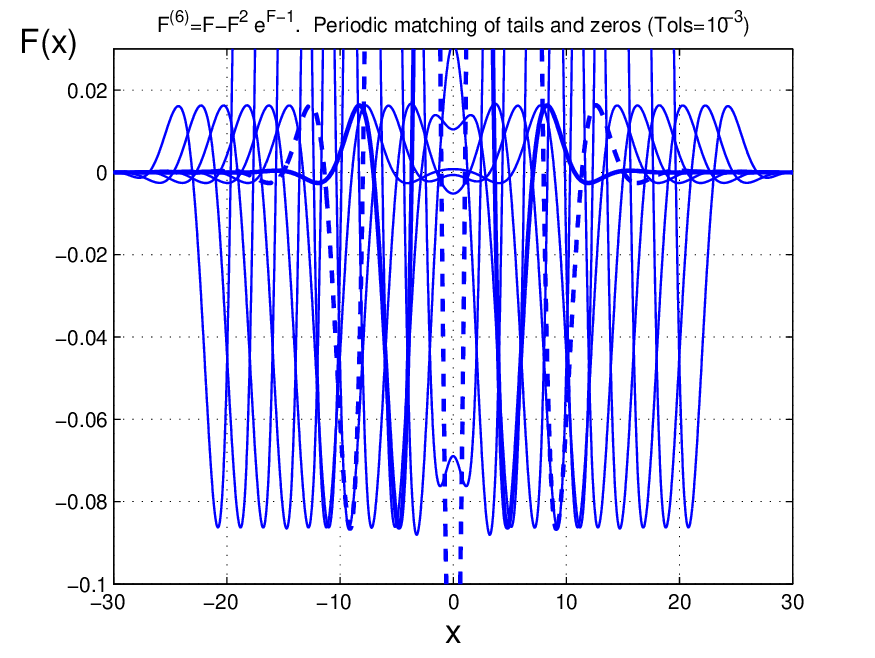}
 \caption{Enlarged zero sets of  these three patterns of \ef{exp42}.}
\label{FigF5}
\end{minipage}
\end{figure}




\subsection{A few comments on a sixth-order ODE with a nonlinearity
of a ``Black Jack" (21st) degree}

Finally, as the last somehow extended example, we consider an
exotic sample with an operator with odd nonlinearity having the 21st degree
 \be
 \label{BJ1}
 \tex{
 F^{(6)}=F -F^{21} \LongA \Phi(F)= \int [(F''')^2 + F^2]- \frac 1{22} \, \int  F^{22}, \quad F \in H^3 \cap L^{22}.
  }
  \ee
Then a standard L--S countable sequence of minmax critical points
  on the proper $H_0$ exists corresponding to increasing genus (category),
   and gluing/matching techniques for new patterns remain
  similar.

 Figure \ref{21.1} shows standard patterns $F_0$, $F_1$ from the first basic L--S family ${\mathcal F}_1$ and $F_{+6}$
from the non L--S family ${\mathcal F}_2$ for \ef{BJ1}. The next  Figure \ref{21.2} describes gluing
 two $\pm F_{+4}$ patterns to create $F_{+4,5,-4}$ from ${\mathcal F}_2$. We again
 underline a clear periodic linearized structure of matched $\pm
 F_{\pm 4}$ profiles in the gluing area near the origin $x=0$,
 where just two periods of  exponential tails therein with proper related shifting distances
 $a_5$'s
 are required. Of course, for locally creating an odd pattern
one should require two other conditions
 $$
F_\s=F_\s''=F_\s^{(4)}=0 \quad \mbox{at the matching point}.
  $$
 At the same time, note that this is not a first matching of $\pm
 F_{\pm 4}$, and not a second, but actually a sixth one.
  The first one with $a_0=0$, creates a single transversal zero in
 between those two structures and leads to the pattern
 $F_{+4,1,-4}$.

\begin{figure}[htbp]
\begin{center}
\subfigure[Patterns $F_{0}$ and $F_1$]{
\includegraphics[scale=0.52]{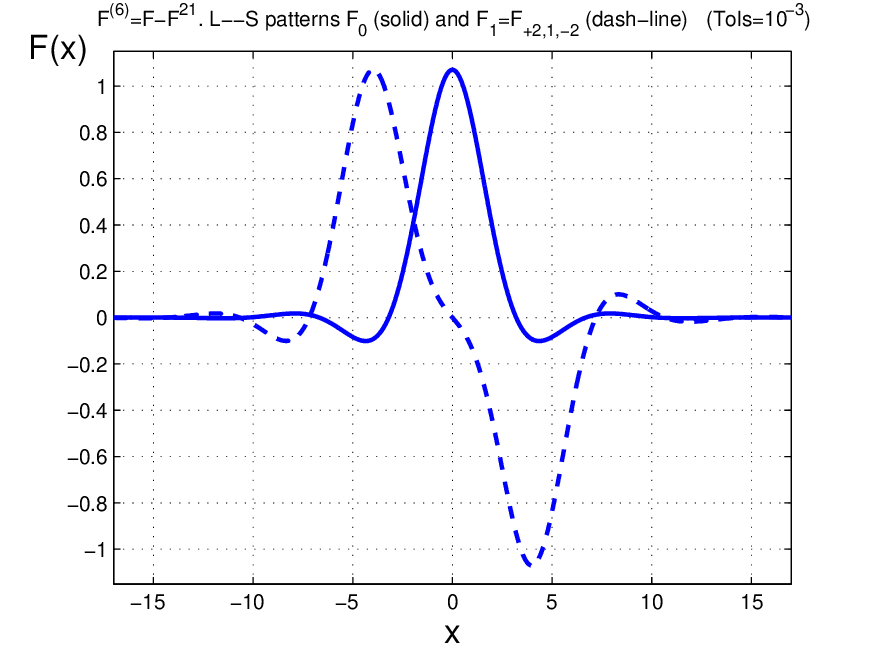} 
} \subfigure[$F_0(x)$ and $F_{+6}(x)$]{
\includegraphics[scale=0.52]{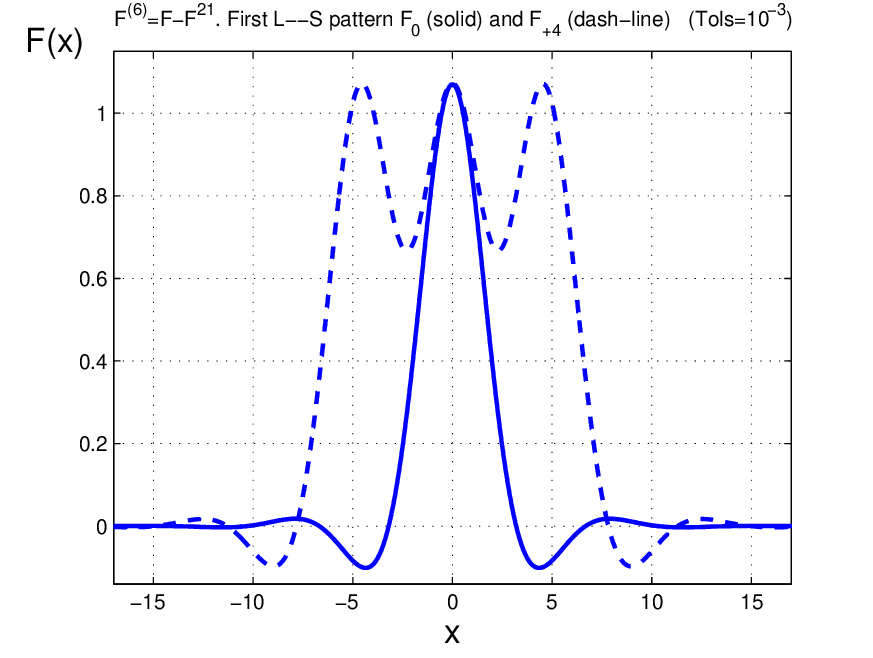} 
}
\caption{Standard
(``dipole", a dash line) patterns of the S--L
type, and non S--L pattern $F_{+6}$ (dash line) of the ODE
\ef{BJ1}.}
\label{21.1}
\end{center}
\end{figure}

\begin{figure}[htbp]
\begin{center}
\includegraphics[scale=0.55]{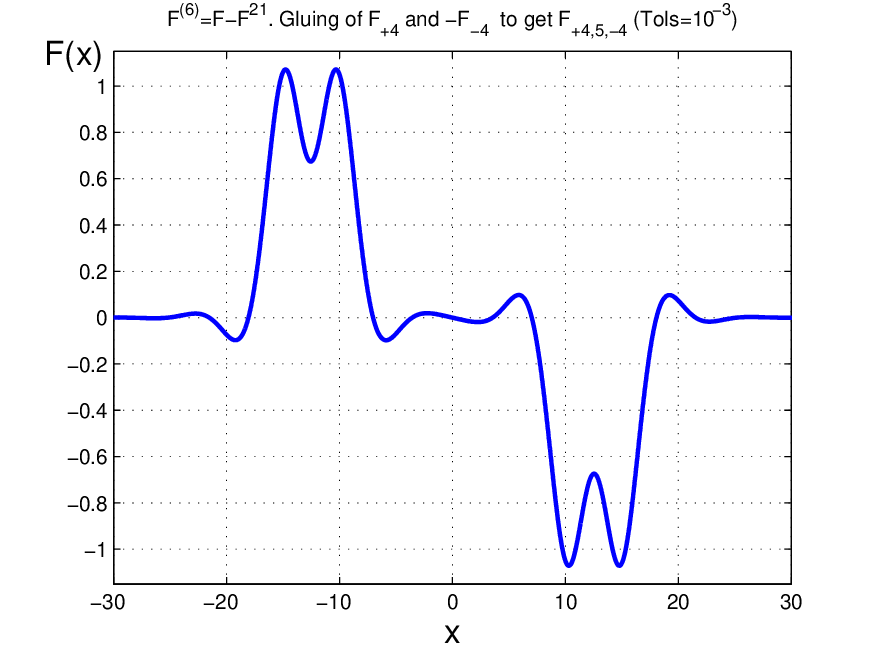} 
\caption{Gluing the pattern $F_{+4,5,-4}$.}
\label{21.2}
\end{center}
\end{figure}

\section{Semilinear elliptic bi-harmonic  equations in $\ren$}
 \label{SectRN1}


We now begin to discuss related nonlinear {\em elliptic} problems.

\subsection{The cubic equation with odd nonlinearities}

First, we consider the elliptic equation with odd nonlinearity
 \be
 \label{ell1}
 \D^2F=-F+F^3 \inB \ren, \quad F(x) \to 0 \asA x \to \iy.
  \ee
The linearized analysis about $F=0$ for the  linear equation
\be
\label{ell2}
 \D^2 F=-F \inB \ren \forA |x| \gg 1
 \ee
 can be partially  performed by a standard  method of separation of variables for linear PDEs giving leading radially
 exponentially decaying solutions (as for $N=1$) plus their 
 angular  distribution via  superpositions of the Laplace--Beltrami operator $\D_\s$:
  \be 
  \label{Belt1}
  \tex{
  \D=\D_r + \frac 1{r^2} \D_\s \LongA
  \D^2=(\D_r + \frac 1{r^2} \D_\s )^2= \D^2_r + \frac 1{r^2} \D_r \D_\s + \D_r \frac 1{r^2} \D_\s + \frac 1{r^4} \D_\s^2.
 }
 \ee 
The Laplace-Beltrami operator $\D_\s$ on the unit
sphere $S^{N-1}$ in $\ren$ is a regular operator with discrete
spectrum in $L^2(S^{N-1})$ (each one repeated as many times as its
multiplicity)
$\s(\D_{\s}) = \{\l_k=-k(k+N-2), \,\, k \ge 0\}$,
and an orthonormal, complete, closed subset $\{f_k(\s)\}$ of
eigenfunctions which are homogeneous harmonic $k$-th order
polynomials restricted to $S^{N-1}$. 
 Solutions of \ef{ell2}
 are exponentially decaying at infinity (see Section \ref{SectRad4}):
  \be 
  \label{As00}
  F(x) = O\big(|x|^{-\frac {N-1}2} {\rm e}^{-{|x|}/{\sqrt 2}}\big) \quad \mbox{as} \quad x \to \iy. 
  \ee

The problem \ef{ell1} is variational  and solutions $F \ne 0$ can be studied as in 1D b
as  critical points as follows:
 \be
 \label{ell31}
\begin{matrix}
 \Phi(F)= \frac 12 \, \int [(\D F)^2 + F^2] - \frac 14 \,
 \int F^4 \inB H^2(\ren), 
 \\
 F=r_0(v)v
\,\, \mbox{on the manifold} \,\,
 v\in  H_0= \big\{ \int [(\D F)^2 + F^2]=1 \big\},
 \\
\mbox{where}\,\,  r_0(v)= \frac 1{\sqrt{\int v^4}}, \,\,\,\Phi(r_0(v)v)= \frac 1{4 \int v^4},
\,\,c_k=\int v_k^4, \,\,\, C_k= \frac 1{4c_k}, \,\, k \ge 0.
    \end{matrix}
    \ee
Here we again apply the PFM \cite{Poh79, Poh08} and the L--S category theory \cite[\S~5.5]{Berger} and use the same calculus and notations as in Section \ref{S2.LS} for the quadratic equation.
   For the main references and  various applications of  the PFM and L--S genus/category theory to analogous
    nonlinear elliptic and ODE problems, see \cite[Ch.~1]{GMPBook}
  and \cite{AEGnegI}. 
  
  As the main conclusion, we obtain that \ef{ell31} admits a countable family of L--S critical points
  which as usual we denote by $\{F_k(x)\}_{k=0}^\iy$, and next we will try to  describe their properties and further extensions by using our previous experience in \cite{AEGnegI}.

 For simplicity, we restrict our attention to  the minimal dimension $N=2$ with a clearer geometric interpretation, where,
  in a ``DSs" language and in a natural sense,
  \be 
  \label{HomR2}
 \mbox{each} \,\,\, F(x) \,\,\, \mbox{represents a {\em ``homoclinic surface"} in variables $x=(x_1,x_2) \in \re^2$}.
 \ee
 We will also need to use suitable notions of   {\em ``periodic surfaces"} in $\re^2$ and others.

\smallskip

Our first suggestions (some are formal) concerning patterns  or homoclinic surfaces in $\re^2$ 
 are as follows. These patterns are  characterized by  much more complicated indices than in 1D
and we use simple ones below to underline their partial odd, even, or radial symmetry properties.

\begin{enumerate}

\item[{\bf (i)}] The first pattern $+F_0(x)$ corresponding  to the first minimum critical value $C_0>0$ ($C_*=0$ for $F=0$)  of the functional \ef{ell31} is {\em unique},   radially symmetric and satisfies  an ODE  to be treated below.
Let us mention again that these properties of the first non-zero pattern $F_0$ of the functional \ef{ell31} can be  covered by symmetrization and level set optimization approaches  in the elliptic theory \cite{GGS10},
though such applications to nonlinear elliptic problems in $\ren$ were not fully justified.

\item[{\bf (ii)}] The second pattern $F_1(x)$ 
 is expected to be not radial
and to take a form of a {\em 2D-dipole}, i.e., it can be constructed  from the first solution in the half-space $\re^2_+=\{x_1>0\}$
with the anti-symmetry conditions on the boundary
 \be 
 \label{F1Sym}
F=F_{x_1 x_1}=0 \atA x_1=0.
\ee
The whole pattern $F_1(x)$ in $\re^2$ is then obtained by the odd reflection $F_1(-x_1,x_2)=-F_1(x_1,x_2)$.
It seems that $F_1(x)$ should  approximately mimic a structure of a double-hump pattern consisting of two neighbouring one in {\bf (ii)} $\sim F_0(x_1-a_1,x_2)$ and $\sim F_0(x_1+a_1,x_2)$ where $a_1>0$ is a minimal distance between $\pm F_0$'s allowing such a matching.
The problem in $\re^2_+$ with the conditions \ef{F1Sym} is also variational, admits the corresponding L--S family of patterns, so that this created already a countable subset of other patterns. One can expect that this sequence of solutions of \ef{ell1}, \ef{F1Sym} of genus $\rho=0,1,2,...$ contains the one of genus $\rho=1$ for \ef{ell1} in $\re^2$. The functional  subset  ${H}_0$ in \ef{ell31} has an infinite category.

\item[{\bf (iii)}] The next pattern $F_2(x)$ most probably  is  radially symmetric and can be obtained by ODE analysis, see the next section.
Moreover, by the PFM and the L--S theory
\be
\label{Ex99}
\exists \,\, \mbox{a countable family of {\em even} radially symmetric patterns $\{F_{2l}(r)\}$.}
\ee
On the other hand similar to the 1D case there is a formal possibility to have a sequence of ``even" non-radial patterns $\{\hat F_{2l}(x)\}$ which are composed from $2l$  most densely packed, properly matched  patterns $\sim \pm F_0(x)$, which are circularly symmetric concentrated around the origin $0$. Such patterns are usually not of  the L--S type.

\item[{\bf (iv)}] The origin of $F_3(x)$ and other odd patterns $F_{2l+1}(x)$, $l=1,2,...$
can be multi-fold and complicated. For instance, $F_3(x)$ might consist of three first patters $\sim$ $+F_0, \, -F_0,\, +F_0$
densely packed along the $x_1$-axis, i.e., as happens in 1D but we are not sure that such a pattern $F_3$ (if any)
is an L--S one. The ``line" of the centers of $l$ shifted and packed  $\sim \pm F_0$ along it to create a pattern $F_l$ for $l \ge 3$ cannot be arbitrary and must follow a periodic structure of exponential tails described by the separator of variables via \ef{ell2}, \ef{Belt1}. 

As another possibility, these 
 might be obtained by minimizing the functional in sectors 
${\mathcal A}_l=\{(r,\varphi):\, r  >0,\,\, 0<\varphi<\varphi_l= \frac {\pi}{l+1}\}$ with the same normal anti-symmetry conditions on the boundary rays with ${\bf n}$ being its unit outward normal vector: 
 \be 
 \label{F771}
 F=F_{\bf nn}=0 \quad (\mbox{Dirichlet-Navier regular b.c.'s, the origin  O is regular}),
\ee 
 allowing the invariant odd reflection at the boundary. Again, each problem in ${\mathcal A}_l$ has a sequence of L--S patterns of categories $\rho=0,1,2,...$ for $H_0$ in the sector and 
 one with $\rho=2l+1$ for $H_0$  in $\re^2$ is expected to exist. 
 Hence, further solutions of \ef{ell1}, \ef{F771}
in those sectors exist, creating other countable patterns families, which eventually can 
lead  to multi-ray star-shaped periodic surfaces formally  corresponding to the infinite genus, {\em etc.}

 
\end{enumerate}

 Here we can observe an infinite subset  of countable L--S families of patterns, though the overall subset is expected to be much larger as it happened in 1D. Unfortunately, any reasonable mathematics concerning matching/gluing of various patterns in $\ren$ is currently absent. 
 
 \smallskip
 
 {\bf Remark: a quasilinear elliptic equation.} Similar properties of solutions can be true for other nonlinear elliptic equations. As a curious example we present an equation with a composed operator 
 with odd nonlinearities including a 4th-order $p$-Laplacian one:
  \be 
  \label{NewEq}
  \tex{
  \D ((\D F)^2 \D F)=-|F|F + F^5 \LongA \Phi(F)=  \int_{\ren}\big[\frac 14 (\D F)^4+\frac 13 |F|^3 \big] - \frac 16 \, \int_{\ren} F^6.
  }
  \ee
 Of course, \ef{NewEq} assumes a preliminary   delicate study of the asymptotic  behaviour of solutions near finite interface surfaces  $\partial \{F=0\}$ (and probably on $\partial \{F=1\}$), 
 but in general  the variational analysis does not require such a detailed information and  just necessary embeddings of the functional spaces in \ef{NewEq}. Note that the oscillatory behaviour of $F(x)$  close to interface surface 
 is asymptotically 1D in the normal direction to $\partial \{F=0\}$ and hence has an algebraic nonlinear structure similar to \ef{Nn2}. For the semilinear equation 
 $$
  \tex{
  \D ^2 F=-F^3 + F^5 \LongA \Phi(F)= \int_{\ren}\big[\frac 12 (\D F)^2+ \frac 14 F^4 \big] - \frac 16 \, \int_{\ren} F^6
  }
 $$
 the oscillatory behaviour at infinity is also algebraic and obeys \ef{mCU1}.
 
\begin{enumerate}
 
\item[{\bf (v)}] {\sc A ``chessboard" periodic solution.} We fix a square $S_{\rm R}$ in $\re^2$ with a side $R>0$ and consider the functional
 \ef{ell31} for functions in $H^2_0(S_{\rm R})$ with the same conditions \ef{F771}.  Since the eigenvalues
 $\{\l_k(R)\}$ of $\D^2$ in $S_R$ behave as
 $$
 \l_k(R)= \l_k(1) R^{-4} \to 0 \asA R \to \iy
 $$
 the genus $\rho_R= \sharp\{\l_k(R)<1\}$ of $H_0$ in $H^2_0(S_{\rm R})$ can be arbitrarily large for  $R \gg 1$, see \cite[\S~6.6]{Berger} and similar applications in  \cite{AEGnegI} and \cite[\S~1.3]{GMPBook}.
 Therefore, there exist at least $\rho_R$ different L--S critical points of the functional. Being reflected at the boundary in the odd manner again and again, these create $\rho_R$  periodic ``chessboard-type" solutions in $\re^2$.  
 
 \item[{\bf (vi)}] {\sc Triangular-shaped  periodic solutions.} Choose a large equilateral triangle $T_{\rm R}$, {\em etc.}
  
  {\bf ...}
  \end{enumerate}
  
  Finally, we note that according to our 1D experience
   those periodic solutions in $\re^2$ can create 
 various patterns by skipping all remote humps and rearranging the rest of them by a complicated matching procedure
 which used a periodic structure induced by exponential tails. This is a completely open problem.

\subsection{Back to the  quadratic nonlinearity: surprisingly,  much less is  known}

Consider the ODE problem \ef{N1.1} in the elliptic
setting
 \be
 \label{ell1Q}
 \D^2F=-F+F^2 \inB \ren, \quad F(x) \to 0 \asA x \to \iy.
  \ee
The first variational pattern $F_0$ is then obtained from the
corresponding functional by using the PFM and the L--S theorem (see Section \ref{S2.LS}):
 \be
 \label{ell31Q}
  \begin{matrix}
 \Phi(F)= \frac 12 \, \int [(\D F)^2 + F^2] - \frac 13 \,
 \int F^3 \inB H^2(\ren),
 \\
F=r_0(v)v\,\,\, \mbox{on} \,\,\,
   H_0= \big\{ \int [(\D F)^2 + F^2]=1 \big\}, \,\,\,\mbox{\em etc.}
    \end{matrix}
    \ee
 The functional is not even, the L--S approach does not apply, so that we can much less say rigorously about a general structure of the patterns and periodic subsets (here again $N=2$):


\begin{enumerate}

\item[{\bf (i)}] The first patterns $F_0(x)$ delivers  the absolute  minimum value $C_0>0$ of the functional \ef{ell31Q}, is  radially symmetric and satisfies an ODE, see below.

\item[{\bf (ii)}] Similar to the above cubic problem, we suggest that  the second pattern $F_1(x)$ 
 is not radial and is expected to be composed from  two neighbouring  radial first patterns  $F_0(r)$, $r= \sqrt{(x_1\pm a_1)^2+x_2^2}$, where $a_1>0$  is the first value for which such a gluing exists. This   
 gluing occurs at the symmetry line $\{x_1=0\}$ with the $x_1$-even symmetry conditions on the normal derivatives
 \be 
 \label{F1SymQ}
F_{x_1}=F_{x_1 x_1 x_1}=0 \atA x_1=0.
\ee
Then it is the absolute minimum point of \ef{ell31Q} in $\re^2_+=\{x_1>0,x_2 \in \re\}$ on functions from $H_0$ satisfying \ef{F1SymQ}. Thus,  
$F_1(x)$ can be composed  by matching at the $x_2$-axis  $\{x_1=0\}$ from two $F_0(x \pm a_1,x_2)$
with a minimally possible shift parameter $a_1>0$,
Due to \ef{F1SymQ}, the reflection $x_1 \mapsto - x_1$ gives a pattern $F_1(x_1,x_2)$ in $\re^2$. 
We expect that such a critical point is not unique and 
 there exists a countable sequence $\{a_k \sim \pi \sqrt 2 k, k \gg 1\}$ for which such a matching is available so that
 a countable family of solution of the problem \ef{ell1}, \ef{F1Sym} in $\re_+$ exists.
 For instance, after the reflection, we expect patterns of an arbitrary number of double humps concentrated along $x_2$-axis for $x_1>0$ and $x_1<0$ and expanding as $x_2 \to \pm \iy$. Those  represent  finite pieces of a 2-rays periodic orbit quite similar to $\{F_k\}$ as ``pieces" of $\Gamma_{\rm max}$ in 1D. Other types of patterns associated with different periodic surfaces can occur in such a construction with  increasing  complexities.

\item[{\bf (iii)}] Similarly, basic patterns $F_{2l-1},\, l=2,3,...$, are obtained by minimizing the functional in sectors 
${\mathcal A}_l=\{(r,\varphi):\, r  >0,\,\, 0<\varphi<\varphi_l= \frac {2\pi}{l+1}\}$ with the same normal symmetry conditions on the boundary rays with $n$ being its unit normal vector: 
 \be 
 \label{F771Q}
 F_{\bf n}=F_{\bf  nnn}=0 \quad (\mbox{Dirichlet-Neumann regular b.c.'s, the origin  O is regular}),
\ee 
 allowing the invariant rotation in $\re^2$ by the angle $\varphi_l$. Then  such $F_l$ 
 clearly are the absolute minimum point of \ef{ell31Q} in ${\mathcal A}_l$ and corresponds to the radial $F_0(|x|$ restricted to ${\mathcal A}_l$, but we expect that further critical points could give more nontrivial solutions.
  Note that by the normalization in \ef{ell31Q} the steady state $F \equiv 1$ is not acceptable. 
  Further solutions of \ef{ell1Q}, \ef{F771Q}
in those sectors can create other countable patterns families and since the genus $\rho({\mathcal A}_L)=\iy$
this construction can lead to 
  $l$-ray star-shaped periodic surfaces.

\item[{\bf (iv)}] ``Even" patterns $F_{+2l}$ (from ${\mathcal F}_{2R}$) can  be radial.
This means that $\Gamma_{\rm minR}$ is  not a periodic orbit (in $r$) but approaches  the 1D periodic one
$\Gamma_{\rm min}$ as $r \to \iy$.

\item[{\bf (v)}] We expect other radial  patterns $\sim  F_{2l}$ for $l \ge 1$ which can be constructed via the ODE.
Those patterns as $l \to \infty$ will form a kind of $\Gamma_{\rm maxR}$, not a periodic orbit but similarly approaching a 1D one $\Gamma_{\rm max}$ as $r \to \iy$.
The above shows that a blow-up unstable  attractor $\sim W^{2,\infty}_{\rm Rad}$ for the radial problem can be defined   in the direction 
of increasing of the time $r>0$ with a similar two embedded wings $\Gamma_{\rm maxR/minR}$ of the asymptotic  periodic  geometric structure to be seen
for $r \gg 1$. In other words, in the radial setting, two  families ${\mathcal F}_{\rm 1R,2R}$ composed from finite
pieces of
 two ``almost" periodic orbits $\Gamma_{\rm maxR/minR}$ exist as in 1D but possibly those are not anymore
 the basic (main) families which now belong to essentially non-radial patterns governed by elliptic problems. 
Such a classification requires an {\em a priory} knowledge of their critical values and/or  a subtle 
application of the Mountain Pass Lemma (if applied) or other profound variational techniques, which inevitable would demand a clearer knowledge of
possible  geometric shapes of patterns under scrutiny.

\item[{\bf (vi)}] Concerning many other $l$-rays, ``chessboard", triangular-shaped periodic,  chaotic (non-homoclinic) surfaces, and the corresponding attractor $W^{\{?\}}$
of the nonlinear elliptic equation \ef{ell1}... 

\end{enumerate}

We stop at this moments and will not try even to attempt  to describe this incredible
and sometimes imaginable  amount of homoclinic, star-, not star-shaped, chessboard-shaped ($\sim F_0$'s packed in a chessboard order on the whole plane)  periodic and chaotic surfaces in $\re^2$ which can be generated by the elliptic equation \ef{ell1Q}.
Some given conclusions and many possible others  are   suggestions only, and, in particular and again, the actual distribution of such patterns among classes ${\mathcal F}_{1,2}$ and others should be checked by using numerical 
estimates of their critical values of the functional. At least, as in 1D case, this would allow us to find those real families ${\mathcal F}_{1,2}$ having larger critical values among others of a similar geometric shape and/or the same number of positive dominant humps.
Further patterns are supposed to be  constructed by matching/gluing of those mentioned above. 
We now show how it works in the radial setting where we can fully use a  1D experience achieved earlier.

\smallskip


{\bf Remark.} Those tricks in the construction of various patterns in $\re^2$ are easier to explain 
for the corresponding P-L approximation of \ef{ell1}:
  \be 
   \label{R98}
\tex{
 \Delta^2 F= | F - \frac 12|- \frac 12 \quad \mbox{in} \quad \re^2.
 }
  \ee 
 Such a simplified P-L approximation is reduced to linear elliptic equations in the corresponding domains which can be analyzed and even sometimes solved explicitly, \cite{AEGII}. This and related P-L problems are better suited to classify those homotopic, periodic, and  other surfaces.



 \section{Radial quadratic ODE in $\ren$: observing multi-layer patterns}
 \label{SectRad4}

We are back to our original quadratic model \eqref{ell1Q}. 
Radially symmetric solutions $F=F(r)$, $r=|x|>0$ of the elliptic problem \ef{ell1Q}
satisfy the ODE
\be
\label{ell51}
 \tex{
    \D^2_r F \equiv F^{(4)}+ \frac{2(N-1)}r \, F''' +
   \frac{(N-1)(N-3)}{r^2} \, F'' +
   \frac{(N-1)(3-N)}{r^3} \, F' =-F +F^2, \,\,\, r>0,
   }
   \ee
and $F(r) \to 0$ 
as $r \to
\iy$. In view of the radial symmetry, there holds
 \be
 \label{ell21}
 F'(0)=F'''(0)=0.
  \ee
Since on  smooth functions $F(r)$ with uniformly bounded derivatives
 \be
 \label{ell4}
 \tex{
  \D^2_r F = F^{(4)} + \frac {2(N-1)}r \, F''' + O \big( \frac 1 {r^2} \big) \asA r \to \iy,
  }
  \ee
  the 2D asymptotic exponentially decaying  manifold is similar
to that  for $N=1$ (i.e., $F(r) \sim {\rm e}^{a r}$ implies $a^4=-1$), but has  an extra algebraic
factor due to the term $\sim \frac 1 r$ in (\ref{ell4}):
 \be
 \label{ell15}
 \tex{
F(r)= r^{- \frac {N-1}{ 2}} \,\eee^{-\mu r}\big[C_1 \cos(\mu
r)+C_2 \sin(\mu r)\big]+... \asA r \to \iy, \quad C_1, C_2 \in
\re, \,\, \mu = \frac 1{\sqrt 2}.
 }
 \ee



\subsection{Numerical evidence}

For solving \ef{ell51}, a  regularization of the operator at $r=0$
by replacing
 $$
 \tex{
 \frac 1r \mapsto \frac 1{\sqrt{\e^2+r^2}} \quad \mbox{with, typically,}
  \quad \e=10^{-2} \,\,\,\mbox{up to} \,\,\, \e=10^{-4},
  }
  $$
  is used that is
   enough not to perturb   the required results.

 In Figure \ref{FRad1}, we present the first radial pattern $F_0(r)$
 for dimensions $N=2,3,4$ together with the already studied for
 $N=1$ for the sake of comparison.

\begin{figure}[htbp]
\begin{center}
\includegraphics[scale=0.62]{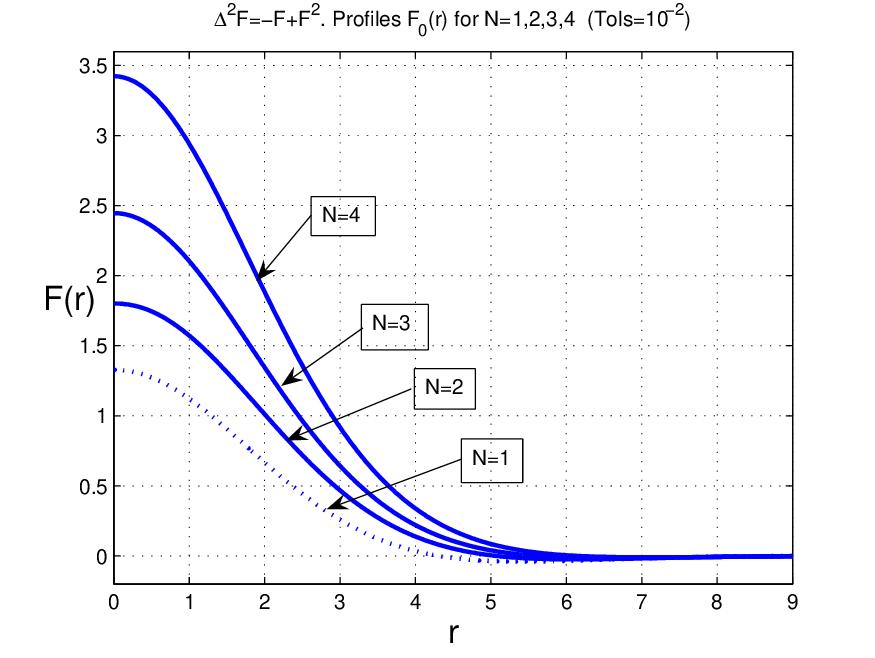}
\caption{The basic pattern $F_0(r)$ for $N=1,2,3,4$.}
\label{FRad1}
\end{center}
\end{figure}

\begin{figure}[htbp]
 \hfill   \hfill  \begin{minipage}[t]{0.40\textwidth} 
    \centering 
    \includegraphics[width=7cm,keepaspectratio]{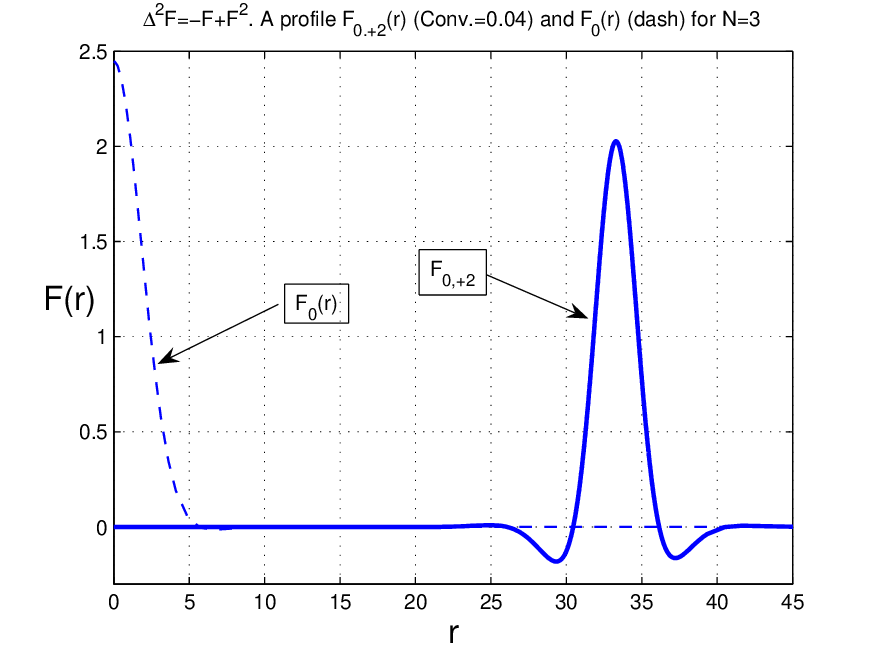}
    \caption{The  pattern $F_{0,+2}$ for $N=3$.}
\label{FRad2}
\end{minipage}
  \hfill   \hfill 
\begin{minipage}[t]{0.5\textwidth}
    \centering 
    \includegraphics[width=7cm,keepaspectratio]{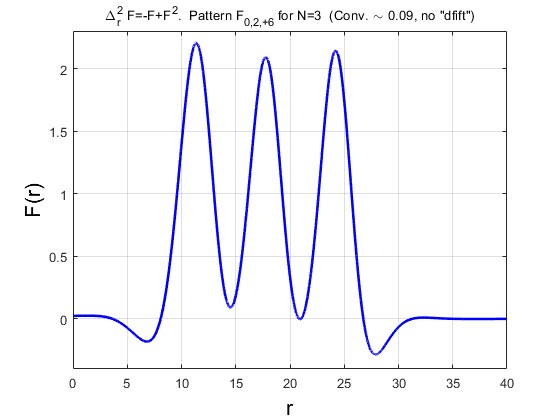}
\caption{A triple layer pattern $F_{0,2,+6}$  for $N=3$.}
\label{Fcy41}
\end{minipage}
\end{figure}

Consider a phenomenon, which was not available for $N=1$ where the ODE is autonomous in $x$.
Figure \ref{FRad2} shows a typical pattern, which in 1D just means
$F_0(x)$ shifted to the right. In the radial geometry in $\re^2$,
this is a pattern denoted by
 $$
 F_{0,+2}(r),
 $$
 where ``$0$" stands for an unknown finite number of zeros of a  small
 tail in a ``zero-hole" around $x=0$. Note again that such a simplified index
 cannot uniquely describe the pattern. Therefore,  on the plane
 $\re^2$, this solution represents a thin concentric layer with
 exponentially small tails around.

Such a double concentric  cylindrical layer for $N=2$ in shown in
Figure \ref{Fcyl1}. A similar ring is shown in Figure \ref{Fcy31}. Note
that in all such cases the exponential tails in a long enough
$0$-zone are practically invisible, sometimes numerically, so we
cannot count the total number of zeros therein. But as we know
from the 1D analysis, for $r \gg 1$, the semi-period of such
sin-oscillations is
 $$
 \tex{
 T_*= \pi {\sqrt 2}=4.4429...\, ,
}
 $$
and each semi-period contains a single zero (or a minimum point).
Hence, the formula for the number of zeros in a hole or a ring
$R_l$ of a given length $l \gg
 T_*$
 $$
 \tex{
 H_{\mbox{zeros in $R_l$}} \sim
 \frac {l}{T_*}
 }
  $$
   can be used but
not too close to the origin and to existing non-zero local
structures, where the linearized analysis does not apply and a
nonlinear interaction of patterns and their not that small tails
is in charge.

\begin{figure}[htbp]
\begin{center}
\includegraphics[scale=0.52]{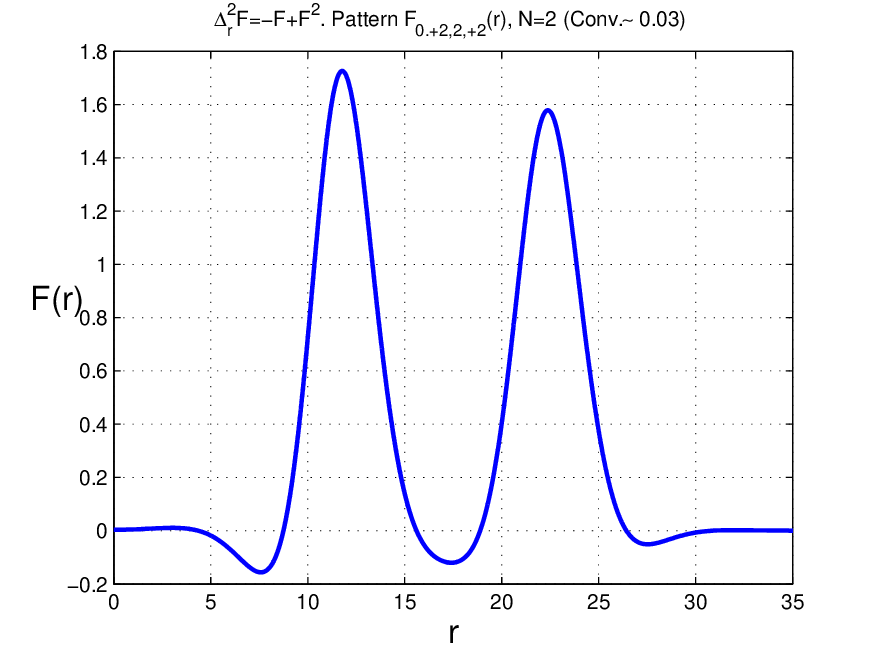}
\caption{Double layer pattern $F_{0,+2,2,+2}$ for $N=2$.}
\label{Fcyl1}
\end{center}
\end{figure}

\begin{figure}[htbp]
\begin{center}
\includegraphics[scale=0.52]{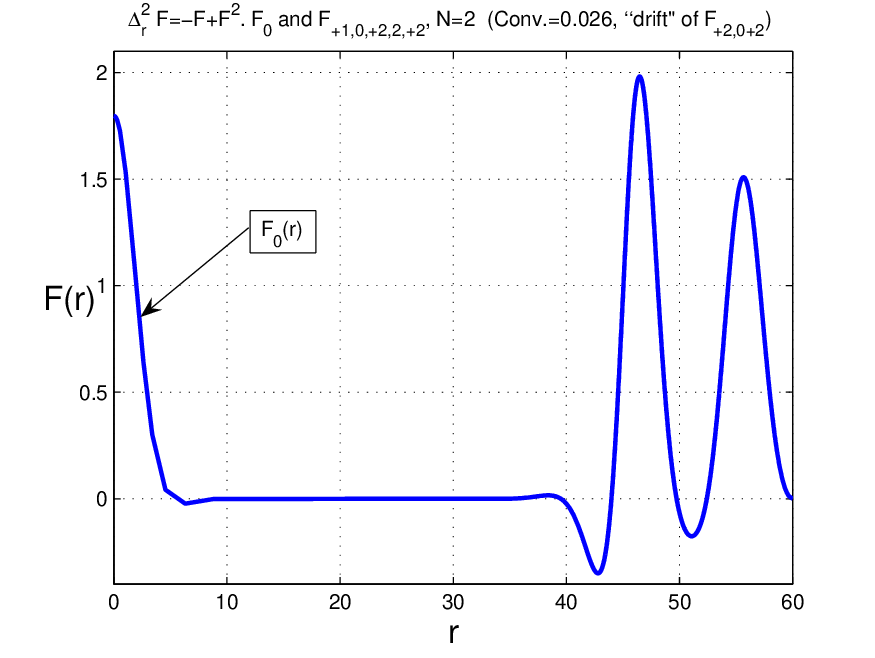}
\caption{A ring composed from $F_0(r)$ in the center and a double
layer around  creating overall a  pattern $F_{+1,2,+2,2,+2}$ for
$N=2$.}
\label{Fcy31}
\end{center}
\end{figure}

 The same phenomenon is shown in Figure \ref{FRad2} for $N=3$, which
 represents a solution of \ef{ell1} as  a concentric spherical layer in $\re^3$ around $x=0$
with a hole and out-space filled with exponentially small tails.
Such a triple spherical layer in $\re^3$ is presented in Figure
\ref{Fcy41} (in 1D it would look like a version of $F_{+6}$ with a
small perturbation near $r=20$ where it touches the zero level
$F=0$; the convergence is not perfect and is difficult to
achieve). For convenience,  Figure \ref{Fcy51} shows a triple layer
simultaneously  for $N=2$, 3, and, for comparison, also for $N=1$,
when it is just a pattern $F_{+6}$. A spherical seven-layer
pattern in $\re^3$,
 $$
 F_\sigma(r), \quad \sigma=\{{\mathbf 2},+2,2,+2,2,+2,2,+2,2,+2,2,+2,2,+2,2\},
  $$
with precisely ${\mathbf 2}$ zeros in a central hole, is shown in
Figure \ref{F77} together with looking  similar $F_{+14}$ for $N=1$.



\begin{figure}[htbp]
 \hfill   \hfill  \begin{minipage}[t]{0.40\textwidth} 
    \centering 
    \includegraphics[width=7cm,keepaspectratio]{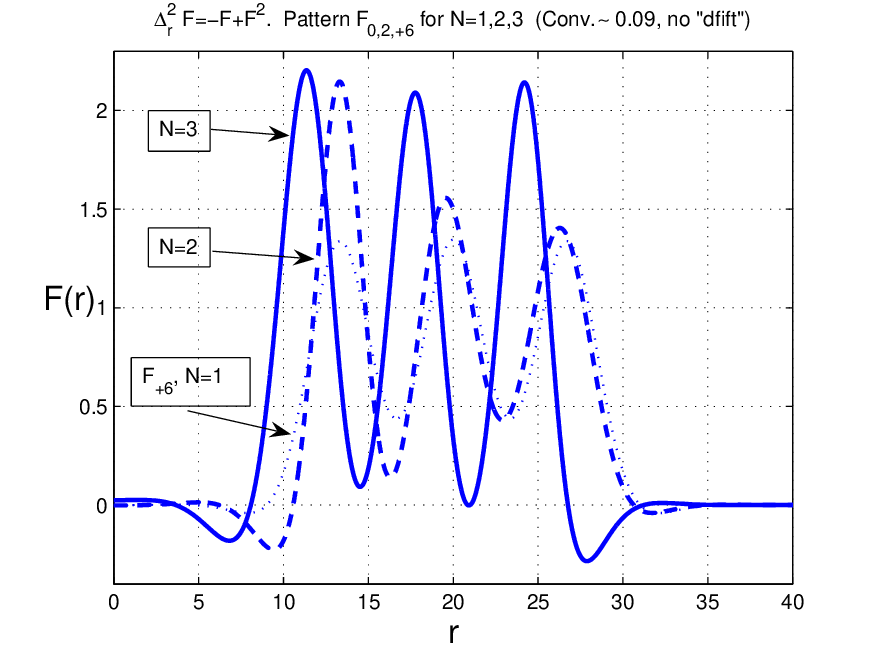}
    \caption{Triple layers for $N=2$, $N=3$, and the corresponding
$F_{+6}$ for  $N=1$ (a dotted line).}
\label{Fcy51}
\end{minipage}
  \hfill   \hfill 
\begin{minipage}[t]{0.5\textwidth}
    \centering 
    \includegraphics[width=7cm,keepaspectratio]{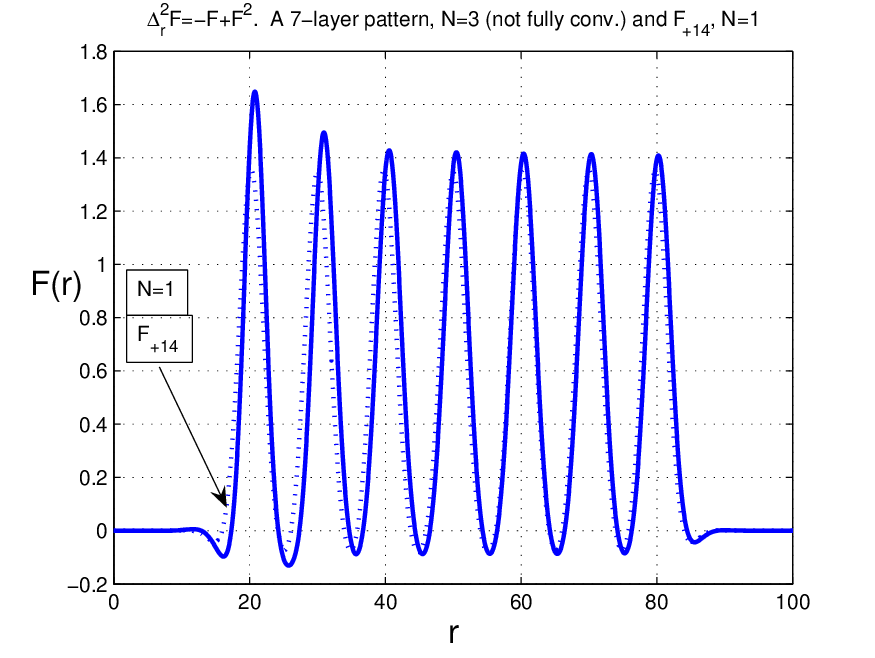}
\caption{A seven layers for $N=3$ (iterations are not fully
convergent) and $F_{14}$ for $N=1$ (a dotted line, an easy fast
convergence).}
\label{F77}
\end{minipage}
\end{figure}



A slow ``numerical drift" of a single pattern (like $F_0$, but not
exactly) for $N=4$ in the direction of increasing $r$ is shown in
Figure \ref{FRad31}, see a comment below.


Finally, in Figure \ref{FRad33}, $N=3$, we present a complicated
pattern
 $$
 \sim F_\sigma(r), \quad \sigma=\{+1,4,+2,5,+2,2,+4,3,+2\}
 $$
  composed from five different lower-order structures
described above separately. The graph is not fully convergent, but
there is no a ``numerical drift" to the right at the last stage of
numerics, so we believe that this structure really exists and will
be fully convergent finally.




\begin{figure}[htbp]
 \hfill   \hfill  \begin{minipage}[t]{0.40\textwidth} 
    \centering 
    \includegraphics[width=7cm,keepaspectratio]{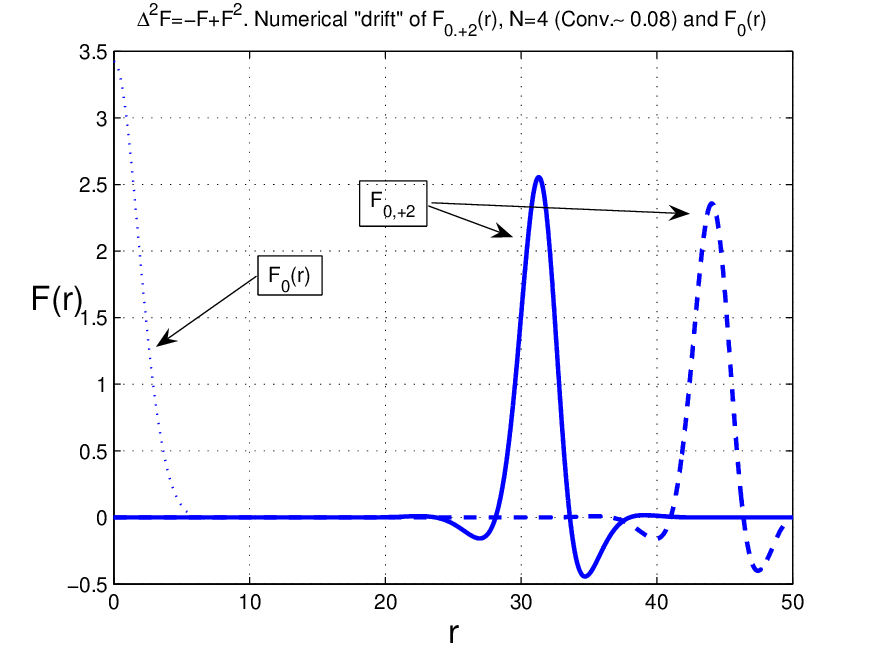}
   \caption{Slow ``drift" of an isolated  $F_0$-like pattern for
$N=4$.}
\label{FRad31}
\end{minipage}
  \hfill   \hfill 
\begin{minipage}[t]{0.4\textwidth}
    \centering 
    \includegraphics[width=7cm,keepaspectratio]{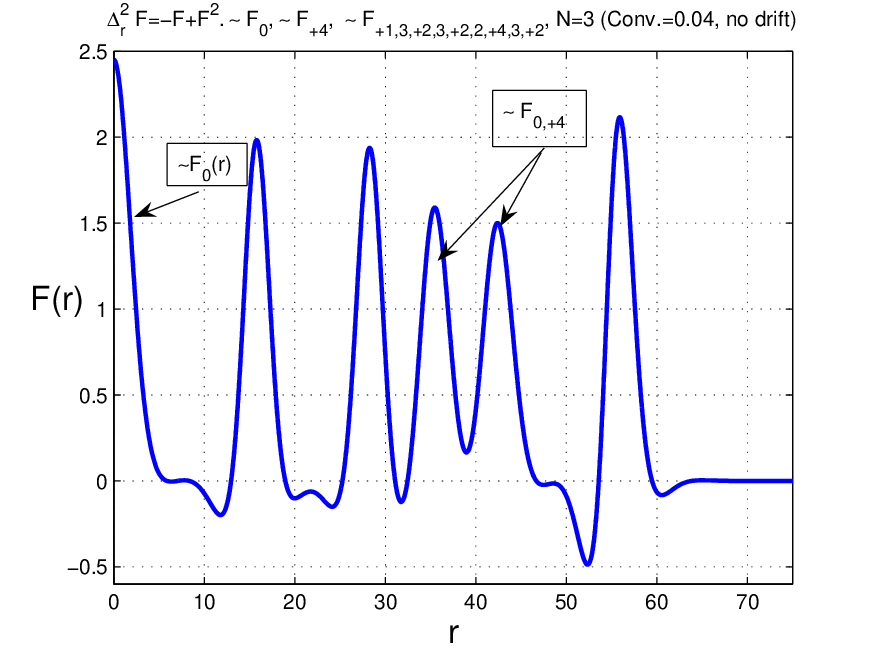}
\caption{A complicated pattern composed from five types of
different ones  for $N=3$.}
\label{FRad33}
\end{minipage}
\end{figure}

Note that numerical modelling of radial solutions $F(r)$ of
\ef{ell51} is more delicate than for \ef{N1.1} for $N=1$. And the
problem is not that the first one \ef{ell51} is more complicated,
It seems that an unavoidable and an invincible  feature occurs:
the overall geometrical radial frame is attached to the single
origin $x=0$ and they move together, unlike the case $N=1$, when,
by invariant translations, $x=0$ can be attributed to any point
not affecting all possible solutions. Therefore, in many numerical
tests, we observe a definite $r$-drift of patterns (already
achieved a proper correct shape) to larger $r$. It looks like
those patterns eventually want to get to $r=\iy$ and a full
``1D-freedom" for the rest of their life. Therefore in a couple of
figures above we did not achieve a full ``stationary" convergence
in our numerics and clearly indicated that.


\smallskip

{\bf Acknowledgements.}  The authors would like to thank Professor M. Grinfeld
 for interesting  discussions concerning  dynamical systems theory and applications.
 


\end{document}